\documentclass[11pt,reqno]{amsart}
\usepackage{xifthen,setspace,bbm}

\usepackage{bm,natbib,multirow,float}
\usepackage{esvect,hyperref}
\usepackage{amsmath,amssymb,mathrsfs,xcolor} 
\usepackage{lipsum}  
\usepackage{subfig}
\usepackage{graphicx}

\newtheorem{thm}{Theorem}[section]
\newtheorem{rem}[thm]{Remark}
\newtheorem{example}[thm]{Example} 
\newtheorem{lemma}[thm]{Lemma} 
\newtheorem{corollary}[thm]{Corollary} 
\newtheorem{proposition}[thm]{Proposition}
\theoremstyle{definition} 
\newtheorem{definition}[thm]{Definition} 
\newtheorem{assumption}[thm]{Assumption} 
\newtheorem{theorem}[thm]{Theorem}

\newcommand{\te}{{\tilde{O}}}

\newcommand{\cE}{\mathcal{E}}
\renewcommand{\leq}{\leqslant}
\renewcommand{\geq}{\geqslant}
\renewcommand{\ge}{\geqslant}
\renewcommand{\le}{\leqslant}

\newcommand{\norm}[1]{\left\lVert#1\right\rVert}

\newcommand{\bh}{{\boldsymbol{\theta}}}
\newcommand{\bg}{{\boldsymbol{\gamma}}}

\newcommand{\e}{{\mathbb{E}}}

\newcommand{\bs}{{\boldsymbol{X}}}

\newcommand{\pl}{{\mathrm{poly}}}
\newcommand{\pll}{{\mathrm{polylog}}}

\newcommand{\cC}{\mathcal{C}}

\newcommand{\R}{{\mathbb{R}}}

\newcommand{\cR}{{\mathcal{R}}}
\newcommand{\mf}{{\bm F}}
\newcommand{\bp}{\mathbb{P}}
\newcommand{\p}{{\mathbb{P}}} 
\newcommand{\cA}{\mathcal{A}}

\def\X{\bm{X}}
\def\Z{\bm{Z}}

\def\A{\bm{A}}
\def\R{\mathbb{R}}
\def\tht{\bm{\theta}}

\usepackage{algorithm}
\usepackage{algpseudocode} 

\allowdisplaybreaks

\title[Logistic Regression Under Network Dependence]{High Dimensional Logistic Regression Under Network Dependence}

\author[Mukherjee]{Somabha Mukherjee} 
\address{Department of Statistics and Data Science, National University of Singapore, Singapore}
\email{somabha@nus.edu.sg}

\author[Niu]{Ziang Niu}
\address{Applied Mathematics and Computational Science, University of Pennsylvania, Philadelphia, USA}
\email{ziangniu@sas.upenn.edu}

\author[Halder]{Sagnik Halder}
\address{Department of Statistics, University of Florida, Gainesville, USA}
\email{shalder@ufl.edu}

\author[Bhattacharya]{Bhaswar B. Bhattacharya}
\address{Department of Statistics and Data Science, University of Pennsylvania, Philadelphia, USA}
\email{bhaswar@wharton.upenn.edu}

\author[Michailidis]{George Michailidis}
\address{Department of Statistics and Informatics Institute, University of  Florida, Gainesville, USA}
\email{gmichail@ufl.edu} 

\begin{document}

	\begin{abstract} 
		Logistic regression is key method for modeling the probability of a binary outcome based on a collection of covariates. However, the classical formulation of  logistic regression relies on the independent sampling assumption, which is often violated when the outcomes interact through an underlying network structure, such as over a temporal/spatial domain or on a social network. This necessitates the development of models that can simultaneously handle both the network `peer-effect' (arising from neighborhood interactions) and the effect of (possibly) high-dimensional covariates. In this paper, we develop a framework for incorporating such dependencies in a high-dimensional logistic regression model by introducing a quadratic interaction term, as in the Ising model, designed to capture the pairwise interactions from the underlying network. The resulting model can also be viewed as an Ising model, where the node-dependent external fields linearly encode the high-dimensional covariates. We propose a penalized maximum pseudo-likelihood method for estimating the network peer-effect and the effect of the covariates (the regression coefficients), which, in addition to handling the high-dimensionality of the parameters, conveniently avoids the computational intractability of the maximum likelihood approach. Under various standard regularity conditions, we show that the corresponding estimate attains the classical high-dimensional rate of consistency. In particular, our results imply that even under network dependence it is possible to consistently estimate the model parameters at the same rate as in classical (independent) logistic regression, when the true parameter is sparse and the underlying network is not too dense. Consequently,  we  derive  the rates of consistency of our proposed estimator for various  natural  graph  ensembles,  such  as bounded degree graphs, sparse Erd\H os-R\'enyi random graphs, and stochastic block models. We also develop an efficient algorithm for computing the estimates and validate our theoretical results in numerical experiments. An application to selecting genes in clustering spatial transcriptomics data is also discussed. 
		
		\vspace{0.1in}
	\end{abstract}
	

	\maketitle

	\section{Introduction}

	Logistic regression \citep{appliedlogistic,glmmn,glmnw} is a very popular and widely used method for modeling the probability of a binary response based on multiple features/predictor variables. It is a mainstay of modern multivariate statistics that has found widespread applications in various fields, including machine learning, biological and medical sciences, economics, marketing and finance industries, and social sciences. For example, in machine learning it is regularly used for image classification and in the medical sciences it is  often used to predict the risk of developing a particular disease based on the patients' observed characteristics, among others. To describe the model formally, denote the vector of predictor variables (covariates) by $\bm Z_1,\ldots, \bm Z_N \in \mathbb{R}^d$ and the independent response variables by $X_1,\ldots,X_N \in \{-1,1\}$. Then, the logistic regression model for the probability of a positive outcome conditional on the covariates is given by 
	\begin{equation*}
		\p(X_i = 1|\bm Z_i) = \frac{e^{\bh^\top \bm Z_i}}{e^{\bh^\top \bm Z_i} + e^{-\bh^\top \bm Z_i}} , 
	\end{equation*} 
	for $1\le i\le N$, where $\bh = (\theta_1, \theta_2, \ldots, \theta_d)^\top \in \mathbb{R}^d$ is the vector of regression coefficients.\footnote{Note that we are parameterizing the outcomes as $\{-1, 1\}$ instead of the more standard $\{0, 1\}$ to integrate this within the framework of the Ising model (defined in \eqref{pmf}), where the $\{-1, 1\}$ notation is more common.} Using the independence of the response variables, the joint distribution of $\bm X := (X_1,\ldots,X_N)$ given $\bm Z := (\bm Z_1,\ldots, \bm Z_N) \in \mathbb R^{d \times N}$ can be written as: 
	\begin{equation}\label{jointlogistic}
		\p(\bm X |\bm Z) = \prod_{i=1}^N \frac{e^{ X_i \bh^\top \bm Z_i}}{e^{\bh^\top \bm Z_i} + e^{-\bh^\top \bm Z_i} }  =  \frac{1}{\mathcal Z_N(\bm \theta, \bm Z)}  \exp\left\{ \sum_{i=1}^N X_i (\bh^\top \bm Z_i) \right\}, 
	\end{equation}
	where $\mathcal Z_N(\bm \theta, \bm Z) = \prod_{i=1}^N \frac{1}{e^{\bh^\top \bm Z_i} + e^{-\bh^\top \bm Z_i} }$ is the normalizing constant. It is well-known from the classical theory of generalized linear models that the parameter $\bm \theta$ in \eqref{jointlogistic} can be estimated at rate $O(1/\sqrt N)$ for fixed dimensions, using the maximum likelihood (ML) method (see, for example, \citet{tsh,glmmn,vandervaart}).

	The classical framework of logistic regression described above is, however, inadequate if the independence assumption on the response variables is violated, which is often the case if the observations are collected over a temporal or spatial domain or on a social network. 
	The recent accumulation of dependent network data in modern applications has accentuated the need to develop realistic and mathematically tractable methods for modeling high-dimensional distributions with underlying network dependencies ({\it network peer-effect}). Towards this, the Ising model, initially proposed in statistical physics to model ferromagnetism \citep{ising}, has turned out to be a useful primitive for modeling such datasets, which arise naturally in spatial statistics, social networks, epidemic modeling, computer vision, neural networks, and computational biology, among others (see~\citet{spatial,geman_graffinge,disease,hopfield,innovations} and references therein). This can be viewed as a discrete exponential family with binary outcomes, wherein the sufficient statistic is a quadratic form designed to capture correlations arising from pairwise interactions. Formally, given an interaction matrix $A := ((a_{ij}))_{1 \leq i, j \leq N}$ and binary vector $\bm X=(X_1, X_2, \cdots, X_N) \in \cC_N=\{-1, 1\}^N$, the Ising model with parameters $\beta$ and $h$ encodes the dependence among the coordinates of $\bm X$ as follows: 
	\begin{equation}
		\mathbb P_{\beta,h}(\bm X)= \frac{1}{2^N \mathcal Z_N(\beta,h)} \exp\left\{\beta \sum_{1 \leq i, j \leq N} a_{ij} X_i X_j + h \sum_{i=1}^N X_i \right\},
		\label{pmf}
	\end{equation}
	where the {\it normalizing constant} $\mathcal Z_N(\beta,h)$ is determined by the condition $\sum_{\bm X\in \cC_N}\mathbb P_{\beta,h}(\bm X)=1$ (so that $\mathbb P_{\beta,h}$ is a probability measure). In statistical physics the parameters $\beta$ and $h$ are referred to as the \textit{inverse temperature} and the \textit{magnetic field}, respectively. 

	This paper is motivated by applications where in addition to peer effects, arising from an underlying network structure, there are other variables (covariates) associated with the nodes of the network, which affect the outcome of the response variables. For example, in the data collected by the National Longitudinal Study \citep{network_data} students in grades 7--12 were asked to name up to 10 friends and answer many questions about age, gender, race, socio-economic background, personal and school life, and health, where it becomes imperative to model the peer-effect and the effect of the covariates simultaneously. Another example where high-dimensional covariates interact through an underlying network structure arises in spatial transcriptomics. This is a relatively new direction in biology made possible by technologies for massively parallelized measurement of cell transcriptomes/proteomes in situ that, unlike standard single cell sequencing methods, retains information regarding the spatial neighborhood of the cells. The spatial context of cells can be encoded into a nearest-neighbor graph or a Voronoi neighborhood graph/Delaunay triangulation \citep{computationalgeometry}, where the nodes are cells and edges link cells that are situated proximal to each other  \citep{t1,t2,t3,t4}. Each node has a high dimensional feature set, encoding the measurements made for that cell, be it gene expression or protein expression, depending on the experimental protocol. Then, the goal is to understand how the spatial niche of a cell contributes to its phenotype (see Section \ref{sec:st} for more details). For other examples of network peer-effect in the presence of covariates, see \citet{dependent_I,dependent_II,dependent_III,dependent_IV,dependent_V,dependent_VI} and references therein. 
	
	In the aforementioned examples, it is natural to envisage a model that combines the logistic model in \eqref{jointlogistic} (which encodes the node-specific covariates) 
	and the Ising model \eqref{pmf} (for capturing the network dependency). Towards this, \citet{cd_ising_II} recently proposed the following model: Suppose for each node $1 \leq i \leq N$ of a network $G_N$ on $N$ vertices, one observes a $d$-dimensional covariate $\bm Z_i \in \R^d$. Then, the joint distribution of the binary outcomes $\bm X=(X_1, X_2, \ldots, X_N)$, conditioned on the network $G_N$ and the observed covariates $\bm Z:=(\bm Z_1, \bm Z_2, \ldots, \bm Z_N)^\top$ is given by:  
	\begin{equation}\label{eq:logisticmodel}
		\bp(\bs \big| \bm Z) \propto \exp\left( \beta \bm X^{\top} \bm A \bm X + \sum_{i=1}^N X_i (\bm \theta^{\top} \bm Z_i)  \right)
	\end{equation}
	where $\bm A  =((a_{ij}))_{1 \leq i, j \leq N}$ is the (appropriately scaled) adjacency matrix of the network $G_N$, the parameter $\beta $ is a measure of the strength of dependence (the network `peer effect'), and the parameter $\bm \theta \in \R^d$ measures the individual effects of the $d$-covariates. Specifically, as $\beta$ becomes more positive, the outcomes of the nodes tend to align with those of their neighbors. On the other hand, when $\beta$ is negative (which is also allowed in our theoretical framework), every node receives negative influences from its neighbors, a phenomenon referred to as {\it antiferromagnetism} in the statistical physics literature. Note that as in the classical Ising model \eqref{pmf}, the quadratic term captures the overall dependency in the network arising from pairwise interactions, while the linear term $\bm \theta^\top \bm Z_i$ encodes the strength of the covariates on the outcome at the $i$-th node, for $1 \leq i \leq N$, as in the logistic model \eqref{jointlogistic}. Moreover, when $\beta=0$, which corresponds to the independent case, \eqref{eq:logisticmodel} reduces to \eqref{jointlogistic}; hence,  \eqref{eq:logisticmodel}  can be viewed as a model for logistic regression with dependent observations.

	The increasing relevance of models \eqref{pmf} and \eqref{eq:logisticmodel} for understanding covariate effects and nearest-neighbor interactions in network data, has made it imperative to develop computationally tractable methods for estimating the model parameters and understanding the statistical properties (rates of convergence) of the resulting estimates.  A typical problem of interest is estimating the parameters of the model given a single sample of binary outcomes from an underlying network.  For the classical Ising model as in \eqref{pmf}, this problem has been extensively studied, beginning with the classical results on consistency and optimality of the maximum likelihood (ML) estimates for models where the underlying network is a spatial lattice \citep{comets_exp,gidas,guyon,mrf_pickard}. However, for general networks, parameter estimation using the ML method turns out to be notoriously hard due to the presence of an intractable normalizing constant in the likelihood. One approach to circumvent this issue that has turned out to be particularly useful, is the {\it maximum pseudolikelihood} (MPL) estimator \citep{besag_lattice,besag_nl}. This provides a computationally efficient method for estimating the parameters of a  Markov random field  that avoids computing the normalizing constant, by maximizing an approximation to the likelihood function (a `pseudo-likelihood'),  based on conditional distributions. This approach was originally explored in the seminal paper of \citet{chatterjee}, where general conditions for  $\sqrt N$-consistency of the MPL estimate for the model \eqref{pmf} were derived.\footnote{
		A sequence of estimators $\{\hat{\beta}_N\}_{N\geq 1}$ is said to be $\sqrt N$-{\it consistent} at $\beta$, if for every $\delta > 0$, there exists $M:=M(\delta, \beta) > 0$ such that $\p(\sqrt N |\hat \beta_N(\bm X) - \beta| \leq M ) > 1-\delta$, for all $N$ large enough.} This result was subsequently extended to more general models in \citet{BM16,cd_ising_II,cd_ising_I,pg_sm,rm_sm,ising_testing,sm_js_bbb}. In particular, for model \eqref{eq:logisticmodel} \citet{cd_ising_I} showed that given a single sample of observations $(X_i, \bm Z_i)_{1 \leq i \leq N}$ from \eqref{eq:logisticmodel}, the MPL estimate of the parameters $(\beta, \bm \theta)$ is $\sqrt N$-consistent, when the dimension $d$ is \textit{fixed}, under various regularity assumptions on the underlying network and the parameters. This result has been subsequently extended to models with higher-order interactions in \citet{cd_ising_I}. The high-dimensional analogue of this problem under an $\ell_1$ norm constraint on the regression parameter has been studied very recently in \citet{cddependent}.

	In this paper, we consider the problem of parameter estimation in model \eqref{eq:logisticmodel} in the \textit{high-dimensional regime} with sparsity constraints, that is, the number of covariates $d$ grows with the size of the network $N$, but there are only a few non-zero regression coefficients. In other words, we assume that the parameter vector $\bm \theta$ is $s$-sparse, that is $||\bm \theta||_0 : = \sum_{i=1}^d \bm 1\{\theta_i \ne 0\} \leq s$, because, despite the fact that many covariates are available, we only expect a few of them to be relevant. Under this assumption, we want to estimate the parameters $(\beta, \bm \theta)$ given a \textit{single} sample of dependent observations $(X_i, \bm Z_i)_{1 \leq i \leq N}$ from model \eqref{eq:logisticmodel}. One of the main reasons this problem is especially delicate, is that we only have access to a \textit{single (dependent) sample} from the underlying model. Consequently, classical results from the $M$-estimation framework, which require \textit{multiple independent samples}, are inapplicable. 
	To deal with this dependence (which leads to the intractable normalizing constant as mentioned before) and the high-dimensionality of the parameter space, we propose a penalized maximum pseudo-likelihood approach for estimating the parameters. To this end, note that the conditional distribution of $X_i$ given $(X_j)_{j \ne i}$ is:  
	\begin{align} \label{eq:probability_conditional}
		\bp(X_i=1 | (X_j)_{j \ne i}, \bm Z) & =  \frac{e^{X_i \bm \theta^{\top} \bm Z_i  + \beta X_i m_i(\bm X) }}{ e^{\bm \theta^{\top} \bm Z_i  + \beta  m_i(\bm X)} + e^{-\bm \theta^{\top} \bm Z_i  - \beta  m_i(\bm X)} }, 
	\end{align}
	where, as before, $m_i(\bm X) = \sum_{j=1}^N a_{ij} X_j$ is the {\it local-effect} at node $i$, for $1 \leq i \leq N$. Therefore, by multiplying \eqref{eq:probability_conditional} over $1 \leq i \leq N$ and taking logarithms, we get the (negative) {\it log-pseudolikelihood} (LPL) function 
	\begin{align}\label{eq:LNbetatheta}
		L_N(\beta, \bm \theta) 
		&= - \frac{1}{N} \sum_{i=1}^{N}\log \bp(X_i=1\big|(X_j)_{j \ne i}, \bm Z) \nonumber \\ 
		& =- \frac{1}{N}\sum_{i=1}^N \{ X_i ( \bm \theta^{\top} \bm Z_i  + \beta  m_i(\bm X) ) - \log \cosh ( \bm \theta^{\top} \bm Z_i  + \beta m_i(\bm X) )\} + \log 2.
	\end{align} 
	To encode the sparsity of the high-dimensional parameters, we propose a {\it penalized maximum pseudo-likelihood} (PMPL) estimator of $(\beta, \bm \theta^\top)$, which, given a regularization (tuning) parameter $\lambda$, is defined as: 
	\begin{align}\label{eq:betatheta}
		(\hat \beta, \hat {\bm \theta}^\top): =  \mathrm{arg}\min_{(\beta, \bm \theta)} \{  L_N(\beta, \bm \theta )   + \lambda || \bm \theta ||_1\},
	\end{align} 
	where $||\bm \theta||_1 := \sum_{i=1}^d |\theta_i|$. Under various 
	regularity assumptions, we prove that if $\lambda$ is chosen proportional to $\sqrt{\log d/n}$, then with high probability,
	\begin{align}\label{eq:bt}
	\| (\hat \beta, \hat {\bm \theta}^\top)  - (\beta, \bm \theta^\top) \| \lesssim_s \sqrt{\log d/N},
	\end{align} 
	whenever $d \rightarrow \infty$ such that $d=o(N)$ (Theorem \ref{thm:estimate}). In particular, it follows from our results that for bounded degree networks, the PMPL estimate attains the same rate as in the independent logistic case \eqref{jointlogistic}, when $d=o(N)$ and the sparsity is bounded. We also have a more  general result that quantifies the dependence on the network sparsity in the rate \eqref{eq:bt}, which allows us to establish consistency of the PMPL estimate beyond bounded degree graphs (Proposition \ref{ppn:estimateF}). Our results are fundamentally different from existing results on parameter estimation in high-dimensional graphical models based on multiple i.i.d. samples 
	(see Section \ref{sec:structurelearning} for a review). Here, we only have access to a single sample from the model \eqref{eq:logisticmodel}, hence, unlike in the multiple samples case, one cannot treat the different neighborhoods in the network as independent, which renders classical techniques for proving consistency inapplicable. Consequently, to handle the dependencies among the different neighborhoods in the pseudo-likelihood function we need to use a different (non-classical) set of tools. Specifically, our proofs combine a conditioning technique introduced in \citet{cd_ising_estimation}, which tranforms a general Ising model to a model satisfying the Dobrushin condition (where the dependence is sufficiently weak), and the concentration inequalities for functions of dependent random variables in the Dobrushin regime, based on the method of exchangeable pairs \citep{scthesis}. 
	
	Next, we study the effect of dependence on estimating the regression parameters $\bm \theta$. Specifically, we want to understand how the presence of dependence through the underlying network structure effects the rate of estimation of the  high-dimensional regression coefficient under sparsity constraints. While there have been several recent attempts to understand the implications of dependence in high-dimensional inference, most of them have focused on Gaussian models. Going beyond Gaussian models, \citet*{rm_sm} and \citet{isingsignal} considered the problem of testing the global null hypothesis (that is, $\bh = \bm 0$) against sparse alternatives in a special case of model \eqref{eq:logisticmodel} (where $d=N$ and the design matrix $\bm Z= \bm I_N$ is the identity). However, to the best of our knowledge, the effect of dependence on parameter estimation in Ising models with covariates has not been explored before. In the sequel,  we establish that the PMPL estimate for $\bh$ in model \eqref{eq:logisticmodel}, given a dependence strength $\beta$ and the sparsity constraint $\norm{\bh}_0 = s$,  attains the classical $O({\sqrt{s\log d/N}})$ rate, \textit{despite the presence of dependence}, in the entire high-dimensional regime (where $d$ can be much larger than $N$) and also recovers the correct dependence on the sparsity $s$ (see Theorem \ref{thm:stheta}).  As a consequence, we establish that the PMPL estimate is $O(1/\sqrt N)$-consistent (up to logarithmic factors) for the model \eqref{eq:logisticmodel} in various natural sparse graph ensembles, such as  Erd\H os-R\'enyi random graphs and inhomogeneous random graphs that include the popular stochastic block model (Theorem \ref{thm:sparse} and Corollary \ref{inh}). We also develop a proximal gradient algorithm for efficiently computing the PMPL estimate and evaluate its performance in numerical experiments for Erd\H os-R\'enyi random graphs, inhomogeneous random graph models, such as the stochastic block model and the $\beta$-model, and the preferential attachment model (Section \ref{sec:computationandexperiments}). Finally, in Section \ref{sec:st}, we illustrate the effectiveness of the proposed model in selecting relevant genes for clustering spatial gene expression data.

	\subsection{Related Work}
	\label{sec:structurelearning}

	The asymptotic properties of penalized likelihood methods for logistic regression and generalized linear models in high dimensional settings have been extensively studied (see, for example, \citet{bach2010,florentina,kakade,meier2008group,salehi,vandegeer008,vandergreer14} and references therein). These results allow $d$ to be much bigger than $N$ and provide rates of convergence for the high-dimensional parameters under various sparsity constraints. The performance of the ML estimate in the logistic regression model when the dimension $d$ scales proportionally with $N$ has also been studied in a series of recent papers  \citep{sur1,sur2,sur3}. However, as discussed earlier, ML estimation is both computationally  and mathematically intractable in model \eqref{eq:logisticmodel}, because of the dependency induced by the underlying network. That we are able to recover rates similar to those in the classical high-dimensional logistic regression, in spite of this underlying dependency, using the PMPL method, is one of the highlights of our results.

	The problem of estimation and structure learning in graphical models and Markov random fields is another related area of active research. Here, the goal is to estimate the model parameters or learn the underlying graph structure given access to {\it multiple} i.i.d. samples from a graphical model. For more on these results refer to  \citet{structure_learning,bresler,bresler_karzand,discretetree,jointestimationGLMZ,graphical_models_algorithmic,klivansmeka,highdim_ising,graphicalmodelsbinary,vmlc16} and references therein. In particular,  \citet{highdim_ising} and \citet{ising_nonconcave}  use a regularized pseudo-likelihood approach to learn the structure of the interaction $\bm A$ given \textit{multiple} i.i.d. samples from an Ising model. In a related direction, \citet{cd_testing} studied the related problems of identity and independence testing, and \citet{high_tempferro,neykovliu_property} considered problems in graph property testing,  given access to multiple samples from an Ising model.

	All the aforementioned results, however, are in contrast with the present work, where the underlying graph structure is assumed to be {\it known} and the goal is to derive rates of estimation for the parameters given a {\it single} sample from the model. This is motivated by the applications mentioned above where it is often difficult, if not impossible, to generate many independent samples from the model within a reasonable amount of time. More closely related to the present work are results of \citet{bayesianvariableselection} and \citet{spatialbayesianvariableselection} on Bayesian methods for variable selection in high-dimensional covariate spaces with an underlying network structure, where an Ising prior is used on the model space for incorporating the structural information. Recently, \citet{isingbayesestimation} developed a variational Bayes procedure using the pseudo-likelihood for estimation based on a single sample in a two-parameter Ising model. Convergence of coordinate ascent algorithms for mean field variational
	inference in the Ising model has been recently analyzed in \citet{isingvariantionalinference}. 
	
	For continuous response with an underlying network/spatial dependence structure, a related model is the well-known spatial autoregressive (SAR) model and its variants, where likelihood estimation based on conditional distributions have been used as well  (see \cite{huang2019least,lee2004asymptotic,lee2010specification,zhu2020multivariate} and the references therein). 
	
	\subsection{Notation}
	
	The following notation will be used throughout the remainder of the paper. For a vector $\bm a := (a_1,\ldots,a_s) \in \mathbb{R}^s$ and $0 < p < \infty$, $\|\bm a\|_p := \left(\sum_{i=1}^s |a_i|^p\right)^\frac{1}{p}$ denotes its $p$-th norm and $\|\bm a\|_\infty := \max_{1\le i\le s} |a_s|$ its maximum norm, respectively. Moreover, 
	$\|\bm a\|_0 := \sum_{i=1}^s \bm{1}\{a_i \ne 0\}$ denote the `zero-norm' of $\bm a$, which counts the number of non-zero coordinates of $\bm a$.

	For a matrix $\bm M := ((M_{ij}))_{1\le i \le s,1\le j \le t} \in \mathbb{R}^{s\times t}$ we define the following norms: 
	\begin{itemize}
		\item $\|\bm M\|_F := \sqrt{\sum_{i=1}^s \sum_{j=1}^t M_{ij}^2}$, 
		
		\item $\|\bm M\|_\infty := \max_{1\le i\le s} \sum_{j=1}^t |M_{ij}|$ ,
		
		\item $\|\bm M\|_1 := \max_{1\le j\le t} \sum_{i=1}^s |M_{ij}|$, 
		
		\item $\|\bm M\|_2 := \sigma_{\max}(\bm M)$, where $\sigma_{\max}(\bm M)$ denotes the largest singular value of $\bm M$. 
		
	\end{itemize}
	Note that if $\bm M$ is a square matrix, then $\|\bm M\|_2= \max\{|\lambda_{\min}(\bm M)|, |\lambda_{\max}(\bm M)| \}$, where $\lambda_{\max}(\bm M)$ and $\lambda_{\min}(\bm M)$ denote the maximum and minimum eigenvalues of $\bm M$, respectively. 
	
	For positive sequences $\{a_n\}_{n\geq 1}$ and $\{b_n\}_{n\geq 1}$, $a_n = O(b_n)$ means $a_n \leq C_1 b_n$, $a_n =  \Omega(b_n)$ means $a_n \geq C_2 b_n$, and $a_n =  \Theta(b_n)$ means $C_2 n \leq a_n \leq C_1 b_n$, for all $n$ large enough and positive constants $C_1, C_2$. Similarly, $a_n \lesssim b_n$ means $a_n = O(b_n)$, $a_n \gtrsim b_n$ means $a_n = \Omega(b_n)$. Moreover, subscripts in the above notations,  for example  $\lesssim_\square$, $O_\square$ and $\Omega_\square$  denote that the hidden constants may depend on the subscripted parameters $\square$. 
	Finally, we say that $a_n = \tilde{O}(b_n)$, if $a_n \le C_1 (\log n)^{r_1} b_n$ and $a_n = \tilde{\Theta}(b_n)$, if $ C_2 (\log n)^{r_2} b_n \leq a_n \le C_1 (\log n)^{r_1} b_n$, for all $n$ large enough and some positive constants $C_1, C_2, r_1, r_2$.

	\subsection{Organization}
	
	The remainder of the paper is organized as follows. The rates of consistency of the estimates are presented in Section \ref{sec:covmain}. In Section \ref{sec:examples}, we apply our results to various common network models. The algorithm for computing the estimates and simulation results are presented in Section \ref{sec:computationandexperiments}. The proofs of the technical results are given in the Appendix. 
	

	\section{Rates of Consistency}\label{sec:covmain}

	Next, we present our results on rates of convergence of the PMPL estimator. In Section \ref{sec:mplconsistency},  we present the rates of convergence of the PMPL estimates $(\hat \beta, \hat {\bm \theta}^\top)$. The rate for estimating the regression parameters is presented in Section \ref{sec:estimationtheta}. 

	\subsection{Consistency of the PMPL Estimate} 
	\label{sec:mplconsistency} 
	
	We begin by stating the relevant assumptions: 
	
	\begin{assumption}\label{asm:a1} The interaction matrix $\bm A$ in \eqref{eq:logisticmodel} satisfies the following comdition:
		$$\sup_{N \ge 1} \|\bm A\|_\infty < \infty.$$
	\end{assumption}

	\begin{assumption}\label{asm:a2} The design matrix  $\bm Z := (\bm Z_1,\ldots, \bm Z_N)^\top$ satisfies 
		$$\liminf_{N\rightarrow\infty} \lambda_{\min}\left(\frac{1}{N} \bm Z^\top \bm Z \right) > 0.$$	  
	\end{assumption}

	\begin{assumption}\label{asm:a3} The signal parameters $\bh$ and the covariates $\{\bm Z_i\}_{1 \leq i \leq N}$ are uniformly bounded, that is, there exist positive constants $\Theta$ and $M$ such that $\|\bh\|_\infty < \Theta$ and $\|\bm Z_i\|_\infty < M$, for all $1\le i\le N$. 
	\end{assumption}
	
	Under the above assumptions we establish the rate of convergence of the PMPL estimate \eqref{eq:betatheta} given a single sample of observations from the model \eqref{eq:logisticmodel}, when the parameter vector is sparse, that is, $\| (\beta, \bh^\top)^\top \|_0 = s$. For notational convenience, we henceforth denote the $(d+1)$-dimensional vector of parameters by $\bg:=(\beta, \bh^\top)^\top$ and the  $(d+1)$-dimensional vector of PMPL estimates obtained from \eqref{eq:betatheta} by $\hat{\bg} = (\hat{\beta}, \hat{\bh}^\top)^\top$.

\begin{theorem}\label{thm:estimate} 
 Suppose that Assumptions \ref{asm:a1}, \ref{asm:a2}, \ref{asm:a3} hold, and $\liminf_{N \rightarrow \infty} \frac{1}{N} \|\bm A\|_F^2 > 0$.  Then, there exists a constant $\delta > 0$ such that by choosing $\lambda := \delta \sqrt{\log(d+1)/N}$ in the objective function in \eqref{eq:pseudolikelihood} we have,
	\begin{align}\label{eq:ratelogdN}
	\|\hat{\bg} - \bg \|_2 = O_s\left(\sqrt{\frac{\log d}{N}}\right)  \quad \text{ and } \quad \|\hat{\bg} - \bg \|_1  = O_s\left(\sqrt{\frac{\log d}{N}}\right) ,
	\end{align} 
	with probability $1-o(1)$, as $N\rightarrow \infty$ and $d \rightarrow \infty$ such that $d=o(N)$. 
\end{theorem}

	The conditions in Theorem \ref{thm:estimate} combined aim to strike a balance between the signal from the peer-effects to that from the covariates, to ensure consistent estimation for all $\beta$. In particular, a control on $\|\bm A\|_\infty$ is required to ensure that the peer effects coming from the quadratic dependence term in the probability mass function \eqref{eq:logisticmodel} do not overpower the effect of the signal $\bh$ coming from the linear terms $\bh^\top \bm Z_i$, thereby hindering joint recovery of the correlation term $\beta$ and the signal term $\bh$. At the same time, we also require the interaction matrix to be not too sparse, and its entries to be not too small, in order to ensure that the effect of the correlation parameter $\beta$ is not nullified. This is guaranteed by the condition $\|\bm A\|_F^2 = \Omega(N)$. For example, when $\bm A$ is the scaled adjacency of a graph $G_N$, then Assumption \ref{asm:a1} together with the condition $\|\bm A\|_F^2 = \Omega(N)$ implies that $G_N$ has bounded maximum degree (see Section \ref{sec:examples}). In fact, in the proof we keep track of the dependence on $\|\bm A\|_F$ in the error rate (see Proposition \ref{ppn:estimateF} in Section \ref{prfinalmain}), which allows us to establish consistency of the PMPL estimate beyond bounded degree graphs (see Section \ref{sec:GN} for details). 
	
\begin{rem} {\em It is worth noting that it may be impossible to estimate $(\beta,\bh)$ consistently without any diverging lower bound on $\|\bm A\|_F^2 = \Omega(N)$ or, in other words, if the graph $G_N$ is too dense. This phenomenon is observed in the Curie-Weiss Model (where the interaction matrix  $a_{ij}=1/N$, for $1 \leq i \ne j \leq N$) \citep{comets}. In this example, each entry of $\bm A$ is $O(1/N)$, and hence, $\|\bm A\|_F^2 = O(1)$, and even when $d=1$ (and $\bm Z_1=\ldots =\bm Z_N$)  consistent estimation of the parameters $\beta$ and $\theta$ is impossible (see  Theorem 1.13 in \citet{pg_sm}). } 
	\end{rem}

The proof of Theorem \ref{thm:estimate} is given in Section \ref{prfinalmain}. As mentioned before, this is a consequence of a more general result which gives rates of convergence for the PMPL estimate in terms of the $\|\bm A\|_F$ (Proposition \ref{ppn:estimateF}). 
	Broadly speaking, the proof involves the following two steps: 
	
	\begin{itemize} 
		
		\item {\it Concentration of the gradient}:  The first step in the proof of Theorem \ref{thm:estimate} is to show that the gradient of the logarithm of the pseudo-likelihood function $L_N$ (recall \eqref{eq:LNbetatheta}) is concentrated around zero in the $\ell_\infty$ norm. For this step, we use the conditioning trick introduced in \citet{cd_ising_estimation}, which reduces a general Ising model to an Ising model in the high-temperature regime, where exponential concentration inequalities for functionals of Ising models are available \citep{scthesis}. The details are formalized in Lemma \ref{eq:gradientestimate}. 
		
		\item {\it Strong-concavity of the pseudo-likelihood}:  In the second step, we show that the logarithm of the pseudo-likelihood function is strongly concave with high probability. This entails showing that the lowest eigenvalue of the Hessian of $L_N$ is bounded away from zero with high probability. Towards this, the minimum eigenvalue condition in Assumption \ref{asm:a2}, which is standard in the high-dimensional literature (see \citet{loh1,highdim_ising}), is crucial. In particular, this condition holds with high probability, if the covariates $\bm Z_1,\ldots,\bm Z_N$ are i.i.d. realizations from a sub-Gaussian distribution on $\mathbb{R}^d$ (see Theorem 2.1 in \citet{cd_ising_I}). Under Assumption \ref{asm:a2} and using lower bounds on the variance of linear projections of $\bm X$ developed in \citet{cd_ising_estimation} and concentration  results from \citet{radek}, we establish the strong-concavity of the pseudo-likelihood in Lemma \ref{rsccond}. 

	\end{itemize}

	\begin{rem}\label{remark:sdependence} 
		{\em Note that the rate in \eqref{eq:ratelogdN} suppresses the dependence on the sparsity parameter $s$ in the order term. From the proof of Theorem \ref{thm:estimate}, it will be seen that the dependence is, in general, exponential in $s$. However, if one replaces Assumption \ref{asm:a3} with the stronger assumption that the $\ell_2$ norms of the parameters and the covariates are \textit{bounded}, that is, $\|\bh\|_2 < \Theta$ and $\|\bm Z_i\|_2 < M$, for all $1\le i\le N$, then our proof can be easily modified to recover the standard high-dimensional  $O(\sqrt{s \log d /N})$ rate (see Remark \ref{remark:M}). In fact, this stronger assumption has been used recently in \citet{cddependent} to derive rates of the pseudo-likelihood estimate under $\ell_1$ sparsity.  Specifically, \citep[Theorem 2]{cddependent} showed that if $\| \bg \|_1 \leq s$ and the parameters and the covariates are $\ell_2$-bounded, their estimate $\tilde \bg$ under Assumption \ref{asm:a1} satisfies: 
			$$\| \tilde \bg - \bg \|_2 = O\left( \left(\frac{s \log d}{N}\right)^{\frac{1}{6}} \right) , $$ 
			with high probability. Note that the dependence on $N$ in the RHS above is worse than the expected $1/\sqrt N$-rate. Moreover, the $\ell_2$-boundedness of the covariates is quite restrictive in the high-dimensional setup. On the other hand, this work derives rates under $\ell_0$ sparsity and a more realistic $\ell_\infty$-bounded condition (Assumption \ref{asm:a3}). Under this condition, we are able to derive the correct dependence on $N$ and $d$ (and also on $s$, if the stronger $\ell_2$-boundedness is imposed as mentioned above) in the regime where $d=o(N)$. An exponential dependence on $d$ also appears in \citet{cd_ising_I} (see footnote in page 4), where the convergence rate of the MPL estimate is derived in the fixed $d$ regime. This rate can be improved to $O(\sqrt{d/N})$ under the $\ell_2$- boundedness assumption (see, for example, \citet{cd_ising_II}). Our results show that this can be further improved to $O(\sqrt{\log d/N})$ in the regime of constant sparsity.} 
	\end{rem}

	\subsection{Estimation of the Regression Coefficients}
	\label{sec:estimationtheta}

	In this section, we consider the problem of estimating the regression coefficients $\bh$, for fixed $\beta$. The goal is to understand how network dependence may	affect our ability to estimate the high-dimensional regression coefficient under sparsity constraints. Towards this, we study the properties of the following PMPL estimator for the regression coefficients $\bh$: 
	\begin{align}\label{eq:pltheta}
		\hat \bh :=  \mathrm{arg}\min_{ \bm \theta } \{  L_{\beta, N}( \bm \theta )   + \lambda || \bm \theta ||_1\}, 
	\end{align} 
	where $\lambda$ is a regularization parameter and (recalling \eqref{eq:LNbetatheta}) 
	\begin{align}\label{eq:LNtheta}
		L_{\beta, N}(\bm \theta) & =- \frac{1}{N}\sum_{i=1}^N \{ X_i ( \bm \theta^{\top} \bm Z_i  + \beta  m_i(\bm X) ) - \log \cosh ( \bm \theta^{\top} \bm Z_i  + \beta m_i(\bm X) )\} + \log 2.
	\end{align}

	To handle the high-dimensional regime, we need to make the following assumption on the  design matrix  $\bm Z := (\bm Z_1,\ldots, \bm Z_N)^\top$. To this end, we define the Rademacher complexity of the $\{\bm Z_i\}_{1 \leq i \leq N}$ as: 
	\begin{align}\label{eq:RZ} 
		\mathcal R_N: = \mathbb{E}\left(\left\|\frac{1}{N}\sum_{i=1}^{N}\varepsilon_i\Z_i\right\|_{\infty}\right) , 
	\end{align} 
	where $\{\varepsilon_i\}_{1 \leq i \leq N}$ is a sequence of i.i.d. Rademacher random variables and the expectation in \eqref{eq:RZ} is taken jointly over the randomness of $\{\bm Z_i\}_{1 \leq i \leq N}$ and 
	$\{\varepsilon_i\}_{1 \leq i \leq N}$.

	\begin{assumption}\label{asm:RSC_random} 
		Suppose the covariates $\Z_1, \Z_2, \ldots, \Z_N$ are drawn i.i.d. from a distribution  with mean zero and satisfying the following conditions: 
		
		\begin{itemize} 
			
			\item[$(1)$] There exist positive constants $\kappa_1, \kappa_2$ such that 
			\begin{align*}
				\mathbb{E}\left(\langle\bm \eta,\Z_1 \rangle^{2}\right)\geq \kappa_1 \text{ and } \mathbb{E}\left(\langle\bm \eta,\Z_1 \rangle^{4}\right)\leq \kappa_2,\end{align*}
			for all $\bm \eta\in\mathbb{R}^{d}$ such that $\|\bm \eta\|_2=1$. 
			
			\item[$(2)$] $\mathcal R_N =O(\sqrt{\log d/N})$.  
			
			\item[$(3)$] There exists a constant $C > 0$ such that $\max_{1 \leq j \leq d } \frac{1}{N} \sum_{i=1}^N Z_{ij}^2 \leq C$ holds with probability 1. 
			
		\end{itemize} 
		
	\end{assumption}
	
	%
	%

	These types of conditions are standard in the high-dimensional statistics literature (see~\citet{bickelritovtsybakov,candestao,ny,Mestimation,rwy,geerbuhlmann} and references therein), which are known to hold for many natural classes of design matrices. For example, if $\Z_1, \Z_2, \ldots, \Z_n$ are i.i.d. sub-Gaussian random variables with mean zero, then \citet[Exercise 9.8]{wainwrighthighdimensional} implies that $\mathcal R_N =O(\sqrt{\log d/N})$.

	\begin{rem} 
		{\em 
			As mentioned before, when $\Z_1, \Z_2, \ldots, \Z_n$ are i.i.d Gaussian with mean zero and covariance matrix $\Sigma$, then Assumption \ref{asm:RSC_random} (2) holds (by 
			\citet[Exercise 9.8]{wainwrighthighdimensional}). To understand when Assumption \ref{asm:RSC_random} (1) holds, note that for $\bm \eta \in\mathbb{R}^{d}$ such that $\|\bm \eta\|_2=1$ we have 
			$$\mathbb{E}\left(\langle\bm \eta,\Z\rangle^{2}\right) = \mathbb{E}( N(0, \bm \eta^\top \Sigma \bm \eta )^2 ) =  \bm \eta^\top \Sigma \bm \eta  \geq\lambda_{\min}(\Sigma)$$ 
			and 
			$$\mathbb{E}\left(\langle\bm \eta,\Z\rangle^{4}\right)=\mathbb{E}( N(0, \bm \eta^\top \Sigma \bm \eta )^4 )  = 3 ( \bm \eta^\top \Sigma \bm \eta )^4 \leq 3 \lambda^{2}_{\max}(\Sigma) , $$ 
			where $\lambda_{\min}(\Sigma)$ and $\lambda_{\max}(\Sigma)$ are the minimum and maximum eigenvalues of $\Sigma$, respectively. 
			Therefore, Assumption \ref{asm:RSC_random} (1) holds, if we assume that there exist positive constants $c_{*},c^{*}$, such that the covariance matrix $\Sigma$ satisfies $c_{*}\leq\lambda_{\min}(\Sigma)\leq\lambda_{\max}(\Sigma)\leq c^{*}$. 
		} 
	\end{rem}

	Under the above assumptions, we now show in the following theorem that one can consistently estimate the regression parameters of the model \eqref{eq:logisticmodel} at the same rate as the classical (independent) logistic regression model \eqref{jointlogistic}.

	\begin{theorem}\label{thm:stheta}
		Fix $\beta \in \mathbb{R}$. Suppose the interaction matrix $\bm A$ in \eqref{eq:logisticmodel} satisfies Assumptions \ref{asm:a1} and the covariates $\Z_1, \Z_2, \ldots, \Z_N$ satisfy Assumption \ref{asm:RSC_random}. Moreover, assume that there exists a positive constant $\Theta$ such that $\|\tht\|_2 \leq \Theta$. Then, there exists a constant $\delta > 0$ such that by choosing $\lambda := \delta \sqrt{\log d/N}$ in the objective function in \eqref{eq:pltheta} we have, 
		\begin{align*}
			\|\hat{\bh} - \bh \|_2 = O\left(\sqrt{\frac{s \log d}{N}}\right) \quad \text{ and } \quad \|\hat{\bh} - \bh \|_1 = O\left(s \sqrt{\frac{\log d}{N}}\right),
		\end{align*} 
		with probability $1-o(1)$, as $N, d\rightarrow \infty$ such that $s \sqrt{\log d/N} = o(1)$. 
	\end{theorem}

	The proof of Theorem \ref{thm:stheta} is given in Section \ref{sec:sthetapf} in the Appendix. 
	We follow the strategy outlined in \citet[Chapter 9]{wainwrighthighdimensional} for showing rates of consistency for high-dimensional generalized linear models. In particular, we show that the  pseudo-likelihood loss function satisfies the restricted strong concavity condition (Proposition \ref{proposition:RSC_random}) under Assumption \ref{asm:RSC_random}. Consequently, we can establish the consistency of the PMPL estimate of the regression parameters in the entire high-dimensional regime (where $d$ can be much larger than $N$) and also recover the correct dependence on the sparsity $s$.

	\begin{rem} {\em Note that, unlike in Theorem \ref{thm:estimate}, the Frobenius norm assumption $\|\bm A\|_F^2 = \Omega(N)$, is not required in Theorem \ref{thm:stheta}. In particular, the only assumption on $\bm A$
one needs in Theorem \ref{thm:stheta} is $\| \bm A \|_\infty < 1$ (Assumption \ref{asm:a1}). For example, when $\bm A$ is the scaled adjacency matrix of a graph $G_N$, the assumption $\| \bm A \|_\infty < 1$ is equivalent to the maximum degree of $G_N$ being of the same order as the average degree of  $G_N$ (see \eqref{eq:degree1}). 
This is expected because when $\beta$ is known, the parameter $\bm \theta$ can be estimated at the classical high-dimensional rate, irrespective of the total edge density of the network, as long  as the peer effects coming from the quadratic dependence term do not overshadow effect of the linear term $\bm \theta^\top \bm Z$, which is ensured by the condition $\| \bm A \|_\infty < 1$. This condition, in particular, implies that no node of the network has an unduly large effect on the corresponding model, and is satisfied by most Ising models that are commonly studied in the literature. 
} 
	\end{rem}
	
%
%

	\section{Application to Various Network Structures}\label{sec:examples}

	In this section, we apply Theorem \ref{thm:estimate} to establish the consistency of the PMPL estimator \eqref{eq:pseudolikelihood} for various natural network models. To this end, let $G_N=(V(G_N), E(G_N))$ be a sequence of graphs with $V(G_N)=[N]:=\{1, 2, \ldots, N\}$ and adjacency matrix $\cA(G_N)=((a_{ij}))_{1 \leq i, j \leq N}$. We denote by $d_v$ the degree of the vertex $v \in V(G_n)$.  To ensure that the model \eqref{eq:logisticmodel} has non-trivial scaling properties, one needs to chose the interaction matrix $\bm A$ as the scaled adjacency matrix of $G_N$. In particular, define 
	\begin{equation}\label{eq:graph}
		\bm A_{G_N} := \frac{N}{2|E(G_N)|} \cdot \mathcal{A}(G_N)
	\end{equation}
	Throughout this section, we consider model \eqref{eq:logisticmodel} with $\bm A=\bm A_{G_N}$ as above. We also assume that the number of non-isolated vertices in $G_N$ (that is, the number of vertices in $G_N$ with degree at least 1) is $\Omega(N)$. Note that this implies, $|E(G_N)| \gtrsim N$. Finally, we also assume that the sparsity $s=O(1)$ and consequently, absorb the dependence on $s$ in the $O$-terms (recall Remark \ref{remark:sdependence}).

	\subsection{Bounded Degree Graphs}

	A sequence of graphs $\{G_N\}_{N \geq 1}$ is said to have bounded maximum degree if its maximum degree is uniformly bounded, that is, 
	$\sup_{N \geq 1} d_{\mathrm{max}} < \infty$, where $d_{\mathrm{max}} := \max_{v \in V(G_n)} d_v$ is the maximum degree of $G_N$. Note that if $G_N$ has bounded maximum degree and has $\Omega(N)$ non-isolated vertices, then $|E(G_N)| = \Theta(N)$.

	Networks arising in certain applications, especially those with an underlying spatial or lattice structure generally have bounded degree. These include planar maps which encode neighborhood relations \citep{planar1,planar2}, lattice models for capturing nearest-neighbor interactions between image pixels, and demand-aware networks \citep{demandaware} among others. It is easy to check that the conditions in Assumption \ref{asm:a1} are satisfied for bounded degree graphs. Towards this, note that, under the scaling in \eqref{eq:graph}, the condition $\sup_{N\ge 1} \|\bm A\|_\infty < \infty$ is equivalent to $ \|\bm A_{G_N}\|_\infty = \frac{N}{2|E(G_N)|}  \max_{v \in V(G_n)} d_v$. Hence, under the scaling in \eqref{eq:graph}, the assumption $\sup_{N\ge 1} \|\bm A_{G_N}\|_\infty < \infty$ is equivalent to
	\begin{align}\label{eq:degree1}
		d_{\max} := \max_{v \in V(G_n)} d_v= O\left(\frac{|E(G)|}{N}\right), 
	\end{align}  
	that is, the maximum degree of $G_N$ is of the same order as its average degree. Moreover, the condition $\liminf_{N\rightarrow \infty} \frac{1}{N} \|\bm A_{G_N}\|_F^2 > 0$ is equivalent to 
	\begin{align}\label{eq:degree2}
		\limsup_{N \rightarrow \infty} \frac{|E(G_N)|}{N} < \infty,
	\end{align} 
	that is, the average degree of $G_N$ is bounded. Therefore, \eqref{eq:degree1} and \eqref{eq:degree2} are together equivalent to the condition that $G_N$ has bounded maximum degree. Hence, the PMPL estimate is $\sqrt{\log d /N}$-consistent for any sequence of graphs of  bounded maximum degree, whenever the assumptions in Theorem \ref{thm:estimate} hold.

	\subsection{Sparse Inhomogeneous Random Graphs} 
	
	Although Theorem \ref{thm:estimate} requires that the maximum degree of $G_N$ has to be of the same order as the average degree (see \eqref{eq:degree1}), our proofs can be easily adapted to establish similar rates of consistency of the PMPL estimate (up to $\mathrm{polylog}(N)$ factors), if the maximum degree $G_N$ grows poly-logarithmically with respect to the average degree, which, in particular, is the case for sparse inhomogeneous random graphs. This is summarized in the following result. The proof is given in Section \ref{proof_sparse} of the Appendix.

	\begin{theorem}\label{thm:sparse} Suppose $\{G_N\}_{N \geq 1}$ is a sequence of graphs with $|E(G_N)| = O(N)$, $d_{\max} = \te(1)$, and $\Omega(N)$ non-isolated vertices. Then for $\lambda = \tilde \Theta \left( \sqrt{\log d/N}\right)$, 
		\begin{align*}
			\| \hat{\bg} - \bg \|_2 = \te\left( \frac{1}{\sqrt N}\right) 
		\end{align*}
		with probability $1-o(1)$ as $N,d \rightarrow \infty$ such that $d = o(N)$. 
	\end{theorem}


	Theorem \ref{thm:sparse} can be applied to obtain rates of convergence in sparse inhomogeneous random graph models. 
	
	\begin{definition} \citep{bollobas} 
		Given a symmetric matrix $\bm P^{(N)} = ((p_{uv})) \in [0, 1]^{N \times N}$ with zeroes on the diagonal, the {\it inhomogeneous random graph} $\mathcal G(N, \bm P^{(N)})$ is the graph with vertex set $[N]:=\{1, 2, \ldots, N\}$ where the edge $(u, v)$ is present with probability $p_{uv} $, independent of the other edges, for every $1\le u < v \le N$. 
	\end{definition}

	The class of  inhomogeneous random graph models defined above includes several popular network models, such as the Chung-Lu model \citep{degree_network}, the $\beta$-model \citep{chatterjee_degree}, random dot product graphs \citep{young2007random,tang17}, and stochastic block models \citep{holland1983stochastic,lei16}. 
	Next, we consider the sparse regime wherein  
	\begin{align}\label{eq:P1} 
		\max_{1\le u, v \le N} p_{uv}  = O\left(\frac{1}{N}\right). 
	\end{align} 
	In this regime, the expected degree remains bounded, although the maximum degree can diverge at rate $O(\log N)$ \citep{eigenerdos,largest_eigenvalue}. We will also assume that there exists $\varepsilon \in (0,1)$ and $\Omega(N)$ vertices $u \in G_N$, such that
	\begin{align}\label{eq:P2}
		\limsup_{N\rightarrow \infty} \prod_{v=1}^N \left(1- p_{uv}\right) < \varepsilon. 
	\end{align} 
	This will ensure $\mathcal G(N, \bm P^{(N)})$ has $\Omega(N)$ non-isolated vertices. Under these assumptions we have the following result:

	\begin{corollary}\label{inh}
		Suppose $G_N$ is a realization of the inhomogeneous random graph  $\mathcal G(N, \bm P^{(N)})$, where $\bm P^{(N)}$ satisfies the conditions in \eqref{eq:P1} and \eqref{eq:P2}. Then for $\lambda = \tilde \Theta \left( \sqrt{\log(d+1)/N}\right)$, 
		$$\| \hat{\bg} - \bg \|_2 = \te\left( \frac{1}{\sqrt{N}}\right),$$ 
		with probability $1-o(1)$ as $N,d \rightarrow \infty$ such that $d = o(N)$. 
	\end{corollary}
	
	Corollary \ref{inh} is proved in Section \ref{sec:inhomogeneouspf} of the Appendix. In the following example, we illustrate how it can be applied to sparse stochastic block models, in particular, sparse Erd\H{o}s-R\'enyi random graphs.

	\begin{example}[Sparse stochastic block models] {\em Fix $K\geq 1$,  a vector of community proportions $\bm \lambda := (\lambda_1,\ldots,\lambda_K) \in (0, 1)^K$, such that $\sum_{j=1}^K \lambda_j = 1$, and a symmetric probability matrix $\bm B := ((b_{ij}))_{1\leq i, j \leq K}$,  where $b_{ij} \in [0, 1]$, for all $1\leq i, j\leq K$ and $b_{ij} > 0$ for some $ 1\leq i, j \leq K$. The (sparse)  stochastic block model with proportion vector $\bm \lambda$ and probability matrix $\bm B$ is the inhomogeneous random graph $\mathcal G(N, \bm P^{(N)})$, with  
			$\bm P^{(N)} = ((p_{uv}))_{1 \leq u, v \leq N}$ where  
			\begin{align}\label{eq:psbm}
				p_{uv} = \frac{b_{ij}}{N} \quad \textrm{ for } \quad (u, v) \in B_{i} \times B_{j}, 
			\end{align} 
			where $B_j := (N \sum_{i=1}^{j-1}\lambda_i, N \sum_{i=1}^j \lambda_i] \bigcap [N]$, for $j \in \{1,\ldots,K\}$. In other words, the set of vertices is divided into $K$ blocks (communities) $B_1, B_2, \ldots, B_K$, such that the edge between vertices $u \in B_i$ and $v \in B_j$ occurs independently with probability $b_{ij}/N$. 
			Clearly, in this case \eqref{eq:P1} holds. Next, to check that \eqref{eq:P2} holds, choose $1\leq i, j \leq K$ such that $b_{ij} > 0$. Then for all $u \in B_i$,  $$\limsup_{N \rightarrow \infty} \prod_{v=1}^N \left(1- p_{u v} \right) \le \limsup_{N \rightarrow \infty} \prod_{v \in B_j} \left(1-\frac{b_{ij}}{N}\right) = \exp(- \lambda_j b_{ij}) < 1,$$ 
			which verifies \eqref{eq:P2}, since $|B_j| = \Omega (N)$. Hence, by Corollary \ref{inh}, the PMPL estimate \eqref{eq:betatheta} is $\te(1/\sqrt{N})$-consistent in this example. 
			As a consequence, the PMPL estimate is $\te(1/\sqrt{N})$-consistent for sparse 
			Erd\H{o}s-R\'enyi random graphs $G(N, c/N)$, which corresponds to setting $b_{ij} = c $, for all $1 \leq i, j \leq K$ in \eqref{eq:psbm}. } 
		\label{example:Pgraph}
	\end{example}

	\subsection{Beyond Bounded Degree Graphs} 
	\label{sec:GN}
	
We can also establish the consistency of the PMPL estimate beyond bounded degree graphs using Proposition \ref{ppn:estimateF}, which provides error rates in terms of $\| \bm A \|_F$. To this end, note that when $\bm A = \bm A_{G_N}$ is the scaled adjacency matrix of $G_N$ as in \eqref{eq:graph}, then 
 $$\| \bm A \|_F = \Theta\left(\frac{N}{\sqrt {|E(G_N)|}}\right).$$
	Hence, whenever \eqref{eq:degree1} holds, Proposition \ref{ppn:estimateF} implies, 
		\begin{align*}
			\|\hat{\bg} - \bg \|_2 = O_s\left(\sqrt{\frac{ |E(G_N)|^2 \log d }{N^3}}\right)  \quad 		\end{align*} 
		with probability $1-o(1)$, whenever $d = o(N^2/|E(G_N)|)$. This shows that the PMPL estimate is consistent whenever $|E(G_N)| = o(N^{3/2})$ (up to log-factors) and $d = o(N^2/|E(G_N)|)$.  In particular, if $G_N$ is $\Delta$-regular (that is, all vertices have of $G_N$ has degree $\Delta$), then the rate of convergence becomes $O_s( \Delta  \sqrt{ \log d/N})$, if $d = o(N/\Delta)$.

	\section{Computation and Experiments}  
	\label{sec:computationandexperiments}

	Next, we discuss an algorithm for computing the PMPL estimates (Section \ref{sec:computation}) and evaluate its performance in numerical experiments using synthetic data (Section \ref{sec:experiments}).

	\subsection{Computation of the PMPL Estimates} 
	\label{sec:computation}

	A classical method developed for solving sparse estimation problems is the proximal descent algorithm \citep{friedman2010regularization}. We employ this algorithm to the optimization problem \eqref{eq:pltheta}. To describe the algorithm, let
	\begin{align} \label{eq:LNh}
		f(\bm z) : = L_N(\bm z) + \lambda ||\bm z||_1 ,  
	\end{align}  
	for $\bm z \in \R^{d+1}$, with $L_N( \cdot )$ as defined in \eqref{eq:LNbetatheta}. Also, for $t \in \R$ and $\bm x \in \R^{d+1}$ define 
	\begin{align*}
		G_t(\bm x)=\frac{1}{t}\left(\bm x-\mathrm{prox}_{t}(\bm x-t \nabla L_N(\bm x))\right) , 
	\end{align*}
	where 
	\begin{align*} 
		\mathrm{prox}_{t}(\bm x) :=\arg\min_{\bm z \in \R^{d+1}}\left\{\frac{1}{2t}\left\|\bm x - \bm z \right\|_2^{2} + \lambda ||\bm z||_1 \right\} =   \left( \frac{x_i}{t} \left( 1 - \frac{\lambda}{\left | \frac{x_i}{t} \right| } \right)_+ \right)_{1 \leq i \leq d+1}, 
	\end{align*}
	is minimized by the soft thresholding estimator. To chose the step size in the proximal descent algorithm, we employ a backtracking line search, which is commonly used in gradient-based as well as in lasso-type problems  \citep{qin2013efficient}. To justify this we invoke the following result applied to the function $f$ defined in \eqref{eq:LNh}: 
	
	\begin{proposition}\citep{vandenberghe}~
		Fix $s \geq 1$ and a step size $t > 0$. Suppose at the $s$-th iteration the following bound holds: 
		\begin{align}\label{eq:line_search}
			L_{N}(\bm \gamma^{(s)}-tG_t(\bm \gamma^{(s)}))\leq L_{N}(\bm \gamma^{(s)})-t\nabla L_{N}(\bm \gamma^{(s)})^{\top}G_t(\bm \gamma^{(s)})+\frac{t}{2}\|G_{t}(\bm \gamma^{(s)})\|_2^{2} . 
		\end{align} 
		Then for all $\bm \gamma \in \R^{d+1}$, 
		\begin{align}\label{eq:fgamma}
			f(\bm \gamma^{(s)}-tG_{t}(\bm \gamma^{(s)}))\leq f( \bm \gamma)+G_t(\bm \gamma^{(s)})^{\top} (\bm \gamma^{(s)}-\bm \gamma)-\frac{t}{2}\|G_{t}(\bm \gamma^{(s)})\|_2^2 . 
		\end{align}
		\label{ppn:LNgamma}
	\end{proposition}

	Note that setting  $\bm \gamma=\bm \gamma^{(s)}$ in \eqref{eq:fgamma} gives, 
	\begin{align*}
		f(\bm \gamma^{(s)}-tG_{t}(\bm \gamma^{(s)}))\leq f(\bm \gamma^{(s)})-\frac{t}{2}\|G_{t}(\bm \gamma^{(s)})\|_2^2.
	\end{align*} 
	This shows that whenever the line-search condition \eqref{eq:line_search} holds, the descent of the objective function is guaranteed by setting 
	$$\bm \gamma^{(s+1)}  \gets \bm \gamma^{(s)}-tG_{t}(\bm \gamma^{(s)}) = \mathrm{prox}_{ t }(\bm \gamma^{(s)}- t \nabla L_{N}(\bm \gamma^{(s)})) .$$ 
	Therefore, the proximal gradient descent algorithm for optimization problem \eqref{eq:pltheta}, with step size chosen by backtracking line search, proceeds in the following two steps:  We initialize with $\bm \gamma^{(0)} = \bm 0 \in \R^{d+1}$ and $t^{(0)} =1$. 
	\begin{itemize} 
		
		\item  If, at the $s$-th iteration $(\bm \gamma^{ (s) }, t^{(s)})$ satisfies the line-search condition \eqref{eq:line_search}, then we update  the estimates 
		\begin{align}\label{eq:gammas1}
			\bm \gamma^{(s+1)} & \gets\mathrm{prox}_{ t^{(s)} }(\bm \gamma^{(s)}-t^{(s)}\nabla L_{N}(\bm \gamma^{(s)})) \nonumber \\ 
			& = \left(\frac{\bm \gamma^{(s)}}{t^{(s)}} - \nabla L_{N}(\bm \gamma^{(s)}) \right)\left(\bm{1}-\frac{\lambda}{\left|  \frac{\bm \gamma^{(s)}}{t^{(s)}} - \nabla L_{N}( \bm \gamma^{(s)}) \right|} \right)_{+} ,   
		\end{align}
		and keep the step size unchanged, that is, $t^{(s+1)}\gets t^{(s)}$. 
		
		\item If at the $s$-th iteration $(\bm \gamma^{ (s) }, t^{(s)})$ does not satisfy the line-search condition \eqref{eq:line_search} we shrink the step size by a factor of $\tau \in (0, 1)$, that is, $t^{(s+1)}\gets \tau t^{(s)}$, and keep the estimates unchanged, that is, $\bm \gamma^{ (s+1) } \gets \bm \gamma^{ (s) }$. 
		
	\end{itemize} 
	The procedure is summarized in Algorithm \ref{alg:algorithmmple}. 
	
	\begin{algorithm}
		\caption{}\label{alg:algorithmmple}
		\begin{algorithmic}
			\State Fix a value of $\tau \in (0, 1)$ and $\delta > 0$. Initialize with $\bm \gamma^{(0)} = \bm 0 \in \R^{d+1}$ and $t^{(0)} =1$. 
			\While{ $\|\bm \gamma^{(s+1)}-\bm \gamma^{(s)}\|_1> \delta$}
			
			{    
				\If{$L_{N}(\bm  \gamma^{(s)}-t^{(s)} G_{t^{(s)}}(\bm  \gamma^{(s)}))\leq L_{N}(\bm  \gamma^{(s)})-t^{(s)}\nabla L_{N}(\bm  \gamma^{(s)})^{\top}G_{t^{(s)}} (\bm  \gamma^{(s)})+\frac{t^{(s)}}{2}\|G_{t^{(s)}}(\bm  \gamma^{(s)})\|_2^{2}$}
				\State $$\bm \gamma^{(s+1)}\gets\mathrm{prox}_{ t^{(s)} }(\bm \gamma^{(s)}-t^{(s)}\nabla L_{N}(\bm \gamma^{(s)})) \text{ and } t^{(s+1)}\gets t^{(s)} . $$ 
				
				\EndIf
			}
			
			{
				\If{$L_{N}(\bm \gamma^{(s)}-t^{(s)}G_{t^{(s)}}(\bm \gamma^{(s)}))> L_{N}(\bm \gamma^{(s)})-t^{(s)}\nabla L_{N}(\bm \gamma^{(s)})^{\top}G_{t^{(s)}}(\bm \gamma^{(s)})+\frac{t^{(s)}}{2}\|G_{t^{(s)}}(\bm \gamma^{(s)})\|_2^{2}$}
				\State $$ \bm \gamma^{(s+1)}\gets \bm \gamma^{(s)} \text{ and } t^{(s+1)}\gets\tau t^{(s)} . $$ 
				\EndIf
			}
			\EndWhile
		\end{algorithmic}
	\end{algorithm}

	Note that the smooth part of the objective function $L_N$ is differentiable and its gradient $\nabla L_N$ is Lipschitz (by Lemma \ref{lm:LNgamma12}). Hence, Algorithm \ref{alg:algorithmmple} reaches $\varepsilon$-close to the optimum value in $O(1/\varepsilon)$ iterations \citep{vandenberghe}.

	\subsection{Numerical Experiments}
	\label{sec:experiments}

	We evaluate the performance of the PMPL estimator using Algorithm \ref{alg:algorithmmple} on synthetic data. The first step is to develop an algorithm to sample from model \eqref{eq:logisticmodel}. As mentioned before, direct sampling from the model \eqref{eq:logisticmodel} is computationally challenging due to the presence of an intractable normalizing constant. To circumvent this issue, we deploy a Gibbs sampling algorithm which iteratively updates each outcome variable $X_i$, for $1 \leq i \leq N$, based on the conditional distribution $\mathbb{P}(X_i| (X_j)_{j \ne i}, \Z)$ (recall \eqref{eq:probability_conditional}). Formally, the sampling algorithm can be described as follows: 
	\begin{itemize}
		\item Start with an initial configuration $\bm X^{(0)} := (X_i^{(0)})_{1 \leq i \leq N} \in \{ -1, +1\}^N$.
		\item At the $(s+1)$-th step, for $s \geq 1$, choose a vertex of $G_N$ uniformly at random. If the vertex $1 \leq i \leq N$ is selected, then update $X_i^{(s)}$ to
		\[
		X_i^{(s+1)} = 
		\begin{cases}
			+1,      & \text{with probability  } \mathbb P(X_i^{(s)} = 1 | (X_j^{(s)})_{ j \neq i }, \bm Z)\\
			-1,      & \text{with probability  } \mathbb P(X_i^{(s)} = -1 | (X_j^{(s)})_{ j \neq i }, \bm Z )
		\end{cases},
		\]
		and keep $X_j^{(s + 1)} = X_j^{(s)}$, for $j \neq i$. Define $\bm X^{(s+1)} := ( X_i^{(s+1)} )_{1 \leq i \leq N}$. 
	\end{itemize} 
	The Markov chain $\{\bm X^{(s+1) }\}_{s \geq 0}$ has stationary distribution \eqref{eq:logisticmodel} and, hence, can be used to generate approximate samples from \eqref{eq:logisticmodel}.  
	
	For the numerical experiments, we consider $\beta=0.3$ and choose the first $s$ regression coefficients $\theta_1, \theta_2, \ldots, \theta_s$ independently from $\mathrm{Uniform}([-1, -\frac{1}{2}] \cup [\frac{1}{2}, 1])$, while the remaining $d-s$ regression coefficients $\theta_{s+1}, \theta_{s+2}, \ldots, \theta_d$ are set to zero. The covariates $\bm Z_1, \bm Z_2, \ldots, \bm Z_N$ are sampled i.i.d. from a $d$-dimensional multivariate Gaussian distribution with mean vector $\bm 0$ and covariance matrix $\Sigma = ((\sigma_{ij}))_{1 \leq i, j \leq d}$, with $\sigma_{ij} = 0.2^{|i-j|}$. With the aforementioned choices of the parameters and the covariates, we generate a sample from the model \eqref{eq:logisticmodel} by running the Gibbs sampling algorithm described above for 30000 iterations.  We then apply Algorithm \ref{alg:algorithmmple} by setting $\varepsilon=0.001$, $\tau=0.8$, and consider the solution paths of the PMPL estimator as a function of $\log (\lambda)$, when the network $G_N$ is selected to be the Erd\H os-R\'enyi (ER) model and the stochastic block model (SBM). We set the range of $\lambda$ to be a geometric sequence of length $100$ from $0.001$ to $0.1$.

	\begin{itemize} 
		
		\item Figure \ref{fig:solutionpatherdosrenyi} depicts the solution paths of the PMPL estimate when $G_N$ is a realization of the Erd\H os-R\'enyi random graph $G(N, 5/N)$. Figure \ref{fig:solutionpatherdosrenyi} (a), corresponds to a setting $N=1200$, $d=200$, $s=5$, while Figure \ref{fig:solutionpatherdosrenyi} (b) to $N=1200$, $d=600$, $s=600$.
		
		\begin{figure}[ht]
			\begin{minipage}[b]{0.48\linewidth}
				\centering
				\includegraphics[width=\textwidth]{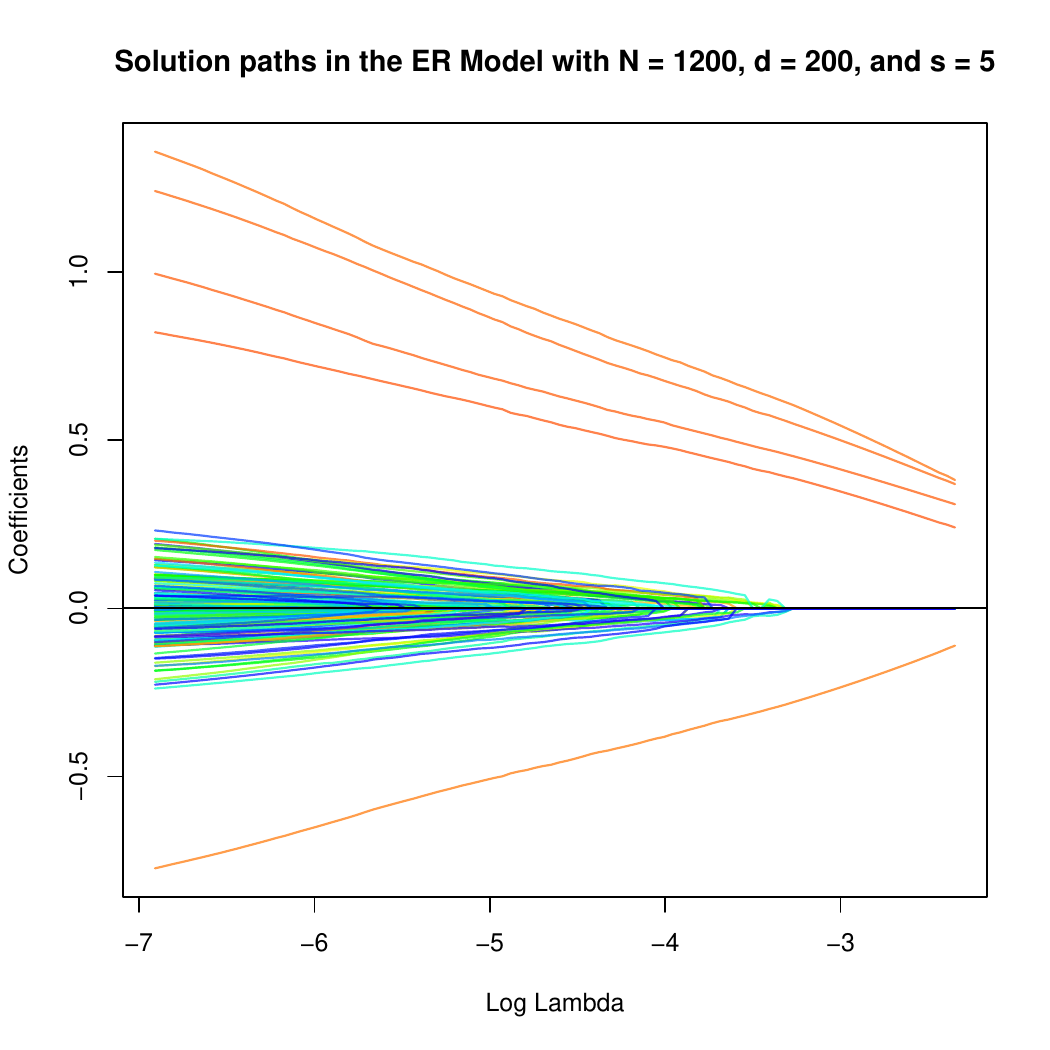}
				\caption*{(a)}
			\end{minipage}
			\begin{minipage}[b]{0.48\linewidth}
				\centering
				\includegraphics[width=\textwidth]{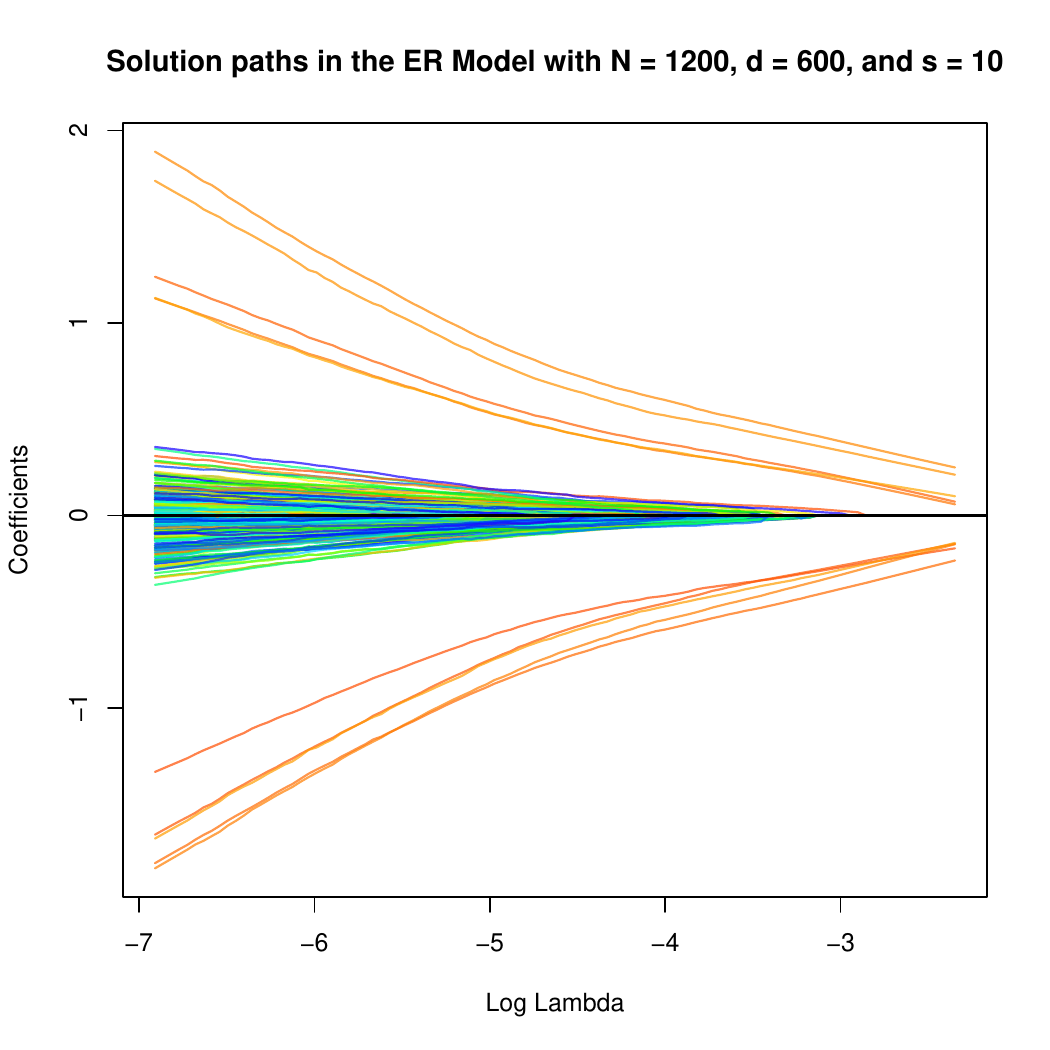}
				\caption*{(b)}
			\end{minipage}
			\caption{\small{Solution paths of the PMPL estimates in the Erd\H os-R\'enyi model $G(N, 5/N)$: (a)  $N=1200$, $d=200$, $s=5$, and (b)  $N=1200$, $d=600$, $s=10$. }}
			\label{fig:solutionpatherdosrenyi} 
		\end{figure} 
		
		\item Figure \ref{fig:solutionpathsbm} shows the solution paths of the PMPL estimate when $G_N$ is a realization of a SBM with $K=2$, $\lambda_1 = \lambda_2 = \frac{1}{2}$, $p_{11}=p_{22}=4/N$, and $p_{12}=8/N$ (that is, a SBM with 2 equal size blocks with within block connection probability $4/N$ and between block connection probability $8/N$ (recall Example \ref{example:Pgraph})). In 
		Figure \ref{fig:solutionpathsbm} (a) we have $N=1200$, $d=200$, $s=5$, and in  Figure \ref{fig:solutionpathsbm} (b) $N=1200$, $d=600$, $s=10$. 
	\end{itemize} 
	From the plots in Figures \ref{fig:solutionpatherdosrenyi} and \ref{fig:solutionpathsbm}, it is evident that the first $5$ signal (non-zero) coefficients remain non-zero throughout the range of tuning parameters $\lambda$ considered.  Moreover, as expected, $\lambda$ needs to be larger when $d=600$ for the non-signal (zero) coefficients to shrink to zero exactly.

	\begin{figure}[ht]
		\begin{minipage}[b]{0.48\linewidth}
			\centering
			\includegraphics[width=\textwidth]{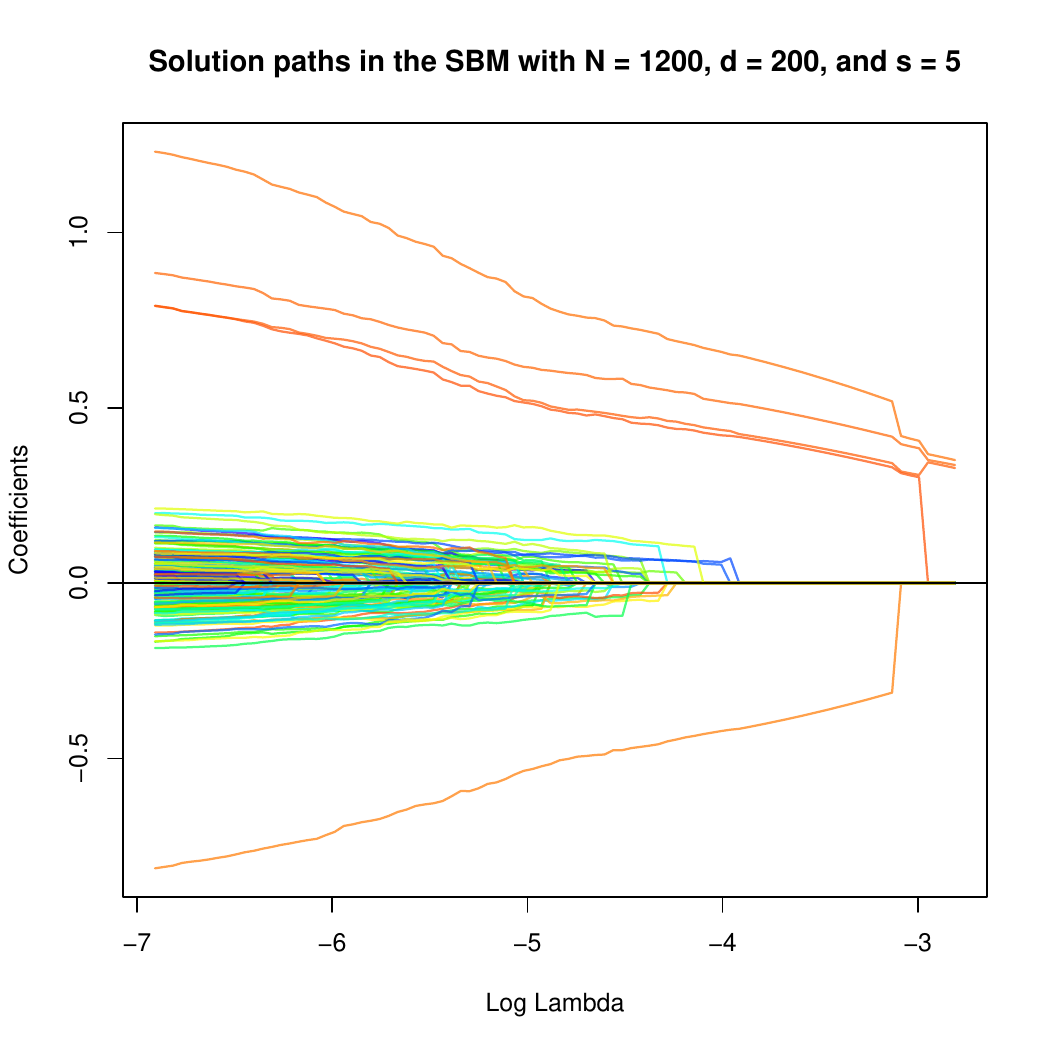}
			\caption*{(a)}
		\end{minipage}
		\begin{minipage}[b]{0.48\linewidth}
			\centering
			\includegraphics[width=\textwidth]{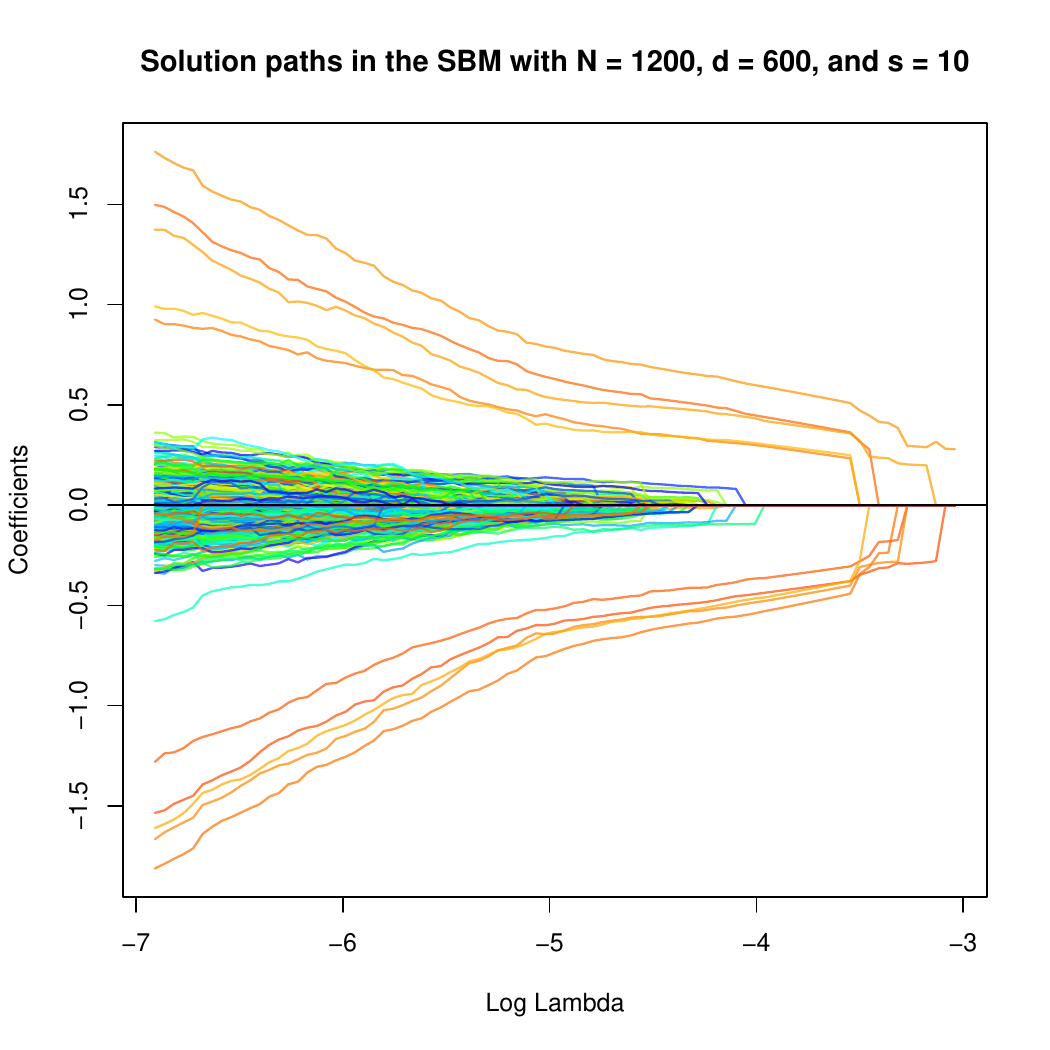}
			\caption*{(b)}
		\end{minipage}
		\caption{\small{Solution paths of the PMPL estimates in the stochastic block model: (a)  $N=1200$, $d=200$, $s=5$, and (b)  $N=1200$, $d=600$, $s=10$. }}
		\label{fig:solutionpathsbm} 
	\end{figure}

	Next, we investigate the estimation errors  by varying the size $N$ of the network $G_N$. To select the regularization parameter $\lambda$, we use a Bayesian Information Criterion (BIC). Specifically, we define 
	\begin{align*}
		\mathrm{BIC}(\lambda) = L_N(\hat \beta_\lambda, \hat{\bm \theta}_\lambda )   + \mathrm{df}(\lambda) \log N , 
	\end{align*} 
	where $\hat \beta_\lambda$, $\hat{\bm \theta}_\lambda=(\hat \theta_{\lambda, i})_{1 \leq i \leq d}$ are the PMPL estimates obtained from \eqref{eq:betatheta} for a fixed value of $\lambda$ and $\mathrm{df}(\lambda) = |\{1 \leq i \leq d: \hat \theta_{\lambda, i} \ne 0 \}|$. We choose $\hat \lambda$ by minimizing $\mathrm{BIC}(\lambda)$ over a grid of values of $\lambda$ and denote the corresponding PMPL estimates by $\tilde{\bm \gamma} = (\hat{\beta}_{\hat \lambda}, \hat{\bm \theta}_{\hat \lambda}^\top)$.  Figure \ref{fig:L12} shows the average $\ell_1$ and $\ell_2$ estimation errors  $\|\tilde{\bm \gamma} - \bm \gamma \|_1$ and $\|\tilde{\bm \gamma} - \bm \gamma \|_2$ and their 1-standard deviation error bars (over 200 repititions) for the Erd\H os-R\'enyi (ER) model and the SBM. We refer to these by \texttt{IsingL1} and \texttt{IsingL2} in Figure \ref{fig:L12}, respectively.  For comparison purposes, we also show the $\ell_1$ and $\ell_2$ estimation errors for the classical penalized logistic regression (with no interaction term, that is, $\beta=0$), denoted by \texttt{LogisticL1} and \texttt{LogisticL2} in Figure \ref{fig:L12}, respectively. The parameters in the numerical experiment are set as follows: $\beta=0.3$, $d=50$, the first $s=5$ regression coefficients $\theta_1, \theta_2, \ldots, \theta_5$ are independent samples from $\mathrm{Uniform}([-1, -\frac{1}{2}] \cup [\frac{1}{2}, 1])$ and the remaining 45 regression coefficients $\theta_6, \theta_7, \ldots, \theta_{50}$ are set to zero. As before, the  covariates $\bm Z_1, \bm Z_2, \ldots, \bm Z_N$ are sampled i.i.d. from a 50-dimensional multivariate Gaussian distribution with mean vector $\bm 0$ and covariance matrix $\Sigma = ((\sigma_{ij}))_{1 \leq i, j \leq 100}$, with $\sigma_{ij} = 0.2^{|i-j|}$.
	\begin{itemize} 
		
		\item Figure \ref{fig:L12} (a) shows the estimation errors when $G_N$ is a realization of the Erd\H os-R\'enyi random graph $G(N, 1/N)$, as $N$ varies from 200 to 1200 over a grid of 6 values. 
		
		\item Figure \ref{fig:L12} (b) shows the estimation errors when $G_N$ is a realization of a SBM with $K=2$, $\lambda_1 = \lambda_2 = \frac{1}{2}$, $p_{11}=p_{22}=0.5/N$, and $p_{12}=1/N$, with $N$ varying as before. 
   
        \item Figure \ref{fig:L12} (c) shows the estimation errors when $G_N$ is a realization from a $\beta$-model \citep{chatterjee_degree}. The $\beta$-model is an inhomogeneous random graph model where each edge $(i, j)$, for $1 \leq u < v \leq N$, is present independently with probability $$p_{uv} = \frac{e^{\beta_u+ \beta_v}}{1+ e^{\beta_u + \beta_v}},$$
     with $(\beta_1,\ldots,\beta_N) \in \R^n$. In Figure \ref{fig:L12} (c) we chose $\beta_{u}=- c \cdot u \log(\log(u + 1)) $, for $1 \leq u \leq N$, where $c = 200/N$ and $N$ varies from  200 to 1200 over a grid of 6 values. 
     
        \item Figure \ref{fig:L12} (d) shows the estimation errors when $G_N$ is a realization from the linear preferential attachment model with one edge added each time,  with $N$ varying as before. 
The linear preferential attachment  graph evolves sequentially one vertex at a time, where each new vertex connects to an existing vertex with probability proportional to their degrees (see \citet{bollobas2003directed,krapivsky2001organization}). Consequently, the model exhibits the `rich gets richer' phenomenon and the degree sequence follows a power law distribution \citep{barabasi1999emergence}. 
        
	\end{itemize} 
	
	The plots in Figure \ref{fig:L12} show that the estimation errors of PMPL estimates exhibit a decreasing trend as $N$ increases, validating the consistency results established in Section \ref{sec:covmain}. Although the $\ell_2$ errors of the PMPL and penalized logistic regression estimates are similar for small $N$, the PMPL errors are better as $N$ increases. Also, the difference between the $\ell_1$ errors of the PMPL and penalized logistic regression estimates is significant. While the $\ell_1$ errors of PMPL estimates show consistent decreasing trends in all four settings, those for the penalized logistic regression estimates are much higher. Moreover, as expected, the empirical variances of the $\ell_1$ and $\ell_2$ errors for the penalized logistic regression estimate are significantly larger than those for the PMPL estimate. These findings illustrate the effectiveness of the proposed method for modeling dependent network data for range of network models, encompassing different network topologies, such as community structure and degree distribution.  

 	\begin{figure}[H]




		\begin{minipage}{0.49\linewidth}
			\centering
			\includegraphics[width=7.25cm]{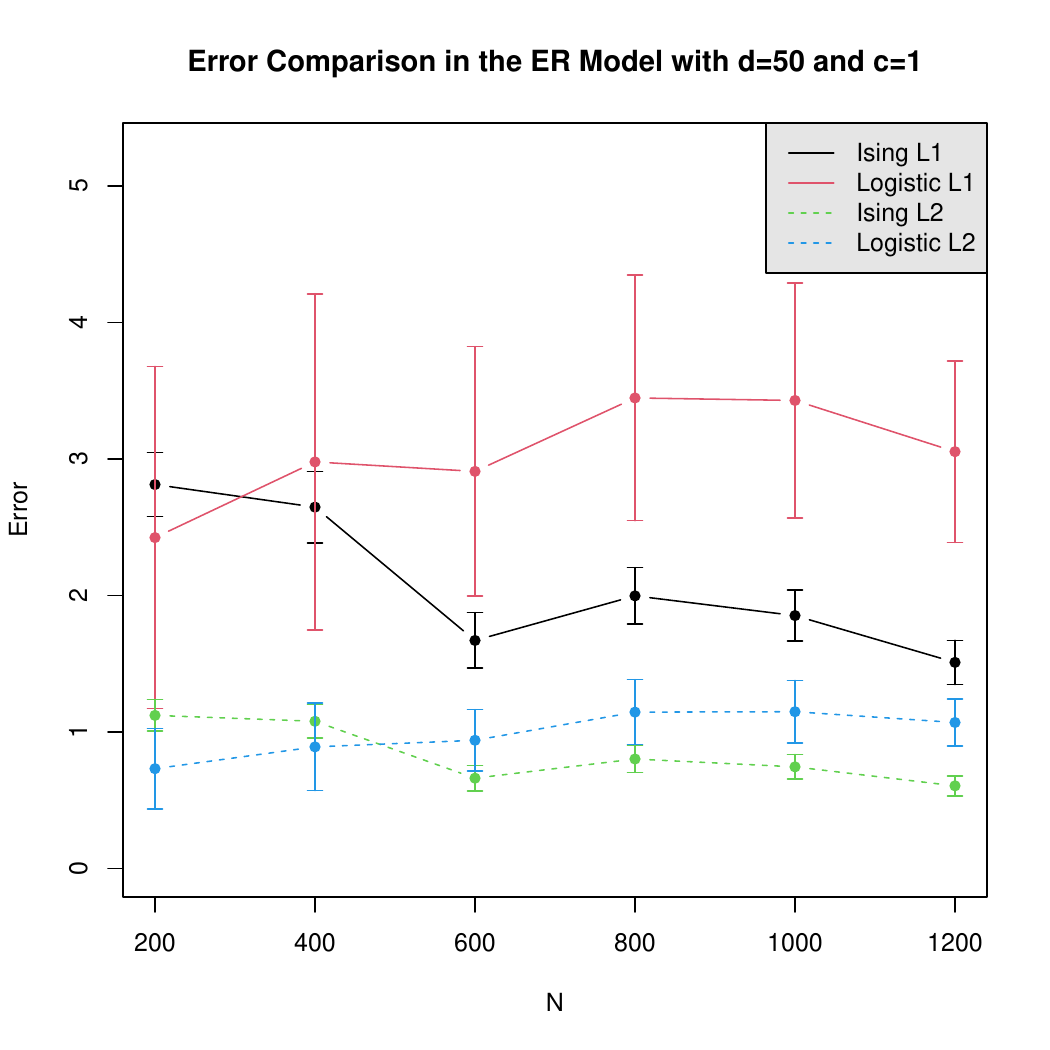}
			\caption*{ (a) }
		\end{minipage}
		\begin{minipage}{0.49\linewidth}
			\centering
			\includegraphics[width=7.25cm]{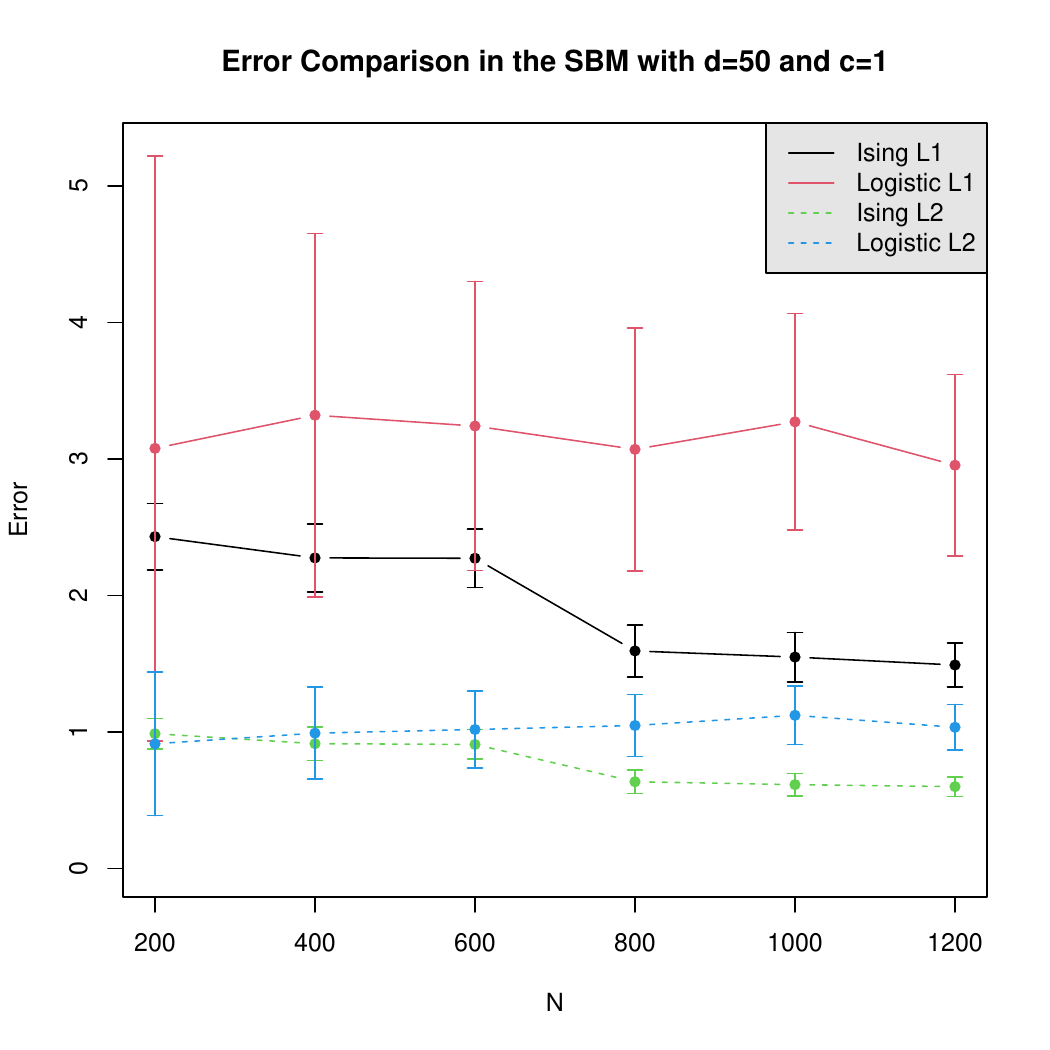}
			\caption*{ (b) }
		\end{minipage}
        		\begin{minipage}{0.49\linewidth}
			\centering
			\includegraphics[width=7.25cm]{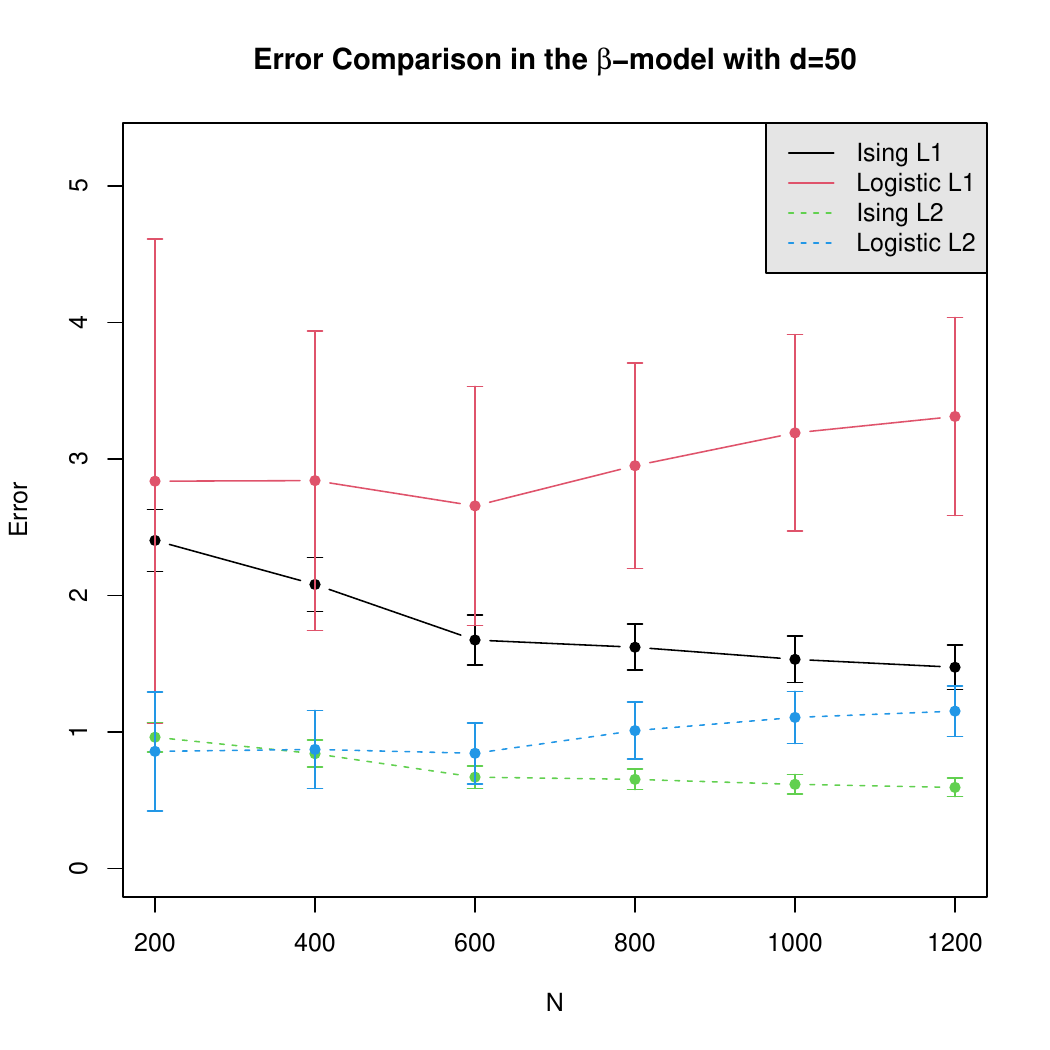}
			\caption*{ (c) }
		\end{minipage}
        \begin{minipage}{0.49\linewidth}
			\centering
			\includegraphics[width=7.25cm]{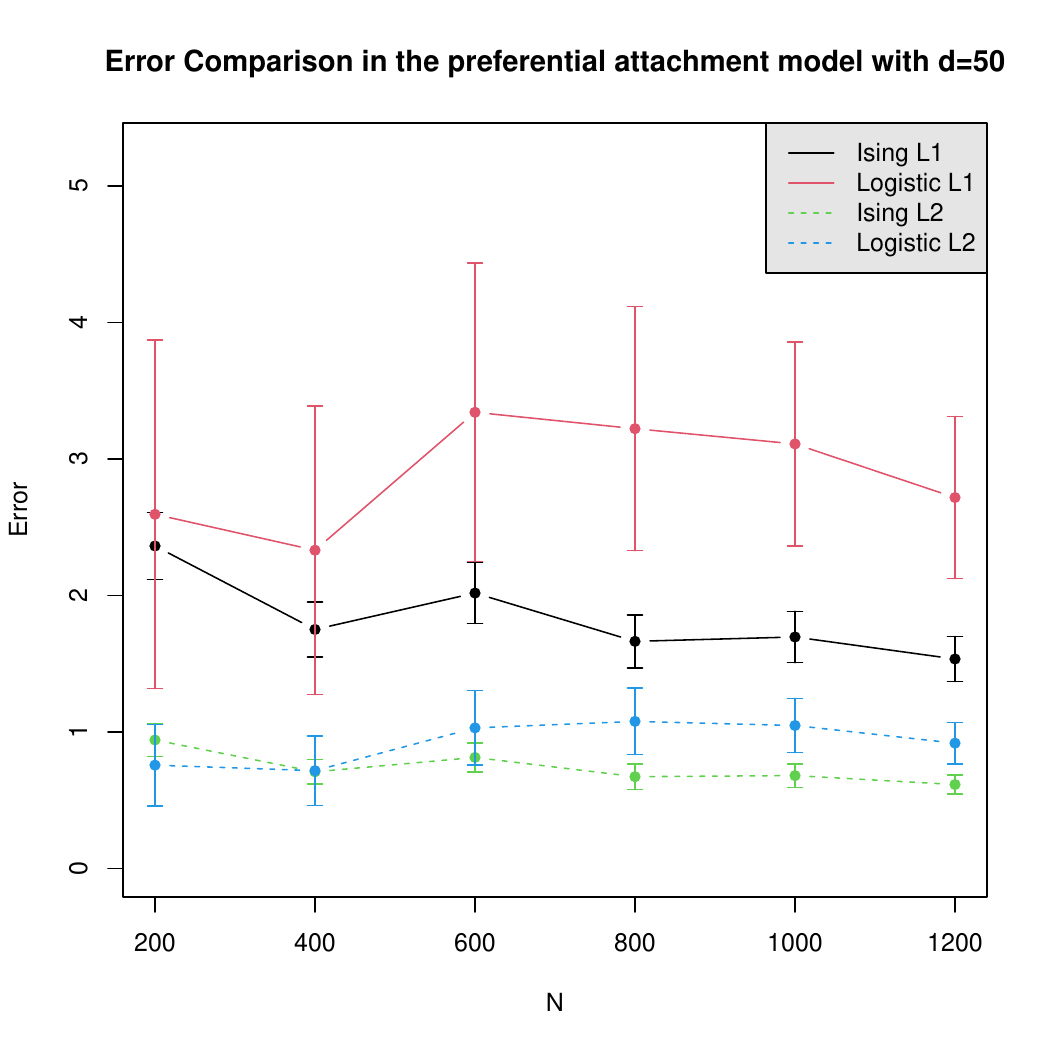}
			\caption*{ (d) }
		\end{minipage}
		\caption{\small{ Estimation errors of the PMPL and the penalized logistic regression estimates in the (a) Erd\H os-R\'enyi model and (b) the stochastic block model, (c) the $\beta$-model, and (d)  the preferential attachment model. }}
		\label{fig:L12}
	\end{figure}
 
We also investigate how the PMPL estimate performs with respect to the density of the network. To this end, we consider the Erd\H os-R\'enyi random graph $G(N, c/N)$, with $N=600$, and vary $c$. Figure \ref{fig:decreasing_varying_c} shows the estimation errors as $c$ increases, with  dependence parameter (a) $\beta=0.15$ and (b) $\beta=0.3$ in the respective sub-plots. As expected, the error curves for the PMPL estimates are generally better than those for the penalized logistic regression estimates. Moreover, the estimation errors are relatively small to begin with (when $c$ is small), but starts to show an increasing trend with $c$ after a while. This is expected because as the network density increases the rate of convergence slows down and, as a result, consistent estimation becomes harder (recall the discussion in Section \ref{sec:GN}). 

\begin{figure}[H]




		\begin{minipage}{0.49\linewidth}
			\centering
			\includegraphics[width=7.25cm]{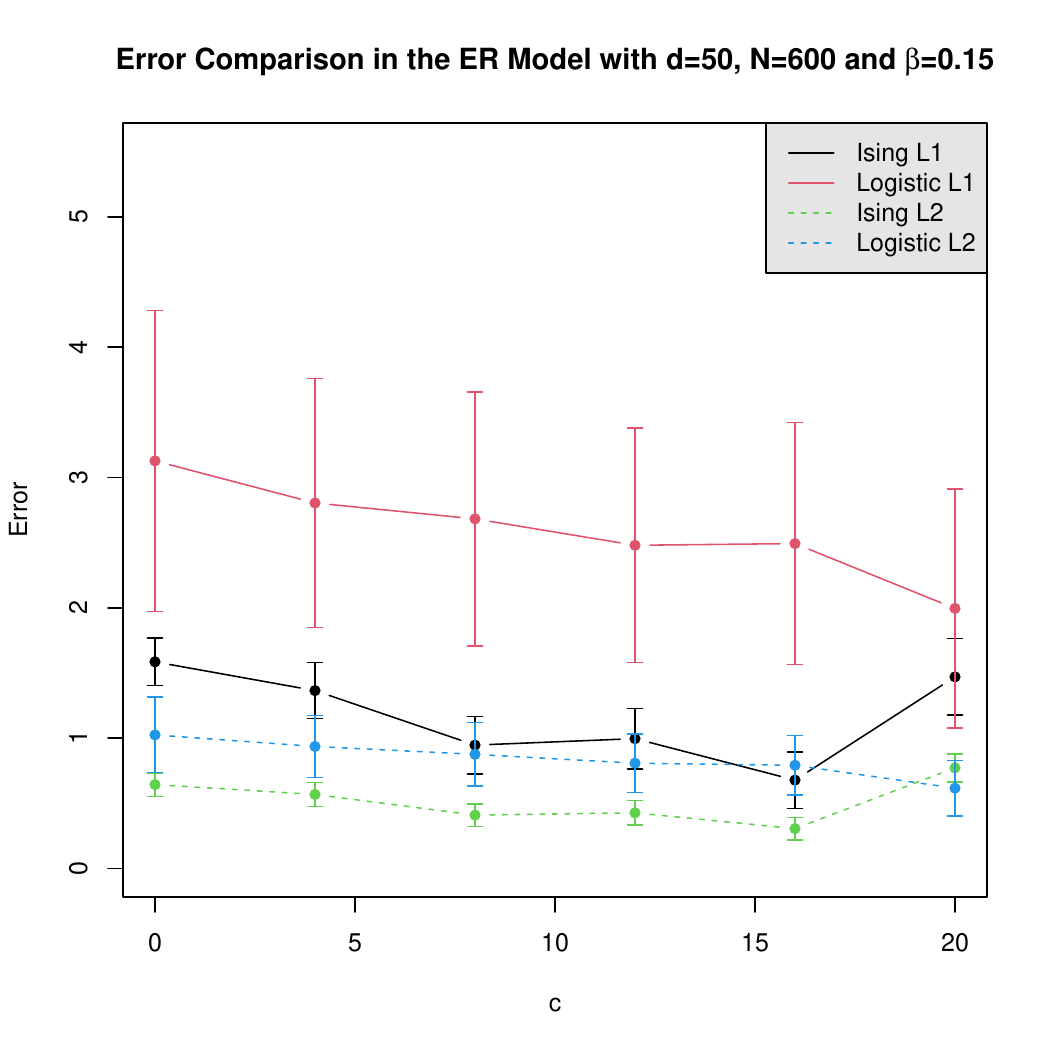}
			\caption*{ (a) }
		\end{minipage}
		\begin{minipage}{0.49\linewidth}
			\centering
			\includegraphics[width=7.25cm]{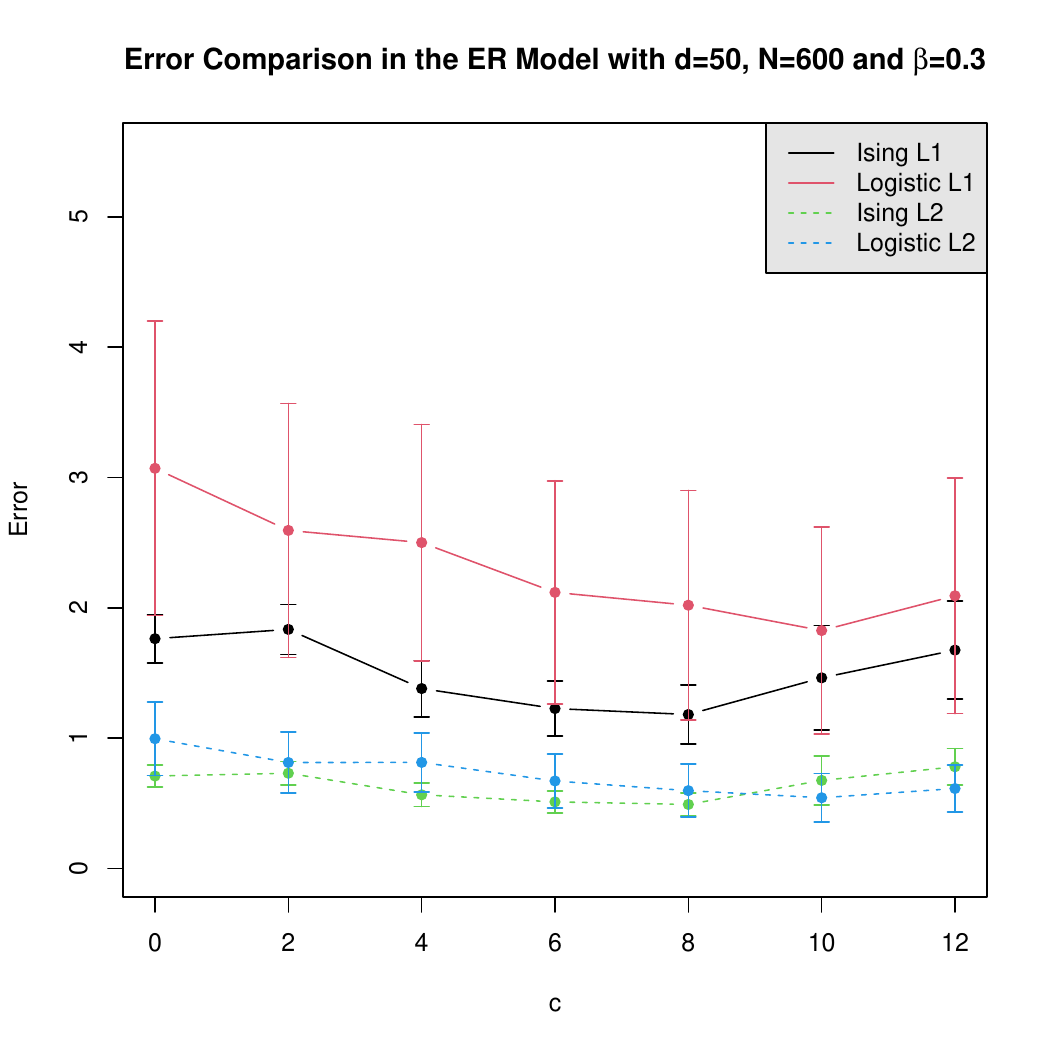}
			\caption*{ (b) }
		\end{minipage}
		\caption{\small{ Estimation errors of the PMPL and the penalized logistic regression estimates in the Erd\H os-R\'enyi model $G(N, c/N)$, with $N=600$, as $c$ varies, where (a) $\beta=0.15$ and (b) $\beta=0.3$. }}
		\label{fig:decreasing_varying_c}
	\end{figure}

\section{Application to Spatial Transcriptomics} 
\label{sec:st} 



%


In this section, we illustrate how the proposed model can be useful in selecting relevant genes in spatial gene expression data. As mentioned in the Introduction, spatial transcriptomics is a new direction in molecular biology where, in addition to measuring the gene expression levels of individual cells, one also has information about the spatial location of the cells \citep{t1,t2,t3,t4}. To understand how the spatial location of a cell affects its phenotype, it is natural to consider a model as in \eqref{eq:logisticmodel} with a nearest neighbor graph of the cell locations as the underlying network. 

We consider the Visium spatial gene expression data set for human breast cancer (see \url{https://www.10xgenomics.com/spatial-transcriptomics} for details about the spatial capture technology) available in the Python package {\tt scanpy}. The data set is available at \url{https://support.10xgenomics.com/spatial-gene-expression/datasets} and can be loaded using the Python command: 
{\tt \begin{verbatim} 
scanpy.datasets.visium_sge(sample_id=`V1_Breast_Cancer_Block_A_Section_1')
\end{verbatim} } 

The data consists of $36601$ genes and $3798$ cells along with their spatial locations. To obtain the cell labels, we first filter out the top 50 highly variable genes, that is, the genes whose expression variance is within the top $50$ among all genes. Subsequently, we cluster the cells based on the expression levels of these 50 gene into 2 types (clusters) using the Leiden algorithm \citep{cellclustering}. The output of the clustering algorithm visualized using the Python command {\tt sc.pl.spatial} is shown in Figure \ref{fig:clustering} (a). 
%
%
Using the cell labels obtained as above and the first 100 highly variable genes as the covariates, we then fit the model \eqref{eq:logisticmodel} with the 1-nearest neighbor graph of the spatial location of the cells as the underlying network, using the PMPL method. The optimal $\lambda$ is chosen using the BIC criterion. 
 
%
%


The PMPL method with the BIC chosen regularization parameter, selects 6 genes among the top 100 highly variable genes. Among the selected ones, four of them are actually in the top $50$ highly variable gene set obtained in the first filtering step. These genes are shown in Table \ref{tab:gene_selected}. Next, we re-cluster the cells based on only the  6 selected genes (see Figure \ref{fig:clustering} (b)). 
Interestingly, just using the $6$ selected genes we can recover the clustering result obtained with the top 50 variable genes with high accuracy. 
This illustrates how incorporating spatial information can significantly reduce dimensionality for clustering single cell data and the usefulness of our method in selecting relevant genes.

\begin{table}[!ht]
	\begin{center} 
 \small 
		\begin{tabular}{ |c|c|c|c|c| } 
		\hline
		Selected genes among top $100$ & Estimated coefficients & Selected genes among top $50$ \\ 
		\hline 
		S100A9 & 0.0087 & S100A9 \\ 
		CPB1 & -0.0330  & CPB1 \\ 
		SPP1 & 0.0123 & SPP1 \\ 
		CRISP3 & 0.1465 & CRISP3 \\ 
		\hline
		SLITRK6 & 0.1276 &  \\ 
		IGLC2 & -0.0392 &  \\ 
		\hline
		\end{tabular}
		\caption{\label{tab:gene_selected}Names of the selected genes and the estimates of the corresponding regression coefficients. The estimate of $\beta$ is $\hat \beta = 0.1203$. }
		\end{center}		
\end{table}

\normalsize 

\begin{figure}[H]
	\centering 
			\begin{minipage}{0.45\linewidth}
		\centering
		\includegraphics[width=\textwidth]{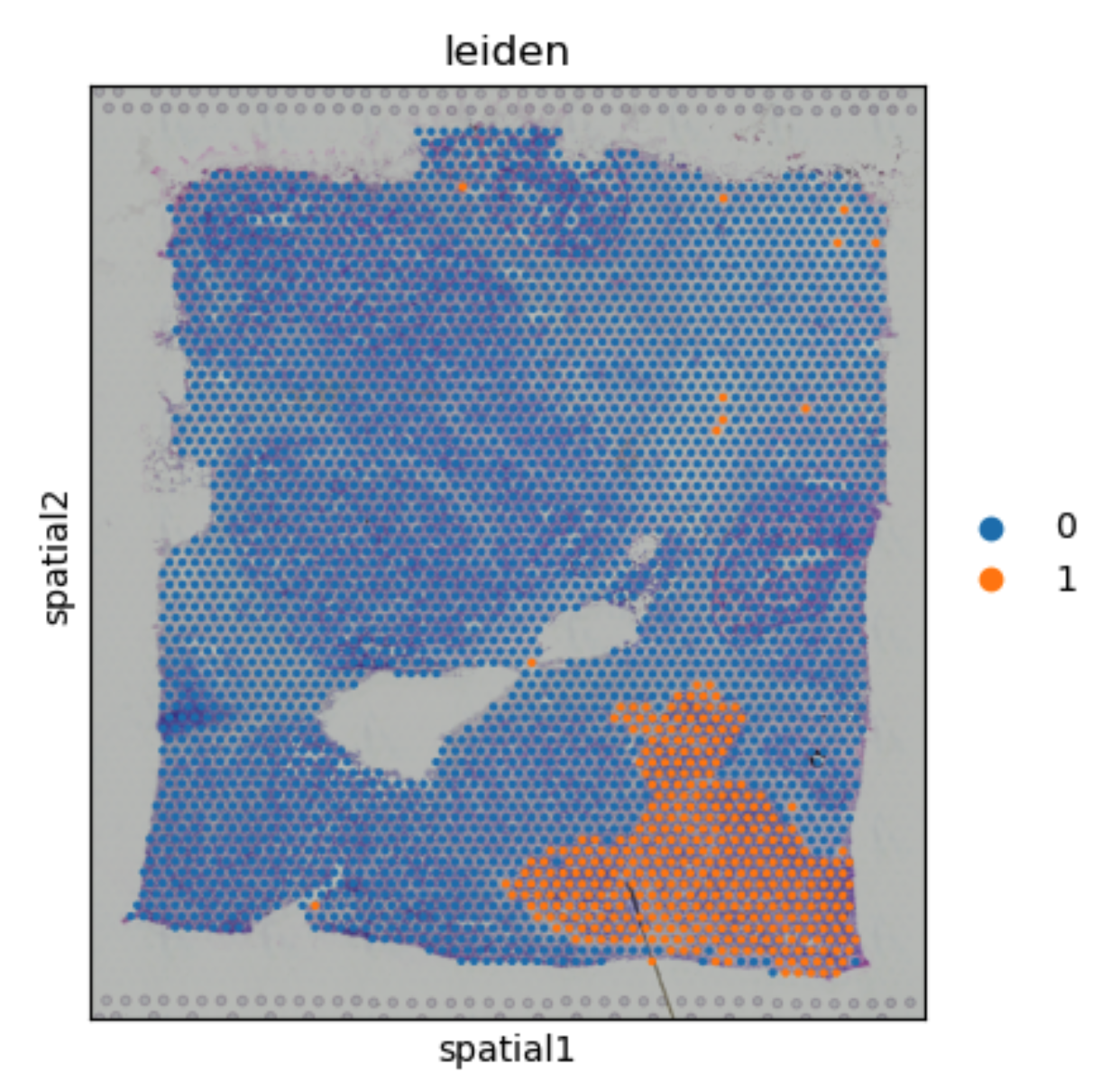} \\
		\small{(a)} 
		\end{minipage}
				\begin{minipage}{0.45\linewidth}
		\centering
		\includegraphics[width=\textwidth]{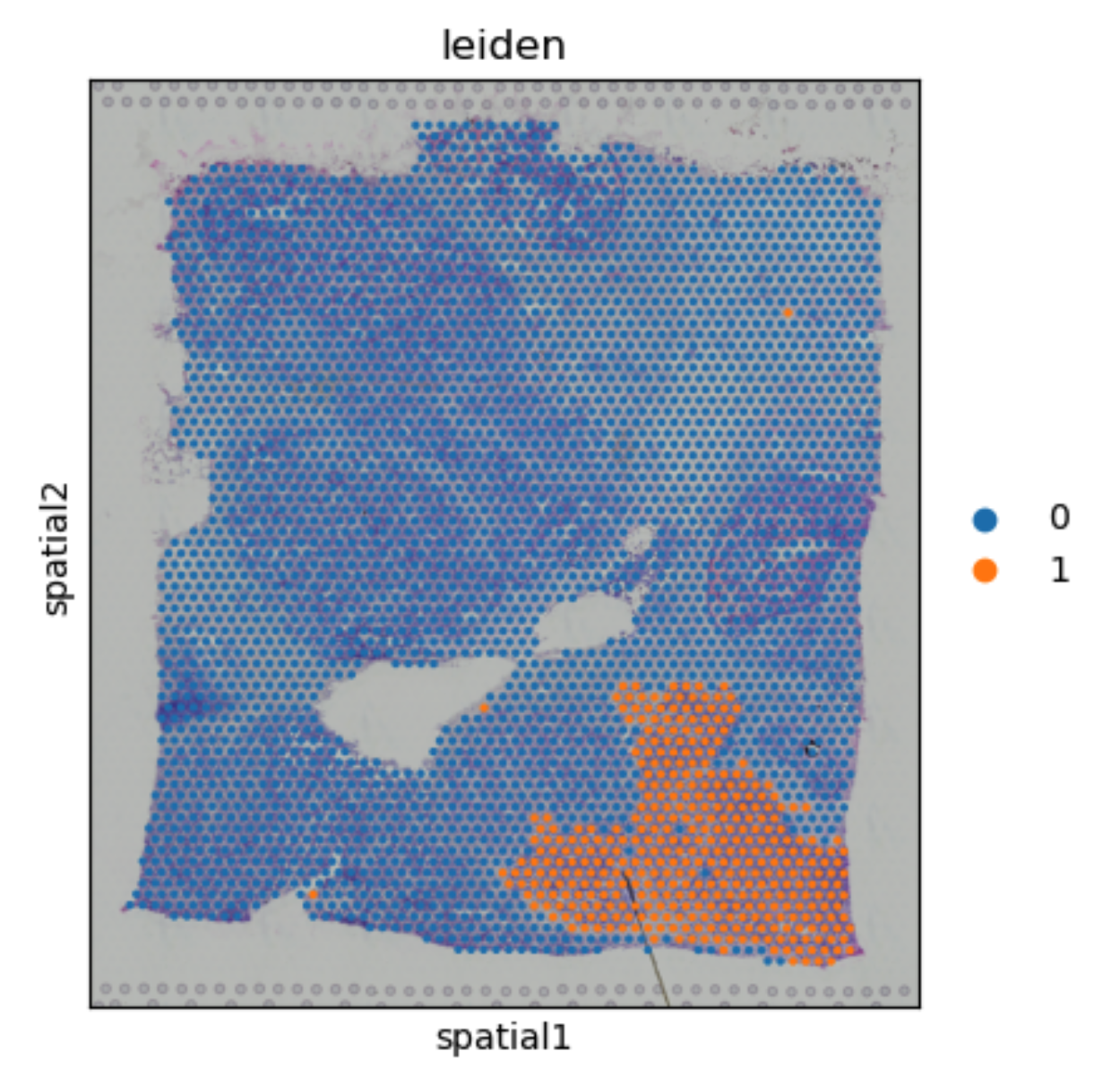} \\ 
		\small{(b)} 
		\end{minipage}
	\caption{\small{Clustering results using the Leiden algorithm: (a) with the top $50$ highly variable genes, (b) with the $6$ selected genes.} }
	\label{fig:clustering} 
\end{figure}

To capture the spatial dependence one can, more generally, consider the $K$-nearest neighbor graph (instead of the 1-nearest neighbor graph as above) of the spatial locations of the cells in the model \eqref{eq:logisticmodel}. To understand the sensitivity of the PMPL method on the choice of the number of nearest neighbors, we repeat the experiment with $K=1$, $K=2$, and $K=3$. The genes selected by the PMPL method and the estimates of the corresponding regression coefficients for each of these settings are shown in Table \ref{tab:varying_nearest_neighbor_selection}. It turns out that for $K=1$ and $K=2$ the genes selected are the same, and for $K=3$ the genes selected match except one (the gene SPP1 is no longer selected). 
This shows that the PMPL method is quite robust to choice of the underlying nearest-neighbor graph as long as $K$ is not too large. While one can incorporate more distant spatial dependencies by increasing $K$, 
this makes the graph denser and, as a result, the rate of estimation becomes slower (as shown in Section \ref{sec:GN}). In practice, especially for spatial problems, where the dependence often decreases with distance, choosing a small value of $K$ should suffice.

\begin{table}[!ht]
	\begin{center}
 \small
		\begin{tabular}{ |c|c|c|c|c| } 
		\hline
		 & Selected genes among top $100$ & Estimated coefficients & Selected genes among top $50$ \\ 
		\hline 
		\multirow{6}{*}{$K = 1$}& S100A9 & 0.0087 & S100A9 \\ 
		 & CPB1 & -0.0330  & CPB1 \\ 
		& SPP1 & 0.0123 &  SPP1 \\ 
		& CRISP3 & 0.1465 &  CRISP3 \\ 
		& SLITRK6 & 0.1276 &  \\ 
		& IGLC2 & -0.0392  &  \\ 
		\hline
		\multirow{6}{*}{$K = 2$} & S100A9 & 0.0047 & S100A9 \\ 
		 & CPB1 & -0.0250  & CPB1 \\ 
		& SPP1 & 0.0037 & SPP1 \\ 
		& CRISP3 & 0.1121 &  CRISP3 \\ 
		& SLITRK6 & 0.1131 &  \\ 
		& IGLC2  & -0.0317 &  \\ 
		\hline
		\multirow{5}{*}{$K = 3$}  & S100A9 & 0.0034 & S100A9 \\ 
		 & CPB1 & -0.0177  & CPB1 \\ 
		& CRISP3 & 0.0805 & CRISP3 \\ 
		& SLITRK6 & 0.0845 &  \\ 
		& IGLC2 & -0.0241 &  \\ 
		\hline
		\end{tabular}
		\caption{\label{tab:varying_nearest_neighbor_selection} Names of the selected genes and the estimates of the corresponding regression coefficients for the $K$-nearest-neighbor graph, with $K=1$, $K=2$, and $K=3$. The estimates of $\beta$ are $\hat\beta=0.1203$, $\hat \beta= 0.2434$, and $\hat \beta=0.2870$  for $K=1$, $K=2$, and $K=3$, respectively.}
		\end{center}		
\end{table}

\normalsize

\section{Conclusion} 	
 
Understanding the effect of dependence in high-dimensional inference tasks for non-Gaussian models is an emerging research direction. In this paper, we develop a framework for efficient parameter estimation in a model for dependent network data with binary outcomes and high-dimensional covariates. The model combines the classical high-dimensional logistic regression with the Ising model from statistical physics to simultaneously capture dependence from the underlying network  and the effect of high-dimensional covariates. This dependence makes the model different and the analysis more challenging compared to existing results based on independent samples. In the this paper we develop an efficient algorithm for jointly estimating the effect of dependence and the high-dimensional regressions parameters using a penalized maximum pseudo-likelihood (PMPL) method and derive its rate of consistency. To understand which of the covariates have an effect on the outcome under the presence of network dependence, we also consider the problem of estimation given a fixed (known) level of dependence. Towards this, we show that using the PMPL method the regression parameters can be estimated at the classical high-dimensional rate, despite the presence of dependence, in the entire high-dimensional regime. 
 We expect the model to be broadly useful in network econometrics and spatial statistics for understanding dependent binary data with an underlying network geometry. 
As an application, we apply the proposed model to select genes in spatial transcriptomics data.

Various questions remain and future directions emerge. Theoretically, it would be interesting to see if the conditions for joint estimation can be relaxed.  Computationally, it would be interesting to explore more efficient sampling schemes for Ising models with covariates. Incorporating dependence in other generalized linear models and high-dimensional distributions, through the lens of the Ising and more general graphical models, is another interesting direction for future research.

	\section{Acknowledgement} 
	Bhaswar B. Bhattacharya thanks Rajarshi Mukherjee and Nancy R. Zhang for many helpful discussions. The authors also thank the anonymous referees for their  insightful comments which improved the quality and the presentation of the paper.  The authors are also grateful to Sagnik Nandy for his help with the data analysis. Bhaswar B. Bhattacharya was supported by NSF CAREER grant DMS 2046393 and a Sloan Research Fellowship. Somabha Mukherjee was supported by the National University of Singapore start-up grant WBS A0008523-00-00 and the FoS Tier 1 grant WBS A-8001449-00-00. George Michailidis was supported by NSF grant DMS 2334735.

	\appendix

	\small

	\vskip 0.2in
	%
	%

	\appendix

	\normalsize 

\bigskip\bigskip\noindent
	\appendix
	
	\section{Proof of Theorem \ref{thm:estimate}}\label{prfinalmain} 
	
Theorem \ref{thm:estimate} is a consequence of the following more general result which provides rates of consistency for the PMPL estimate in terms of $\|\bm A\|_F^2$.

	\begin{proposition}\label{ppn:estimateF} 
		 Suppose that Assumptions \ref{asm:a1}, \ref{asm:a2}, and \ref{asm:a3} hold.  Then, there exists a constant $\delta > 0$ such that by choosing $\lambda := \delta \sqrt{\log(d+1)/N}$ in the objective function in \eqref{eq:pseudolikelihood} we have,
		\begin{align*}
			\|\hat{\bg} - \bg \|_2 = O_s\left(\sqrt{\frac{\log d}{\|\bm A\|_F^4/N}}\right)  \quad \text{ and } \quad \|\hat{\bg} - \bg \|_1  = O_s\left(\sqrt{\frac{\log d}{\|\bm A\|_F^4/N}}\right) ,
		\end{align*} 
		with probability $1-o(1)$, as $N\rightarrow \infty$ and $d \rightarrow \infty$ such that $d = o(\|\bm A\|_F^2)$ and $\log d = o(\|\bm A\|_F^4/N)$. 
	\end{proposition}

Note that when $\liminf_{N \rightarrow \infty} \frac{1}{N} \|\bm A\|_F^2 > 0$, then rates in Theorem \ref{ppn:estimateF} is an immediate consequence of Proposition \ref{ppn:estimateF}. 

%
%

The rest of this section is devoted the proof of Proposition \ref{ppn:estimateF}. To this end, recall from \eqref{eq:betatheta} that our PMPL estimator is defined as:
	\begin{equation}\label{eq:pseudolikelihood}
		(\hat{\beta},\hat{\bh}^\top) :=  \underset{({\beta},{\bh}^\top) \in \mathbb{R}^{d+1}}{\mathrm{argmin}} L_N({\beta},{\bh}) + \lambda \norm{\bh}_1 
	\end{equation}
	where $\lambda > 0$ is a tuning parameter and
	\begin{eqnarray*}
		L_N({\beta},{\bh}) = \frac{1}{N} \sum_{i=1}^N \left[\log \cosh \left({\beta} \sum_{j=1}^N A_{ij} X_j + {\bh}^\top \bm Z_i\right) - X_i \left({\beta} \sum_{j=1}^N A_{ij} X_j + {\bh}^\top \bm Z_i\right)\right] , 
	\end{eqnarray*}
	is as defined in \eqref{eq:LNbetatheta} (where we have dropped the additive factor of $\log 2$). To begin with, note that since Assumption \ref{asm:a1} holds, by scaling the interaction matrix and the covariate vectors by $\|\bm A\|_\infty$ we can assume without loss of generality, 
	\begin{align}\label{eq:Anorm}
		\sup_{N\ge 1} \|\bm A\|_\infty \le 1.
	\end{align} 
	The first step towards the proof of Theorem  \ref{thm:estimate} is to establish the concentration of the pseudo-likelihood gradient vector $\nabla L_N(\hat{\bg})$ in the $\ell_\infty$ norm. This is formalized in the following lemma which is proved in Section \ref{sec:devondpf}.

	\begin{lemma}[Concentration of the gradient]\label{lm:gradient} For $\hat \bg$ and any $\lambda > 0$, 
		\begin{align}\label{eq:gradientestimate}
			\norm{\nabla L_N(\hat{\bg})}_\infty \leq \lambda.
		\end{align} 
		Moreover, there exists  $\delta > 0$ such that with  $\lambda := \delta \sqrt{\log(d+1)/N}$  the following holds: 
		\begin{align}\label{eq:gradient}
			\bp\left(\norm{\nabla L_N(\bg)}_\infty > \frac{\lambda}{2}\right) = o(1), 
		\end{align} 
		where the $o(1)$-term goes to infinity as $d \rightarrow \infty$. 	
	\end{lemma}
	
	The next lemma shows that the pseudo-likelihood function is strongly concave with high probability. The proof of this lemma is given in Section \ref{poshes}.  
	
	\begin{lemma}[Strong concavity of pseudo-likelihood] \label{rsccond} Suppose the assumptions of Theorem \ref{thm:estimate} hold. Then, there exists a constant $\kappa  := \kappa (s, M, \beta, \Theta) > 0$, such that
		$$L_N(\hat{\bg}) - L_N(\bg) - \nabla L_N(\bg)^\top (\hat{\bg}-\bg) \geq \kappa   \frac{\|\bm A\|_F^2 \norm{ \hat{\bg}-\bg }_2^2}{N},$$
		with probability $1-o(1)$. 
	\end{lemma}

	The proof of Theorem \ref{thm:estimate} can now be easily completed using the above lemmas. Towards this define:
	$$S := \{1\le i\le d: \theta_i \neq 0\}.$$ Moreover, for any vector $\bm a \in \mathbb{R}^{p}$ and any set $Q \subseteq \{1,\ldots,p\}$, we denote by $\bm a_Q$ the vector $(a_i)_{i\in Q}$. Now, for the constant $\kappa $ as in Lemma \ref{rsccond}, consider the event 
	\begin{align*}
	\cE_{N} := \Big \{ \bm X \in \cC_N: & \norm{\nabla L_N(\bg)}_\infty \leq \tfrac{\lambda}{2} \nonumber \\ 
	& \text{ and } L_N(\hat{\bg}) - L_N(\bg) - \nabla L_N(\bg)^\top (\hat{\bg}-\bg) \geq \kappa   \frac{\|\bm A\|_F^2 \norm{ \hat{\bg}-\bg }_2^2}{N} \Big \}. 
	\end{align*}
	Clearly, from Lemma \ref{lm:gradient} and Lemma \ref{rsccond}, $\mathbb P(\cE_N^c) = o(1)$.

	Next, suppose $\bm X \in \cE_N$. From the definition of $\hat{\bg}$ it follows that 
	\begin{equation}\label{lneqn}
		L_N(\hat{\bg}) + \lambda \|\hat{\bh}\|_1 \leq L_N(\bg) + \lambda \|\bh\|_1.
	\end{equation}
	Hence,
	\begin{align}\label{firsteq}
		\lambda(\|\bh\|_1 - \|\hat{\bh}\|_1) & \geq L_N(\hat{\bg}) - L_N(\bg) \nonumber \\ 
		& = \nabla L_N(\bg)^\top (\hat{\bg}-\bg) + (L_N(\hat{\bg}) - L_N(\bg) - \nabla L_N(\bg)^\top (\hat{\bg}-\bg)) \nonumber \\ 
		& \geq -\|\nabla L_N(\bg)\|_\infty \|\hat{\bg}-\bg\|_1 + (L_N(\hat{\bg}) - L_N(\bg) - \nabla L_N(\bg)^\top (\hat{\bg}-\bg)) \nonumber \\ 
		& \geq -\frac{\lambda \|\hat{\bg}-\bg\|_1 }{2}  + (L_N(\hat{\bg}) - L_N(\bg) - \nabla L_N(\bg)^\top (\hat{\bg}-\bg)), 
	\end{align}
	where the last step uses $\norm{\nabla L_N(\bg)}_\infty \leq \tfrac{\lambda}{2}$, for $\bm X \in \cE_N$. Next, note that 	
	\begin{align}
		\|\hat{\bh}\|_1  = \|\bh_{S} + (\hat{\bh}-\bh)_S\|_1 + \|(\hat{\bh}-\bh)_{S^c}\|_1 
		& \ge \|\bh_{S}\|_1 - \|(\hat{\bh}-\bh)_S\|_1 + \|(\hat{\bh}-\bh)_{S^c}\|_1 \nonumber \\ 
		& = \|\bh\|_1 - \|(\hat{\bh}-\bh)_S\|_1 + \|(\hat{\bh}-\bh)_{S^c}\|_1 . \nonumber 
	\end{align} 		
	This implies,  
	\begin{equation}\label{secstep2}
		\|(\hat{\bh}-\bh)_S\|_1 - \|(\hat{\bh}-\bh)_{S^c}\|_1 \ge  \|\bh\|_1 - \|\hat{\bh}\|_1   .  
	\end{equation} 
	Combining \eqref{firsteq} and \eqref{secstep2} it follows that
	\begin{align}\label{mainbound}
		\lambda\left(\|(\hat{\bh}-\bh)_S\|_1 - \|(\hat{\bh}-\bh)_{S^c}\|_1 + \frac{\|\hat{\bg}-\bg\|_1}{2}  \right)  & \geq L_N(\hat{\bg}) - L_N(\bg) - \nabla L_N(\bg)^\top (\hat{\bg}-\bg) \nonumber \\  
		& \geq  \kappa   \frac{\|\bm A\|_F^2 \norm{ \hat{\bg}-\bg }_2^2}{N} , 
	\end{align} 
	where the last inequality uses $L_N(\hat{\bg}) - L_N(\bg) - \nabla L_N(\bg)^\top (\hat{\bg}-\bg) \geq \kappa   \frac{\|\bm A\|_F^2 \norm{ \hat{\bg}-\bg }_2^2}{N}$, for $\bm X \in \cE_N$. Using $\|\hat{\bg}-\bg\|_1 =  \|(\hat{\bh}-\bh)_S\|_1 + \|(\hat{\bh}-\bh)_{S^c}\|_1 + | \hat \beta - \beta | $ in the LHS of \eqref{mainbound} now gives, 
	\begin{align*}
		\|\hat{\bg}-\bg\|_2^2  & \leq \frac{\lambda N}{\kappa \|\bm A\|_F^2} \left(\frac{3 \|(\hat{\bh}-\bh)_S\|_1 }{2} - \frac{\|(\hat{\bh}-\bh)_{S^c}\|_1}{2} + \frac{ | \hat \beta - \beta | }{2} \right) \nonumber \\ 
		& \lesssim_{\kappa } \frac{\lambda N}{\|\bm A\|_F^2} \left ( \|(\hat{\bh}-\bh)_S\|_1 + | \hat \beta - \beta | \right ) \nonumber \\ 
		& \lesssim_{\kappa}  \frac{\lambda N}{\|\bm A\|_F^2}  \sqrt{s+1} \sqrt{ \sum_{i \in S} ( \hat \theta_i - \theta_i )^2 + (\hat \beta - \beta)^2 } \nonumber \\ 
		& \lesssim \frac{\lambda N}{\|\bm A\|_F^2} \sqrt{s} \|\hat{\bg}-\bg\|_2 . 
	\end{align*}
	This implies, for $\bm X \in \cE_N$, 
	\begin{align*} 
		||\hat{\bg}-\bg||_2 = O_{\kappa, \delta}\left(\frac{N}{\|\bm A\|_F^2}\sqrt{\frac{s \log d}{N}} \right). 
	\end{align*} 
	This completes the proof of the $\ell_2$ error rate in Theorem \ref{thm:estimate}, since $\mathbb P(\cE_N) = 1 - o(1)$. The bound on the $\ell_1$ error $||\hat{\bg}-\bg||_1$ is shown in Lemma \ref{lm:L1}.

	\subsection{Proof of Lemma \ref{lm:gradient}} 
	\label{sec:devondpf}

	First, we establish that $||\nabla L_N(\hat \bg)||_\infty \leq \lambda$. Fix $1\le i\le d$ and define the univariate function:
	$$f(x) := L_N(\hat \beta, \hat \theta_1, \ldots,\hat{\theta}_{i-1},x,\hat{\theta}_{i+1},\ldots,\hat{\theta}_{d}).$$
	Note that $f'(\hat{\theta}_i) = \frac{\partial}{\partial \underline{\theta}_i } \nabla  L_N(\underline{\bg})\mid_{\underline{\bg}= \hat{\bg}}$. Now, by the definition of $\hat{\bg}$  we have, $f(\hat{\theta}_i) + \lambda |\hat{\theta}_i| \le f(x) + \lambda |x|$, which implies, 
	$$f(x)-f(\hat{\theta}_i) \ge \lambda(|\hat{\theta}_i| - |x|).$$ Then consider the following cases: 
	
	\begin{itemize}
		
		\item $\hat{\theta}_i > 0$: Then, for all $x > \hat{\theta}_i$, $$\frac{f(x)-f(\hat{\theta}_i)}{x-\hat{\theta}_i} \ge \lambda \frac{|\hat{\theta}_i|-|x|}{x-\hat{\theta}_i} = -\lambda. $$ Similarly, for all $0<x<\hat{\theta}_i$,	$\frac{f(x)-f(\hat{\theta}_i)}{x-\hat{\theta}_i} \le \lambda \frac{|\hat{\theta}_i|-|x|}{x-\hat{\theta}_i} = -\lambda$. This implies,  $f'(\hat{\theta}_i) = -\lambda$.
		
		\item $\hat{\theta}_i < 0$: Then, for all $0>x > \hat{\theta}_i$, 
		$$\frac{f(x)-f(\hat{\theta}_i)}{x-\hat{\theta}_i} \ge \lambda \frac{|\hat{\theta}_i|-|x|}{x-\hat{\theta}_i} = \lambda.$$
		Similarly, $x<\hat{\theta}_i$, $\frac{f(x)-f(\hat{\theta}_i)}{x-\hat{\theta}_i} \le \lambda \frac{|\hat{\theta}_i|-|x|}{x-\hat{\theta}_i} = \lambda$. Hence, in this case,  $f'(\hat{\theta}_i) = \lambda$. 
		
		\item $\hat{\theta}_i = 0$: In this case,  for all $x > 0$, $$\frac{f(x) - f(0)}{x} \ge -\lambda \frac{|x|}{x} = -\lambda$$ and for all $x<0$, $\frac{f(x) - f(0)}{x} \le -\lambda \frac{|x|}{x} = \lambda$. Since $f'$ exists, this implies that $|f'(0)| \le \lambda$. 
	\end{itemize} 
	
	Next, define 
	$$g(x) := L_N(x, \hat \theta_1, \ldots,\hat{\theta}_{i-1}, \hat \theta_i, \hat{\theta}_{i+1},\ldots,\hat{\theta}_{d}).$$ 
	Note that $g'(\hat \beta) = \frac{\partial}{\partial \underline{ \beta } } \nabla  L_N(\underline{\bg})\mid_{\underline{\bg}= \hat{\bg}}$. By the definition of $\hat{\bg}$  we have, $g( \hat \beta ) \le g(x) $. This implies that $g'(\hat \beta) =0$.  Combining the above, it follows that 
	$$||\nabla L_N(\hat \bg)||_\infty = \max_{j \in [d]} |f'(\hat{\theta}_i)| \leq \lambda,$$ 
	completing the proof of \eqref{eq:gradientestimate}.

	Next, we establish the concentration of $\norm{\nabla L_N(\bg)}_\infty$ as in \eqref{eq:gradient}. For this step, we require the following definitions. For $1 \leq i \leq N$, denote  
	$$m_i(\bs) := \sum_{j=1}^N a_{ij} X_j.$$  
	Define functions $\phi_i : \cC_N \rightarrow \R$, for $1 \leq i \leq N$, as follows: 	
	\begin{align}\label{eq:function1}
		\phi_{i}(\bm x) := -\frac{1}{N}\left\{ m_i(\bm x) \left(X_i - \tanh(\beta m_i(\bm x) + \bh^\top \bm Z_i)\right) \right\}, 
	\end{align}  
	for $\bm x = (x_1, x_2, \ldots, x_n) \in \cC_N$. Similarly, define functions $\phi_{i, s}: \cC_N \rightarrow \R$,  for $1 \leq i \leq N$ and $1 \leq s \leq d$, as follows: 	
	\begin{align}\label{eq:function2}
		\phi_{i, s}(\bm x) := -\frac{1}{N} \left\{ Z_{i, s} \left(x_i - \tanh(\beta m_i(\bm x) + \bh^\top \bm Z_i)\right) \right\} .
	\end{align} 
	Note that $\nabla L_N = (\frac{\partial L_N}{\partial \beta}, \frac{\partial L_N}{\partial \theta_1}, \ldots, \frac{\partial L_N}{\partial \theta_d})^\top$ where 
	$$\frac{\partial L_N}{\partial \beta} = \sum_{i=1}^N \phi_{i}(\bm X) \quad \textrm{and}\quad \frac{\partial L_N}{\partial \theta_s} = \sum_{i=1}^N \phi_{i, s}(\bm X), \quad \text{ for } 1 \leq s \leq d.$$ 
	To establish the concentration of $\norm{\nabla L_N(\bg)}_\infty$, we use the conditioning trick from \citet{cd_ising_estimation}, which allows to reduce the model \eqref{eq:logisticmodel} to an Ising model in the Dobrushin regime (where the correlations are sufficiently weak and the model approximately behaves like a product measure), by conditioning on a subset of the nodes.  To describe this, we need the following definition: 
	\begin{definition}
		Suppose that $\boldsymbol{\sigma} \in \{-1,1\}^N$ is a sample from the Ising model:
		\begin{equation}\label{mgamma}
			\p_{\beta,\bm h}(\boldsymbol{\sigma}) ~\propto~ \exp\left( \boldsymbol{\sigma}^\top \bm D \boldsymbol \sigma + \sum_{i=1}^N h_i\sigma_i \right) , 
		\end{equation}
		where $\bm h= (h_1, h_2, \ldots, h_n)^\top \in \R^n$ and $\bm D$ is a symmetric matrix with zeros on the diagonal with $\sup_{N \geq 1} \|\bm D\|_\infty \le R$. Moreover, suppose that with probability $1$, 
		$$\min_{1\le i \le N} \mathrm{Var}(\sigma_i|\boldsymbol{\sigma}_{-i}) \ge \Upsilon,$$
		for some $\Upsilon \geq 0$. Then, the model \eqref{mgamma} is referred to as an $(R,\Upsilon)$-Ising model.
	\end{definition}
	
	Recently, \citet{cd_ising_estimation} developed a technique for reducing an  $(R,\Upsilon)$-Ising model to an  $(\eta,\Upsilon)$-Ising model,  for $0 < \eta < R$, by conditioning on a subset of vertices. As a consequence, by choosing $\eta$ one can ensure that the conditional model is in the Dobrushin high-temperature regime.  Although the Ising model  studied in \citet{cd_ising_estimation} is different, the same proof extends to our model \eqref{eq:logisticmodel} as well. We formalize this in the following lemma, which is  proved in Appendix \ref{sec:combinatorial1pf}. 
	
	\begin{lemma}\label{combinatorial1}
		Fix $R > 0$ and $\eta \in (0,R)$. Let $\bs\in \{-1,1\}^N$ be a sample from an $(R,\Upsilon)$- Ising model. Then there exist subsets $I_1,\ldots,I_\ell\subseteq [N]$ with $\ell \lesssim R^2 \log N/\eta^2$ such that:
		\begin{enumerate}
			\item For all $1\le i \le N$,
			$$|\{j \in [\ell]: i \in I_j \}| = \lceil \eta \ell / 8R\rceil$$
			\item For all $1 \le j \le \ell$, the conditional distribution of $\bm X_{I_j}$ given $\bm X_{I_j^c} := (X_u)_{u\in [N]\setminus I_j}$ is an $(\eta,\Upsilon)$- Ising model.
		\end{enumerate}
		Furthermore, for any non-negative vector $\bm a \in \mathbb{R}^N$, there exists $j \in \ell$ such that
		\begin{align}\label{eq:a}
			\sum_{i \in I_j} a_i \ge \frac{\eta}{8R}\sum_{i=1}^N a_i.\
		\end{align}
	\end{lemma}

	We will apply the above result to our model \eqref{eq:logisticmodel}. Towards this, set $\bm D = \beta \bm A$ and $h_i=\bm \theta^\top \bm Z_i$, for $1 \leq i \leq N$ in \eqref{mgamma}. Under this parametrization, \eqref{eq:logisticmodel} is an $(R,\Upsilon)$-Ising model as shown below:

	\begin{lemma}\label{eq:RU} The model \eqref{eq:logisticmodel} is a $(R, \Upsilon)$-Ising model with $R=|\beta|$ and any $\Upsilon = e^{-4\Theta M s - 4 |\beta| }$. 
	\end{lemma}
	
	\begin{proof} Note that for every $j \in [N]$, 
		$$\mathrm{Var}(X_j|\bm X_{[N]\backslash\{j\}}) = 4p(1-p),$$ 
		where $p= \p(X_j =1| \bm X_{[N]\backslash\{j\}} )$. Now, denoting the elements of of the matrix $\bm D$ as $(d_{ij})_{1 \leq i, j \leq N}$, note that  
		$$p = \frac{\exp\left(h_j + \sum_{v \in [N]\setminus \{j\}} d_{jv} X_v\right)}{2 \cosh\left( h_j +  \sum_{v \in [N] \setminus \{j\}} d_{jv}X_v\right)}.$$ 
		Then using the inequality $\frac{e^x}{2\cosh(x)} \ge \frac{1}{2} e^{-2|x|}$  gives, 
		\begin{equation}\label{minp}
			\min\{p, 1-p\} \ge \frac{1}{2} \exp\left(-2\left|h_j + \sum_{v \in [N]\setminus \{j\}} d_{jv}X_v\right|\right) . 
		\end{equation}
		Next, using
		$$\left|h_j + \sum_{v \in [N] \setminus \{j\}} d_{jv}X_v\right| = \left|\bh^\top \bm Z_j + \sum_{v \in [N]\setminus \{j\}} d_{jv}X_v\right| \le |\bh^\top \bm Z_j| +  \|\bm D\|_\infty \le \Theta M s + |\beta|, $$
		it follows from \eqref{minp} that
		$$\min\{p,1-p\} \ge \frac{1}{2} e^{-2\Theta M s - 2|\beta|}.$$
		Hence, we have:
		$$\mathrm{Var}(X_j|\bm X_{[N]\setminus\{j\}}) \ge e^{-4\Theta M s - 4 |\beta|} . $$   
		This completes the proof of the lemma, since $\| \bm D \|_\infty \leq |\beta|$ (since $\| \bm A \|_\infty \leq 1$ by \eqref{eq:Anorm}).  \end{proof}

	By the above lemma, model \eqref{eq:logisticmodel} is an $(R, \Upsilon)$-Ising model, with $R := |\beta| $ and $\Upsilon = e^{-4\Theta M s - 4 |\beta| }$. Next, choose $$\eta := \min\left\{\frac{1}{16},~|\beta| \right\},$$ and suppose  $I_1,\ldots,I_\ell$ be subsets of $[N]$ as in Lemma \ref{combinatorial1}. Then, defining $\ell' := \lceil \eta \ell / R \rceil$ we get, 
	\begin{equation}\label{intermediate1}
		\left|\sum_{i=1}^N \phi_{i}(\bm X)\right| = \left|\frac{1}{\ell'}\sum_{r=1}^{\ell} \sum_{i \in I_r} \phi_{i}(\bm X)\right| \le \frac{1}{\ell'}\sum_{r=1}^{\ell} \left|\sum_{i \in I_r} \phi_{i}(\bm X)\right| \le \frac{\ell}{\ell'} \max_{r \in [\ell]}\left|Q_{r}(\bm X)\right|,
	\end{equation}
	where 
	\begin{align}\label{eq:qr}
		Q_r (\bm X) := \sum_{i \in I_r} \phi_{i}(\bm X),
	\end{align} 
	for $r \in [\ell]$. Similarly, it follows that, for $s \in [d]$,  
	\begin{equation}\label{intermediate2}
		\left|\sum_{i=1}^N \phi_{i, s}(\bm X)\right| = \left|\frac{1}{\ell'}\sum_{r=1}^{\ell} \sum_{i \in I_r} \phi_{i, s}(\bm X)\right| \le \frac{1}{\ell'}\sum_{r=1}^{\ell} \left|\sum_{i \in I_r} \phi_{i, s}(\bm X)\right| \le \frac{\ell}{\ell'} \max_{r \in [\ell]}\left|Q_{r, s}(\bm X)\right|,
	\end{equation}
	where 
	\begin{align}\label{eq:qrs}
		Q_{r, s} (\bm X) := \sum_{i \in I_r} \phi_{i, s}(\bm X).
	\end{align} 
	The following lemma shows that the functions $Q_r$ and $Q_{r, s}$ are Lipschitz in the Hamming metric. The proof is given in Appendix \ref{sec:partialbetalipspf}. 
	
	\begin{lemma}\label{partialbetalips}
		For $r \in [\ell]$ and $s \in [d]$, let $Q_r$ and $Q_{r, s}$ be as defined in \eqref{eq:qr} and \eqref{eq:qrs}, respectively. Then for any two vectors $\bs,\bs' \in \cC_N$ differing in just the $k$-th coordinate, the following hold: 
		\begin{itemize}
			
			\item[(1)] For $r \in [\ell]$, 
			$$|Q_{r}(\bs) -Q_{r}(\bs')| \leq \frac{2|\beta| + 6}{N}.$$
			
			\item [(2)] Similarly, for $r \in [\ell]$ and $s \in [d]$, 
			$$|Q_{r, s}(\bs) - Q_{r, s}(\bs')| \leq \frac{2|Z_{k,s}|}{N} + \frac{2|\beta|}{N}\sum_{i=1}^N \left| Z_{i,s} a_{ik}\right| =: c_k. 
			$$ 
		\end{itemize} 
	\end{lemma}

	Using the above result together with Lemma \ref{eq:RU}, we can now establish the concentrations of $Q_{r} (\bm X)$ and $Q_{r, s} (\bm X) $, conditional on $\bm X_{I_r^c}$. To this end, recalling the definition of $\phi_i(\cdot)$ from \eqref{eq:function1} note that $\e\left(Q_r(\bm X)| \bm X_{I_r^c}\right) = 0$, for $r \in [\ell]$. Moreover, by Lemma \ref{eq:RU}, $\bm X| \bm X_{I_r^c}$ is an $(\eta, \Upsilon)$-Ising model, where $\eta \leq \frac{1}{16}$. Therefore, since $Q_r$ is $O_{\beta}(1/N)$-Lipschitz (by Lemma \ref{partialbetalips}), applying Theorem 4.3 and Lemma 4.4 in \citet{scthesis} gives, for every $t > 0$, 
	\begin{equation}\label{conc1}
		\bp\left(\left|Q_r(\bm X)\right| \geq t~\Big| \bm X_{I_r^c}\right) \leq 2e^{-O_\beta(Nt^2)}. 
	\end{equation}	
	Similarly, recalling \eqref{eq:function2}, it follows that for each $r \in [\ell]$ and $s \in [d]$, $\e(Q_{r, s}(\bm X) | \bm X_{I_r^c}) = 0$. Then, since $\sum_{k=1}^N c_k^2 = O_M(1)$ under Assumptions \ref{asm:a1} and \ref{asm:RSC_random}, Lemma \ref{partialbetalips} together with Theorem 4.3 and Lemma 4.4 in \citet{scthesis} gives, for every $t \geq 0$, 
	\begin{equation}\label{conc2}
		\bp\left(\left|Q_{r, s}(\bm X)\right| \geq t~\Big| \bm X_{I_r^c}\right) \leq 2e^{-O_{\beta, M}(Nt^2)} .
	\end{equation}
	
	Hence, combining \eqref{intermediate1}, \eqref{conc1}, \eqref{intermediate2}, \eqref{conc2}, and Lemma \ref{combinatorial1} (which implies that $\ell =O(\log N)$) gives, 
	\begin{equation}\label{final2}
		\p\left(\left|\sum_{i=1}^N \phi_{i}(\bm X)\right| \ge t\right) \le 2e^{- O_{\beta}(Nt^2)} \quad \text{and} \quad	 \p\left(\left|\sum_{i=1}^N \phi_{i, s}(\bm X)\right| \ge t\right) \le 2e^{- O_{\beta, M}(Nt^2)} ,
	\end{equation}
	for each $s \in [d]$. It thus follows from \eqref{final2} and a union bound, that
	\begin{equation}\label{devconc}
		\bp\left(\|\nabla L_N(\bg)\|_\infty \geq t\right) \leq 2(d+1)e^{-K Nt^2},
	\end{equation}
	for some constant $K > 0$, depending on $\beta$ and $M$. Now, choosing $t=  \frac{\lambda}{2} = \frac{1}{2} \delta \sqrt{\log (d+1)/N}$ in \eqref{devconc}  above gives, 
	\begin{equation*}
		\bp\left(\|\nabla L_N(\bg)\|_\infty \geq \frac{\lambda}{2} \right) \leq 2(d+1)^{1-\frac{K \delta^2}{4}} = o(1),  
	\end{equation*}
	whenever $\delta^2 < 4/K$. This completes the proof of Lemma \ref{lm:gradient}.

	\subsection{Proof of Lemma \ref{rsccond}}\label{poshes} 
	
	Define the following $(d+1) \times (d+1)$ dimensional matrix, 
	\begin{equation}\label{eq:mZmatrix}
		\bm G :=  \frac{1}{N}\left(
		\begin{array}{cc}
			\bm m^\top \bm m & \bm m^\top \bm Z\\
			\bm Z^{\top} \bm m & \bm Z^\top \bm Z 
		\end{array}
		\right) 
	\end{equation}   
	The key step towards proving Lemma \ref{rsccond} is to show that the lowest eigenvalue of $\nabla^2 L_N$ is bounded away from $0$ with high probability. 
	
	\begin{lemma}\label{mineiglem}
		There exists a constant $C>0$ (depending only on $s,\Theta$ and $M$), such that 
		\begin{align}\label{eq:mineiglem}
			\p\left(\lambda_{\min}(\bm G) \ge \frac{C\|\bm A\|_F^2}{N}\right) \ge 1-e^{-\Omega(\|\bm A\|_F^4/N)},
		\end{align} 
		as $N, d \rightarrow \infty$, such that $d=o(\|\bm A\|_F^2)$.
	\end{lemma}
	
	The proof of Lemma \ref{mineiglem} is given in Section \ref{sec:mineigenpf}. We first show it can be used to complete the proof of Lemma \ref{rsccond}. To this end, by a second order Taylor series expansion, we know that there exists $\alpha \in (0, 1)$ and 
	$\underline{\bg} = (\underline{\beta},\underline{\bh}^\top) ^\top = \alpha \bg + (1-\alpha) \hat{\bg}$ 
	such that 
	\begin{align}\label{second_taylor}
		L_N(\hat{\bg}) - L_N(\bg) - \nabla L_N(\bg) ^\top (\hat{\bg}-\bg) & = \frac{1}{2} (\hat{\bg}-\bg)^\top \nabla^2 L_N(\underline{\bg}) (\hat{\bg}-\bg) \nonumber \\ 
		& = \frac{1}{2N}\sum_{i=1}^N \frac{(\hat{\bg}-\bg)^\top \bm U_i \bm U_i^\top (\hat{\bg}-\bg)}{\cosh^2(\underline{\beta} m_i(\bs) + \underline{\bh}^\top \bm Z_i)} ,   
	\end{align}
	where $\bm U_i := (m_i(\bs),\bm Z_i^\top)^\top$. Now, note that: 
	\begin{align}\label{eq:B}
		|\underline{\beta} m_i(\bs)| \leq |\underline{\beta}| |m_i(\bs)| \leq  
		|\underline{\beta}| ||A||_\infty \leq |\beta| + |\hat \beta-\beta| \leq |\beta| + \|\hat \bg-\bg\|_1 
	\end{align} 
	and 
	\begin{align}\label{eq:Theta}
		|\underline{\bh}^\top \bm Z_i| 
		\leq \|\underline{\bh}\|_1 \|\bm Z_i\|_\infty \leq M\left(\|\underline{\bh}-\bh\|_1 + \|\bh\|_1\right) \leq M(\|\hat{\bg}-\bg\|_1 + s\Theta).
	\end{align} 
	Since $\cosh$ is an even function and increasing on the positive axis, we obtain
	\begin{equation}\label{eq:nnd}
		\frac{1}{2N}\sum_{i=1}^N \frac{(\hat{\bg}-\bg)^\top \bm U_i \bm U_i^\top (\hat{\bg}-\bg)}{\cosh^2(\underline{\beta} m_i(\bs) + \underline{\bh}^\top \bm Z_i)} \ge  \frac{(\hat{\bg}-\bg)^\top \bm G (\hat{\bg}-\bg) }{2\cosh^2(|\beta|  + (M+1)(\|\hat{\bg}-\bg\|_1) + s M \Theta)},
	\end{equation}
	where $\bm m := (m_1(\bm X),\ldots,m_N(\bm X))^\top$, $\bm Z = (\bm Z_1,\ldots,\bm Z_N)^\top$, and $\bm G$ is as defined in \eqref{eq:mZmatrix}. 
	Combining \eqref{second_taylor} and \eqref{eq:nnd} gives, 
	\begin{align}\label{eq:hessianLN}
		L_N(\hat{\bg}) - L_N(\bg) - \nabla L_N(\bg) ^\top (\hat{\bg}-\bg) & = \frac{1}{2} (\hat{\bg}-\bg)^\top \nabla^2 L_N(\underline{\bg}) (\hat{\bg}-\bg) \nonumber \\ 
		& \ge \frac{(\hat{\bg}-\bg)^\top \bm G (\hat{\bg}-\bg)}{2\cosh^2(|\beta| + (M+1)(\|\hat{\bg}-\bg\|_1) + s M \Theta )} . 
	\end{align}
	
	Next, we establish a high probability upper bound on $\|\hat{\bg}-\bg\|_1$ whenever the conditions of Theorem \ref{thm:estimate} are satisfied.

	\begin{lemma}\label{lm:L1} Suppose \eqref{eq:mineiglem} holds. Then, for $\lambda := \delta \sqrt{\log(d+1)/N}$ as in Lemma \ref{lm:gradient}, 
		$$ \|\hat{\bg}-\bg\|_1 =  O_s\left(\frac{N}{\|\bm A\|_F^2}\sqrt{\frac{\log(d+1)}{N}}\right).$$
		with probability $1-o(1)$, whenever $N, d \rightarrow \infty$ such that $\log d = o(\|\bm A\|_F^4/N)$, 
	\end{lemma}

	\begin{proof} By the convexity of the function $L_N$ it follows from \eqref{lneqn} that 
		\begin{align}\label{firsteqqq}
			\lambda(\|\bh \|_1 - \|\hat{\bh }\|_1) & \geq L_N(\hat{\bg}) - L_N(\bg) \nonumber \\ 
			& \geq \nabla L_N(\bg)^\top (\hat{\bg}-\bg) \nonumber \\ 
			& \geq -\|\nabla L_N(\bg)\|_\infty \|\hat{\bg}-\bg\|_1 \nonumber \\ 
			& \geq -\frac{\lambda \|\hat{\bg}-\bg\|_1 }{2} , 
		\end{align}	
		where the last step uses $\norm{\nabla L_N(\bg)}_\infty \leq \tfrac{\lambda}{2}$, for $\bm X \in \cE_N$. Recall from \eqref{secstep2} that $\|\hat{\bh}\|_1 -   \|\bh\|_1  \ge  	 \|(\hat{\bh}-\bh)_{S^c}\|_1 - \|(\hat{\bh}-\bh)_S\|_1 $. Therefore, from \eqref{firsteqqq}, we have: 
		\begin{align*} 
			\|(\hat{\bh}-\bh)_{S^c}\|_1 - \|(\hat{\bh}-\bh)_S\|_1 \leq \|\hat{\bh}\|_1 -   \|\bh\|_1 & \leq \frac{\| \hat{\bg}-\bg \|_1}{2}  \nonumber \\ 
			& = \frac{\| (\hat{\bh}-\bh) _S\|_1}{2} + \frac{\| (\hat{\bh}-\bh)_{S^c} \|_1}{2} + \frac{|\hat \beta - \beta|}{2}. 
		\end{align*}
		This means, $\|(\hat{\bh}-\bh)_{S^c}\|_1 \leq 3 ( \| (\hat{\bh}-\bh) _S\|_1 + |\hat \beta - \beta| ) $, and hence,
		\begin{align}\label{eq:vS} 
			\|(\hat{\bg}-\bg) \|_1  %
			& \leq 4  ( \| (\hat{\bh}-\bh) _S\|_1 + |\hat \beta - \beta| ) .
		\end{align} 
		Denote $\mathcal K :=  \| (\hat{\bh}-\bh) _S\|_1 + |\hat \beta - \beta| $. By the Cauchy-Schwarz inequality, 
		\begin{align}\label{eq:L1L2S}
			\mathcal K \leq  \sqrt{s+1} \sqrt{ \sum_{i \in S} ( \hat \theta_i - \theta_i )^2 + (\hat \beta - \beta)^2 }  \leq \sqrt{s+1} || \hat \bg - \bg ||_2 .  
		\end{align} 
		
		Next, for $t \in [0, 1]$, let $\bg_t := t\hat{\bg} + (1-t) \bg$, and $g(t) := (\hat{\bg} - \bg)^\top \nabla L_N(\bg_t).$ Then 
		\begin{equation}\label{step1}
			|g(1)-g(0)| = \big|(\hat{\bg} - \bg)^\top (\nabla L_N(\hat{\bg}) - \nabla  L_N(\bg))\big| \le \|\hat{\bg} - \bg\|_1 \cdot \|\nabla L_N(\hat{\bg}) - \nabla  L_N(\bg)\|_\infty.
		\end{equation}
		Therefore, 
		\begin{align*}
			g'(t) &= (\hat{\bg}-\bg)^\top \nabla^2 L_N(\bg_t) (\hat{\bg}-\bg)\\
			&= \frac{1}{N}\sum_{i=1}^N \frac{(\hat{\bg}-\bg)^\top \bm U_i \bm U_i^\top (\hat{\bg}-\bg)}{\cosh^2(\beta_t  m_i(\bs) + \bh_t^\top \bm Z_i)} \tag*{(where $\bm U_i := (m_i(\bs), \bm Z_i^\top)^\top)$} \\ 
			&\ge \frac{(\hat{\bg}-\bg)^\top \bm G (\hat{\bg}-\bg)}{\cosh^2(|\beta| + (M + 1) \|\bg_t - \bg\|_1 + s M \Theta)} \tag*{(by \eqref{eq:B} and \eqref{eq:Theta})} \\ 
			&\ge  \frac{\|\hat{\bg}-\bg\|_2^2 ~\lambda_{\mathrm{min}}(\bm G)}{\cosh^2(|\beta| + (M+1)  \|\bg_t - \bg\|_1 + s M \Theta)}\\ 
			&\gtrsim_C \frac{\|\hat{\bg}-\bg\|_2^2 \|\bm A\|_F^2}{N\cosh^2(|\beta| + 4 |t| (M+1) \mathcal K + Ms\Theta)} ,  
		\end{align*}
		where the last step uses \eqref{eq:mineiglem} (which holds with probability $1-o(1)$), $C$ is as in Lemma \ref{mineiglem}, and $\|\bg_t - \bg\|_1 = |t| \|(\hat{\bg}-\bg) \|_1 \leq 4 |t| \mathcal K$ (by \eqref{eq:vS}).  Hence, 
		\begin{align}\label{in3}
			|g(1)-g(0)|&\ge g(1) - g(0)=\int_0^1 g'(t)~dt\nonumber\\ &\ge \int_0^{\min\{1,~ 1/ \mathcal K \}} g'(t)~dt\nonumber\\&\gtrsim_{s} \frac{\|\bm A\|_F^2}{N}\|\hat{\bg} - \bg\|_2^2 ~\min \{1, ~ 1/\mathcal K \} . 
		\end{align}
		Combining \eqref{step1} and \eqref{in3} gives, 
		\begin{align} 
			\min\{ \mathcal K, ~1\} & \lesssim_s \frac{ \mathcal K N \| \hat{\bg}-\bg \|_1 }{\|\bm A\|_F^2 \|\hat{\bg}-\bg\|_2^2} \cdot \|\nabla L_N(\hat{\bg}) - \nabla L_N(\bg)\|_\infty \nonumber \\ 
			& \lesssim \frac{ \mathcal K^2 N}{\|\bm A\|_F^2\|\hat{\bg}-\bg\|_2^2} \cdot \|\nabla L_N(\hat{\bg}) - \nabla L_N(\bg)\|_\infty \tag*{ (by \eqref{eq:vS}) } \nonumber \\ 
			\label{eq:in4} & \lesssim \frac{sN}{\|\bm A\|_F^2} \|\nabla L_N(\hat{\bg}) - \nabla L_N(\bg)\|_\infty ,
		\end{align} 
		using \eqref{eq:L1L2S}. Now, recall that, by Lemma \ref{lm:gradient}, with probability $1 - o(1)$,
		$\|\nabla L_N(\bg)\|_\infty \lesssim_{\delta} \sqrt{\log(d+1)/N}$ and $\|\nabla L_N(\hat{\bg})\|_\infty \lesssim \sqrt{\log(d+1)/N}$. Applying this in \eqref{eq:in4} shows that with probability $1-o(1)$, $$\min\{ \mathcal K,~1\} = O_s\left(\frac{N}{\|\bm A\|_F^2}\sqrt{\frac{\log(d+1)}{N}}\right).$$ This implies,  $$\mathcal K = \| (\hat{\bh}-\bh) _S\|_1 + |\hat \beta - \beta| = O_s\left(\frac{N}{\|\bm A\|_F^2}\sqrt{\frac{\log(d+1)}{N}}\right) , $$ with probability $1-o(1)$,  whenever $N, d\rightarrow \infty$ such that $\log d = o(\|\bm A\|_F^4/N)$. Therefore, by \eqref{eq:vS}, $\|\hat{\bg}-\bg\|_1 = O_s\left(\frac{N}{\|\bm A\|_F^2}\sqrt{\frac{\log(d+1)}{N}}\right)$ with probability $1-o(1)$. 
	\end{proof} 
	
	Using Lemma \ref{mineiglem} and Lemma \ref{lm:L1} in \eqref{eq:hessianLN} it follows that, there exists $\kappa $ (as the statement of Lemma \ref{rsccond}) such that 
	\begin{align*}
		L_N(\hat{\bg}) - L_N(\bg) - \nabla L_N(\bg) ^\top (\hat{\bg}-\bg) 
		& \ge \frac{(\hat{\bg}-\bg)^\top \bm G (\hat{\bg}-\bg)}{2\cosh^2(|\beta| + (M+1) \|\hat{\bg}-\bg\|_1 + s M \Theta )} \nonumber \\
		& \gtrsim_{\beta, M, \kappa} \frac{\|\bm A\|_F^2\norm{ \hat{\bg}-\bg }_2^2}{N}
	\end{align*}
	with probability $1-o(1)$, as $N, d\rightarrow \infty$ such that $d = o(\|\bm A\|_F^2)$ and $\log d = o(\|\bm A\|_F^4/N)$. This completes the proof of Lemma \ref{rsccond}.

	\begin{rem} \label{remark:M} \em{Note that if one assumes $|| \bm Z_i ||_2 \leq M$, for all $1 \leq i \leq N$, and $|| \bh ||_2 \leq \Theta$ (recall the discussion in Remark \ref{remark:sdependence}) then, for $\underline {\bh}$ as in \eqref{eq:Theta}, by the Cauchy-Schwarz inequality,  
			\begin{align*}
				|\underline{\bh}^\top \bm Z_i| 
				\leq \|\underline{\bh}\|_2 \|\bm Z_i\|_2 \leq M\left(\|\underline{\bh}-\bh\|_2 + \|\bh\|_2\right) \leq M(\|\hat{\bg}-\bg\|_1 + \Theta) .
			\end{align*} 
			Using this bound and \eqref{eq:B} we get, 
			\begin{equation*}
				\frac{1}{2N}\sum_{i=1}^N \frac{(\hat{\bg}-\bg)^\top \bm U_i \bm U_i^\top (\hat{\bg}-\bg)}{\cosh^2(\underline{\beta} m_i(\bs) + \underline{\bh}^\top \bm Z_i)} \ge  \frac{(\hat{\bg}-\bg)^\top \bm G (\hat{\bg}-\bg) }{2\cosh^2(|\beta|  + (M+1)(\|\hat{\bg}-\bg\|_1) + M \Theta)}. 
			\end{equation*}
			Note that the bound in the RHS above does not have any dependence on $s$ in the $\mathrm{cosh}$ term (unlike in \eqref{eq:nnd}). Hence, by the same arguments as before we can now get the following rate where the dependence on $s$ matches that in the classical high-dimensional logistic regression: 
			\begin{align*}
				\|\hat{\bh} - \bh \|_2 = O\left(\sqrt{\frac{s \log d}{N}}\right),
			\end{align*} 
			with probability $1-o(1)$, as $N, d\rightarrow \infty$ such that $d=o(N)$. }
	\end{rem}

	\subsubsection{Proof of Lemma \ref{mineiglem}}\label{sec:mineigenpf}
	
	The first step towards proving Lemma \ref{mineiglem} is to observe that:
	$$\det(\bm G-\lambda \bm I) = \left(\frac{1}{N}\|\mf \bm m\|_2^2 - \lambda\right) \cdot \det\left(\frac{1}{N} \bm Z^\top \bm Z - \lambda I\right),$$
	where $\bm F := \bm I - \bm Z (\bm Z^\top \bm Z)^{-1} \bm Z^\top$.
	Hence, 
	\begin{equation*}
		\lambda_{\min}(\bm G) = \min \left\{\lambda_{\min}\left(\frac{1}{N} \bm Z^\top \bm Z\right),~\frac{1}{N}\|\mf \bm m\|_2^2\right\}.
	\end{equation*}
	In view of Assumption \ref{asm:a2}, to prove Lemma \ref{mineiglem} it suffices to show that there exists a constant $C>0$ (depending only on $s,\Theta$ and $M$), such that
	\begin{equation}\label{fm}
	\p\left(\frac{1}{N}\|\mf \bm m\|_2^2 \ge \frac{C\|\bm A\|_F^2}{N}\right) = 1-e^{-\Omega(\|\bm A\|_F^4/N)}. 
	\end{equation} 
	To this end, it suffices to prove the following conditional version of \eqref{fm}:
	\begin{equation}\label{fmm}
		\p\left(\frac{1}{N}\|\mf \bm m\|_2^2 \ge \frac{C\|\bm A\|_F^2}{N}\Big|\bm X_{J^c}\right) = 1-e^{-\Omega(\|\bm A\|_F^4/N)}
	\end{equation} 
	where $J$ is a suitably chosen subset of $[N]$ and $\bm X_{J^c}:= (X_i)_{i\in J^c}$. To this end, note that \eqref{eq:logisticmodel} is a $(|\beta|\|\bm A\|_\infty, \Upsilon)$-Ising model (Lemma \ref{eq:RU}). Now, applying Lemma \ref{combinatorial1} with 
	$$\eta := \min\left\{\frac{1}{16},~|\beta|\|\bm A\|_\infty\right\},$$
	gives subsets $I_1,\ldots,I_\ell\subseteq [N]$ such that, for all $1 \le j \le \ell$, the conditional distribution of $\bm X_{I_j}$ given $\bm X_{I_j^c} := (X_u)_{u\in [N]\setminus I_j}$ is an $(\eta,\Upsilon)$-Ising model. Furthermore, for any vector $\bm a$, there exists $j \in [\ell]$ such that 
	\begin{align}\label{eq:aL1}
		\|\bm a_{I_j}\|_1 \ge \frac{\eta}{8|\beta| \|\bm A\|_\infty} \|\bm a\|_1. 
	\end{align}   
	The proof of \eqref{fmm} now proceeds in the following two steps: First, we show that 
	there exists $j \in [\ell]$ such that the expectation of $N^{-1} \|\bm F \bm m\|_2^2$ conditioned on $\bm X_{I_j}$ is $\Omega(1)$.  Subsequently, we establish that conditioned on $\bm X_{I_j}$, $N^{-1} \|\bm F \bm m\|_2^2$ concentrates around its conditional expectation. These steps are verified in Lemmas \ref{meanlb} and \ref{concmeanlb}, respectively.

	\begin{lemma}\label{meanlb} Under the assumptions of Theorem \ref{thm:estimate}, there exists $J \in \{I_1, I_2, \ldots, I_\ell\}$ such that  for all $N \ge 1$, 
		$$\mathbb{E}\left(\frac{1}{N}\|\bm F \bm m\|_2^2\Big| \bm X_{J^c}\right)\ge  \frac{C\|\bm A\|_F^2}{N},$$
		where $C > 0$ is a constant depending only on $\Theta, M$ and $s$. 
	\end{lemma}
	
	\begin{proof}
		For any $(d+1) \times n$ dimensional matrix $\bm M$, we will denote the $i$-th row of $\bm M$ by $\bm M_i$ and  the $i$-th largest singular value of $\bm M$ by $\sigma_i(\bm M)$, for $1 \leq i \leq d+1$. Also, for $J \subseteq [N]$ denote $(\bm F\bm A)_{i,J} := ((\bm F \bm A)_{i,j})_{j \in J}$. Note that for any $J \in \{I_1, I_2, \ldots, I_\ell\} \subset [N]$, since $\bm m = \bm A \bm X$,  
		\begin{align}\label{fs1}
			\e \left(\|\bm F \bm m\|_2^2\Big| \bm X_{J^c}\right)  = \sum_{i=1}^N \e \left([(\bm F \bm A)_i \bm X]^2\Big| \bm X_{J^c}\right) & \ge \sum_{i=1}^N \mathrm{Var}\left((\bm F \bm A)_i \bm X\Big| \bm X_{J^c}\right) \nonumber \\ 
			& = \sum_{i=1}^N \mathrm{Var}\left((\bm F \bm A)_{i,J} \bm X_J\Big| \bm X_{J^c}\right)  \nonumber \\ 
			& \gtrsim \frac{\Upsilon^2}{\eta} \sum_{i=1}^N \|(\bm F\bm A)_{i, J}\|_2^2 ,
		\end{align} 
		where the last step follows from Lemma \ref{lemma10}. 	
		Now define a vector $\bm a=(a_1, a_2, \ldots, a_N)^{\top}$, with $a_i = \|(\bm F\bm A)_{\cdot, i}\|_2^2$, where $(\bm F\bm A)_{\cdot, i}$ denotes the $i$-th column of the matrix $\bm F \bm A$. Then by \eqref{eq:aL1} there exists $J \in \{I_1, I_2, \ldots, I_\ell\} \subseteq [N]$ such that 
		$$\sum_{i=1}^N \|(\bm F\bm A)_{i, J}\|_2^2 \geq \frac{\eta}{8|\beta| \|\bm A\|_\infty} \sum_{i=1}^N  \|(\bm F\bm A)_{\cdot, i}\|_2^2. $$ 
		Therefore, by \eqref{fs1}, 
		\begin{equation}\label{fs2}
			\e \left(\|\bm F \bm m\|_2^2\Big| \bm X_{J^c}\right)  \gtrsim \frac{\Upsilon^2}{|\beta| \|\bm A\|_\infty} \|\bm F \bm A\|_F^2 \gtrsim_{s, B, M}  \|\bm F \bm A\|_F^2 .
		\end{equation} 
		By Theorem 2 in \citet{wangfr}, we get
		\begin{equation}\label{wangeq}
			\sum_{t=1}^N \sigma_t^2(\bm F\bm A) \ge \sum_{t=1}^N \sigma_t^2(\bm F)~ \sigma_{N-t+1}^2(\bm A)
		\end{equation}
		Since $\bm F$ is idempotent with trace $N-d$, it follows that $\sigma_1(\bm F)=\ldots =\sigma_{N-d}(\bm F) = 1$ and $\sigma_{N-d+1}(\bm F)=\ldots =\sigma_{N}(\bm F) = 0$. Hence, we have from \eqref{wangeq},
		\begin{equation}\label{wangeq2}
			\|\bm F \bm A\|_F^2 = \sum_{t=1}^N \sigma_t^2(\bm F\bm A) \ge \sum_{t=1}^{N-d} \sigma_{N-t+1}^2(\bm A) = \|\bm A\|_F^2 - \sum_{i=1}^d \sigma_i^2(\bm A).
		\end{equation} 
		where the last step uses $\sigma_i^2(\bm A) \le 1$ for all $1\le i\le N$, since $\|\bm A\|_2 \leq \|\bm A\|_\infty \le 1$. Applying the bound in \eqref{wangeq2} to \eqref{fs2} gives, 
		\begin{equation}\label{fs22}
			\e \left(\|\bm F \bm m\|_2^2\Big| \bm X_{J^c}\right)  \gtrsim   \Upsilon^2 \left(\|\bm A\|_F^2 - d\right).
		\end{equation} 
		Lemma \ref{meanlb} now follows from the hypothesis $d=o(\|\bm A\|_F^2)$. 
	\end{proof}
	
	Next, we show that $\|\bm F \bm m\|_2^2$ concentrates around its conditional expectation $\mathbb{E}(\|\bm F \bm m\|_2^2 |\bm X_{J^c})$, for the set $J$ as defined above.

	\begin{lemma}\label{concmeanlb} For any $t > 0$ and $J \in \{I_1, I_2, \ldots, I_\ell\}$, 
		\begin{align}\label{eq:concmean}	
			& \bp \left(\|\bm F \bm m\|_2^2 < \mathbb{E}(\|\bm F \bm m\|_2^2 |\bm X_{J^c}) - t\Big| \bm X_{J^c} \right) \nonumber \\ 
			& \le 	2\exp\left(-C \cdot \min\left\{  \frac{t^2}{8N},~\frac{t}{2}\right\}\right) + 2 \exp\left(-\frac{t^2}{128 N}\right),
		\end{align} 
		where $C$ is a constant depending only on $\Theta,M$ and $s$.
	\end{lemma}
	\begin{proof} Denote $\bm W:= \bm F \bm A$ and let $\bm H$ be the matrix obtained from $\bm W^\top \bm W$ by zeroing out all its diagonal elements. Moreover, throughout the proof we will denote $I:=J^c$. Clearly,
		$$\|\bm F \bm m\|_2^2 - \mathbb{E}(\|\bm F \bm m\|_2^2 |\bm X_I) = \bm X^\top \bm H \bm X - \mathbb{E}\left(\bm X^\top \bm H \bm X | \bm X_I\right).$$
		By permuting the indices let us partition the vector $\bm X$ as $(\bm X_I^\top, \bm X_J^\top)^\top$ and the matrix $\bm H$ as: 
		\begin{equation*}
			\left(
			\begin{array}{cc}
				\bm H_{I,I} & \bm H_{I,J}\\
				\bm H_{I,J}^\top & \bm H_{J,J}
			\end{array}
			\right)
		\end{equation*}
		where for two subsets $A, B$ of $[N]$, we define $\bm H_{A,B} := ((H_{ij}))_{i \in A, j \in B}$. Note that 
		\begin{align}\label{depth1}
			&\bm X^\top \bm H \bm X - \mathbb{E}\left(\bm X^\top \bm H \bm X | \bm X_I\right)\nonumber\\ 
			&= \bm X_J^\top \bm H_{J,J} \bm X_J - \mathbb{E}\left(\bm X_J^\top \bm H_{J,J} \bm X_J | \bm X_I\right) + 2 \bm X_I^\top \bm H_{I,J} \bm X_J - 2 \mathbb{E}\left(\bm X_I^\top \bm H_{I,J} \bm X_J | \bm X_I\right). 
		\end{align}
		Since $\bm X|\bm X_I$ is an $(\eta, \Upsilon)$-Ising model, Example 2.5 in \citet{radek} implies (by taking the parameters $\alpha$ and $\rho$ in \citet{radek} to be $\Theta M s + 1/16$ and $7/8$, respectively), 
		\begin{align}\label{radek1}
			&\bp\left(\bm X_J^\top \bm H_{J,J} \bm X_J < \mathbb{E}\left(\bm X_J^\top \bm H_{J,J} \bm X_J | \bm X_I\right) - t\Big| \bm X_I\right)\nonumber\\ 
			&\le 2\exp \left(-c \cdot \min\left\{\frac{t^2}{\|\bm H_{J,J}\|_F^2 + \|\e (\bm H_{J,J} \bm X_J)\|_2^2},~\frac{t}{\|\bm H_{J,J}\|_2} \right\}\right), 
		\end{align}
		where $c$ is a constant depending only on $\Theta,M$ and $s$. Clearly, one has $\|\bm H_{J,J}\|_F^2 \le \|\bm H\|_F^2$ and, since the spectral norm of a matrix is always greater than or equal to the spectral norm of any of its submatrices, $\|\bm H_{J,J}\|_2 \le \|\bm H\|_2$. Moreover, if $$\bm X_0^\top := (\boldsymbol{0}^\top, \bm X_J^\top),$$ then $\bm H_{J,J} \bm X_J$ is a subvector of $\bm H \bm X_0$, and hence, 
		$$\|\e (\bm H_{J,J} \bm X_J)\|_2^2 \le \|\e(\bm H \bm X_0)\|_2^2.$$ Combining all these, we have from \eqref{radek1}
		\begin{align}\label{sstep} 
			& \bp\left(\bm X_J^\top \bm H_{J,J} \bm X_J < \mathbb{E}\left(\bm X_J^\top \bm H_{J,J} \bm X_J | \bm X_I\right) - t\Big| \bm X_I\right)\nonumber\\ 
			& \le 2\exp \left(-c \cdot \min\left\{\frac{t^2}{\|\bm H\|_F^2 + \|\e (\bm H \bm X_0)\|_2^2},~\frac{t}{\|\bm H\|_2} \right\}\right) 
		\end{align}
		Next, note that, since $\bm H - \bm W^\top \bm W$ is a diagonal matrix, 
		\begin{align}
			\|\bm H\|_F^2 + \|\e (\bm H\bm X_0)\|_2^2 & \le \|\bm H\|_F^2 + \left(\|\e (\bm W^\top \bm W \bm X_0)\|_2 + \|\e [(\bm H - \bm W^\top \bm W)\bm X_0]\|_2\right)^2 \nonumber \\ 
			&\le 2\|\bm H\|_F^2 + 2 \e \|[(\bm H - \bm W^\top \bm W)\bm X_0]\|_2^2 + 2\|\e (\bm W^\top \bm W \bm X_0)\|_2^2\nonumber\\
			&= 2\sum_{1 \leq i\ne j \leq N} (\bm W^\top \bm W)_{ij}^2 + 2 \sum_{i=1}^N (\bm W^\top \bm W)_{ii}^2 + 2\|\e (\bm W^\top \bm W \bm X_0)\|_2^2\nonumber\\\label{eqar1} 
			&= 2\|\bm W^\top \bm W\|_F^2 + 2\|\e (\bm W^\top \bm W\bm X_0)\|_2^2.
		\end{align}
		We also have $\|\bm H\|_2 \le \|\bm W^\top \bm W\|_2$, since for any vector $\bm a \in \mathbb{R}^N$,  
		$$\bm a^\top \bm H\bm a = \bm a^\top \bm W^\top \bm W\bm a - \sum_{i=1}^N (\bm W^\top \bm W)_{ii} a_i^2 \le \bm a^\top \bm W^\top \bm W \bm a.$$ 
		Hence, it follows from \eqref{sstep} and \eqref{eqar1} that 
		\begin{align}\label{radek2}
			&\bp\left(\bm X_J^\top \bm H_{J,J} \bm X_J < \mathbb{E}\left(\bm X_J^\top \bm H_{J,J} \bm X_J | \bm X_I\right) - t\Big| \bm X_I\right)\nonumber\\ 
			&\le 2\exp \left(-c' \cdot \min\left\{\frac{t^2}{\|\bm W^\top \bm W\|_F^2 + \|\e (\bm W^\top \bm W \bm X_0)\|_2^2},~\frac{t}{\|\bm W^\top \bm W\|_2} \right\}\right)
		\end{align}
		where $c' := c/2$.  Next, recall that for any two matrices $\bm U$ and $(\hat{\bg}-\bg)$ such that $\bm U (\hat{\bg}-\bg)$ exists, we have $\|\bm U (\hat{\bg}-\bg)\|_F \le \|\bm U\|_2 ~\|\hat{\bg}-\bg\|_F$. This implies,  
		\begin{equation}\label{rst1}
			\|\bm W^\top \bm W\|_F^2 \le \|\bm W\|_2^2~\|\bm W\|_F^2,
		\end{equation}
		Moreover, by the submultiplicativity of the matrix $\ell_2$ norm, 
		\begin{equation}\label{rst2}
			\|\bm W^\top \bm W\|_2 \le \|\bm W^\top\|_2 \|\bm W\|_2 = \|\bm W\|_2^2
		\end{equation}
		and 
		\begin{equation}\label{rst3}
			\|\e (\bm W^\top \bm W \bm X_0)\|_2^2 = \|\bm W^\top \e(\bm W\bm X_0)\|_2^2 \le \|\bm W\|_2^2\cdot \|\e(\bm W\bm X_0)\|_2^2.
		\end{equation}
		It follows from \eqref{radek2}, \eqref{rst1}, \eqref{rst2} and \eqref{rst3}, that:
		\begin{align}\label{prelem}
			& \bp\left(\bm X_J^\top \bm H_{J,J} \bm X_J < \mathbb{E}\left(\bm X_J^\top \bm H_{J,J} \bm X_J | \bm X_I\right) - t\Big| \bm X_I\right)\nonumber\\ 
			& \le 2\exp\left(-\frac{c'}{\|\bm F \bm A\|_2^2}\cdot \min\left\{  \frac{t^2}{\|\bm F \bm A\|_F^2 + \|\mathbb{E}\left[\bm W \bm X_0\right] \|_2^2},~t\right\}\right).
		\end{align}
		Since $\|\bm F \bm A\|_2^2 \le \|\bm F\|_2^2 \|\bm A\|_2^2 \le 1$, $\|\bm F \bm A\|_F^2 \le N\|\bm F \bm A\|_2^2 \le N$, and $\|\mathbb{E}\left[\bm W \bm X_0\right] \|_2^2 \le \e \|\bm W \bm X_0\|_2^2 \le \|\bm F \bm A\|_2^2 ~\e\|\bm X_0\|_2^2 \le N$, we can conclude from \eqref{prelem} that:
		\begin{equation}\label{prelem2}
			\bp\left(\bm X_J^\top \bm H_{J,J} \bm X_J < \mathbb{E}\left(\bm X_J^\top \bm H_{J,J} \bm X_J | \bm X_I\right) - t\Big| \bm X_I\right)\le 2\exp\left(-c'\cdot \min\left\{  \frac{t^2}{2N},~t\right\}\right).
		\end{equation}
		
		Next, let us define $\bm y := \bm H_{I,J}^\top \bm X_I$. Then,
		$$\bm X_I^\top \bm H_{I,J} \bm X_J - \mathbb{E}\left(\bm X_I^\top \bm H_{I,J} \bm X_J | \bm X_I\right) = \bm y^\top \bm X_J - \e(\bm y^\top \bm X_J|\bm X_I).$$
		By Lemma 4.4 in \citet{scthesis}, Dobrushin's interdependence matrix for the model $\bm X_J |\bm X_I$ is given by $8\bm D$, where $\bm D$ denotes the interaction matrix for the Ising model $\bm X_J |\bm X_I$. Since $\|8 \bm D\|_2 \le \frac{1}{2}$, by Theorem 4.3 in \citet{scthesis}, 
		\begin{equation}\label{lin1}
			\p\left(\bm y^\top \bm X_J < \e(\bm y^\top \bm X_J|\bm X_I) - t ~\Big| \bm X_I \right) \le 2 \exp\left(-\frac{t^2}{8\|\bm y\|_2^2}\right).
		\end{equation} 
		Using the submultiplicativity of the matrix $\ell_2$ norm and the fact that the spectral norm of a matrix is always greater than or equal to the spectral norm of any of its submatrices gives, 
		\begin{equation}\label{normv}
			\|\bm y\|_2^2 = \|\bm H_{I,J}^\top \bm X_I\|_2^2 \le \|\bm H_{I,J}^\top\|_2^2 \cdot \|\bm X_I\|_2^2 \le N \|\bm H\|_2^2 \le N \|\bm W\|_2^4 = N\|\bm F \bm A\|_2^4 \le N.
		\end{equation} 
		Combining \eqref{lin1} and \eqref{normv}, 
		\begin{equation}\label{linear}
			\p\left(\bm y^\top \bm X_J < \e(\bm y^\top \bm X_J|\bm X_I) - t ~\Big| \bm X_I\right) \le 2 \exp\left(-\frac{t^2}{8N}\right).
		\end{equation}
		Finally, combining \eqref{prelem2}, \eqref{linear}, and \eqref{depth1} gives, 
		\begin{align}
			& \bp\left(\bm X^\top \bm H \bm X < \mathbb{E}\left(\bm X^\top \bm H \bm X | \bm X_I\right) - t~\Big| \bm X_I\right) \nonumber \\ 
			&\le
			2\exp\left(-c'\cdot \min\left\{  \frac{t^2}{8N},~\frac{t}{2}\right\}\right) + 2 \exp\left(-\frac{t^2}{128 N}\right). \nonumber 
		\end{align}
		This completes the proof of Lemma \ref{concmeanlb}. 
	\end{proof}
	
	To complete the proof of \eqref{fmm}, we choose $J$ as in Lemma \ref{meanlb} and apply Lemma \ref{concmeanlb} with $t= \frac{1}{2} \e(\|\bm F \bm m\|_2^2 |\bm X_{J^c})$. This implies, there exists a constant $C$, depending only on $\Theta,M$ and $s$, such that 
	$$\p\left(\frac{1}{N}\|\mf \bm m\|_2^2 \ge \frac{C\|\bm A\|_F^2}{N}\Big|\bm X_{J^c}\right) = 1-e^{-\Omega(\|\bm A\|_F^4/N)}.$$
	This proves \eqref{fmm} and completes the proof of Lemma \ref{mineiglem}.  \hfill $\Box$

	\section{Proof of Theorem \ref{thm:stheta}} 
	\label{sec:sthetapf}
	
	Recall the definition of the log-pseudo-likelihood function $L_{\beta, N}(\cdot)$ from \eqref{eq:LNtheta}. As before, the proof of Theorem \ref{thm:estimate} entails showing the following: (1) concentration of the gradient of $L_{\beta, N}$ and (2) restricted strong concavity of  $L_{\beta, N}$.  
	
	The concentration of the gradient $\nabla L_{\beta, N}$ follows by arguments similar to Lemma \ref{lm:gradient}. Towards this, recall the definition of the functions $\phi_{i, s}: \cC_N \rightarrow \R$,  for $1 \leq i \leq N$ and $1 \leq s \leq d$, from \eqref{eq:function2}. Note that $\nabla L_{\beta, N} = (\frac{\partial L_{\beta, N}}{\partial \theta_1}, \ldots, \frac{\partial L_{\beta, N}}{\partial \theta_d})^\top$,  where $\frac{\partial L_{\beta,N}}{\partial \theta_s} = \sum_{i=1}^N \phi_{i, s}(\bm X)$ for $1 \leq s \leq d$. For $k  \in [N]$, recall from Lemma \ref{partialbetalips} the definition of 
	$$c_k := \frac{2|Z_{k,s}|}{N} + \frac{2|\beta|}{N}\sum_{i=1}^N \left|  Z_{i,s} a_{ik}\right|, $$  
	where $s \in [d]$. Therefore, for $s \in [d]$, 
	\begin{align*}
		\sum_{k=1}^N c_k^2 &\lesssim \frac{1}{N^2}\sum_{k=1}^N Z_{k,s}^2 + \frac{\beta^2}{N^2} \sum_{k=1}^N \left(\sum_{i=1}^N |Z_{i,s}a_{ik}|\right)^2\\ 
		&\le \frac{1}{N}\max_{1\le s \le d} \frac{1}{N} \sum_{k=1}^N Z_{k, s}^2 + \frac{\beta^2}{N^2} \sum_{k=1}^N \left(\sum_{i=1}^N Z_{i,s}^2 |a_{ik}| \sum_{i=1}^N |a_{ik}| \right) \tag*{(by the Cauchy-Schwarz inequality)} \\ 
		&\le \frac{1}{N}\max_{1\le s \le d} \frac{1}{N} \sum_{k=1}^N Z_{k, s}^2  + \left(\max_{1\le k\le N} \sum_{i=1}^N |a_{ik}|\right)\frac{\beta^2}{N^2} \sum_{i=1}^N Z_{i, s}^2\sum_{k=1}^N|a_{ik}| \\ 
		&\le \frac{1}{N} \left\{ \max_{1\le s \le d} \frac{1}{N} \sum_{k=1}^N Z_{k, s}^2  + \beta^2 \|\bm A\|_1^2 \max_{1\le s \le d} \frac{1}{N} \sum_{i=1}^N Z_{i, s}^2 \right \} 
		\\
		& \lesssim_{\beta, C} \frac{1}{N}, 
	\end{align*} 
	where the last holds with probability 1 under Assumptions \ref{asm:a1} and \ref{asm:RSC_random}. Then by analogous arguments as in Lemma \ref{lm:gradient} it follows that there exists  $\delta > 0$ such that with  $\lambda := \delta \sqrt{\log d/N}$  the following holds: 
	\begin{align}\label{eq:sthetagradient}
		\bp\left(\norm{\nabla L_{\beta, N}(\bh)}_\infty > \frac{\lambda}{2}\right) = o(1), 
	\end{align} 
	where the $o(1)$-term goes to infinity as $d \rightarrow \infty$.

	To establish the strong concavity of $L_{\beta, N}$ we consider the second-order Taylor expansion: For any $\bm \eta \in \R^d$, 
	\begin{align*}
		L_{\beta,N}(\tht + \bm \eta)-L_{\beta,N}(\tht)-\nabla L_{\beta,N}(\tht)^{\top} \bm \eta=\frac{1}{2}\bm \eta^{\top}\nabla^{2}L_{\beta,N}(\tht+t \bm \eta) \bm \eta , 
	\end{align*}
	for some $t \in (0,1)$. Computing the Hessian matrix $\nabla^{2}L_{\beta,N}(\tht+t \bm \eta)$ gives, 
	\begin{align} 
		L_{\beta,N}(\tht + \bm \eta)  -L_{\beta,N}(\tht) & - \nabla L_{\beta,N}(\tht)^{\top} \bm \eta \nonumber \\ 
		&
		=\frac{1}{2N}\sum_{i=1}^{N}\frac{\bm \eta^{\top}\Z_i\Z_i^{\top}\bm \eta}{\cosh^{2}(\beta m_i(\X)+(\tht+t\bm \eta)^{\top}\Z_i)} \nonumber\\
		&
		\geq \frac{1}{2N}\sum_{i=1}^{N}\frac{\bm \eta^{\top}\Z_i\Z_i^{\top}\bm \eta}{\cosh^{2}(|\beta|\|\A\|_{\infty}+(\tht+t\bm \eta)^{\top}\Z_i)} \tag*{(using $|\beta m_i(\bs)| \leq |\beta| ||\A||_\infty)$ } \nonumber \\
		&
		=\frac{1}{2N}\sum_{i=1}^{N}\langle \bm \eta, \Z_i\rangle^{2}\frac{1}{\cosh^{2}(|\beta|\|\A\|_{\infty}+(\tht+t\bm \eta)^{\top}\Z_i)} \nonumber \\
		&
		=\frac{1}{2N}\sum_{i=1}^{N}\langle \bm \eta, \Z_i\rangle^{2}\psi(\langle \tht,\Z_i\rangle+ t\langle \bm \eta,\Z_i\rangle) , \label{eq:LN2}
	\end{align}
	where $\psi(x): =\mathrm{sech}^{2}(|\beta|\|\A\|_{\infty}+x)$.

	\begin{proposition}\label{proposition:RSC_random} Suppose the assumptions of 
		Theorem \ref{thm:stheta} hold. Then there exists positive constants $\nu, c_0$ (depending on $\kappa_1, \kappa_2$, $\beta$, and $\Theta$) such that for all $t \in (0, 1)$, 
		\begin{align*}
			\frac{1}{N}\sum_{i=1}^{N}\langle \bm \eta, \Z_i\rangle^{2}\psi(\langle \tht,\Z_i\rangle+ t\langle \bm \eta,\Z_i\rangle)\geq \nu \|\bm \eta\|_2^{2}-c_0 \mathcal R_N \|\bm \eta\|_1^{2}, 
		\end{align*}
		with probability at least $1 - o(1)$, for all $\bm \eta\in\mathbb{R}^{d}$  with $\|\bm \eta\|_2\leq 1$. 
	\end{proposition}

	The proof of Proposition \ref{proposition:RSC_random} is given in Section \ref{sec:ppnLN_pf}. Here we use it to complete the proof of Theorem \ref{thm:stheta}. Note that by \eqref{eq:LN2} and Proposition \ref{proposition:RSC_random}, for any $\bm \eta \in \mathbb R^d$ with $\| \bm \eta \|_2 \leq 1$, 
	\begin{align*}
		L_{\beta,N}(\tht + \bm \eta)  -L_{\beta,N}(\tht) & - \nabla L_{\beta,N}(\tht)^{\top} \bm \eta 
		\gtrsim \nu \|\bm \eta\|_2^{2}- c_0 \mathcal R_N \|\bm \eta\|_1^{2}
	\end{align*}
	This establishes the restricted strong convexity (RSC) property for pseudo-likelihood loss function $L_{\beta,N}(\cdot)$. Therefore, by \citet[Corollary 9.20]{wainwrighthighdimensional}, whenever $\mathcal R_N s = s \sqrt{\log d /n} = o(1)$, then when the event $\{\|\nabla L_N(\hat{\tht})\|_{\infty}\leq\frac{\lambda}{2}\}$ happens, 
	\begin{align}\label{eq:theta12}
		\|\hat{\tht}-\tht\|^{2}_2 \lesssim_\nu s\lambda^2 \lesssim_{\nu, \delta} \frac{s \log d}{n} \quad \text{ and } \quad  \|\hat{\tht}-\tht\|_1 \lesssim_\nu s\lambda \lesssim_{\nu, \delta} s \sqrt{\frac{\log d}{n}}. 
	\end{align}
	Since the event $\{\|\nabla L_N(\hat{\tht})\|_{\infty}\leq\frac{\lambda}{2}\}$ happens with probability $1-o(1)$ by \eqref{eq:sthetagradient} (jointly over the randomness of the data  and the covariates), the bounds in \eqref{eq:theta12} hold with probability $1-o(1)$, which completes the proof of Theorem \ref{thm:stheta}.

	\subsection{Proof of Proposition \ref{proposition:RSC_random}}
	\label{sec:ppnLN_pf}
	
	Consider a fixed vector $\bm \eta\in\mathbb{R}^{d}$ with $\|\bm \eta\|_2=r \in(0,1]$ and set $L= L(r): =K r$, for a constant $K>0$ to be chosen. Since the function $\phi_{L}(u): =u^{2}I\{|u|\leq 2 L \} \leq u^2$ and $\psi$ is positive, 
	\begin{align*}
		\frac{1}{N}\sum_{i=1}^{N}\langle \bm \eta, \Z_i\rangle^{2}\psi(\langle \tht,\Z_i\rangle+ t\langle \bm \eta,\Z_i\rangle)\geq\frac{1}{N}\sum_{i=1}^{N}\psi(\langle \tht,\Z_i\rangle+ t\langle \bm \eta,\Z_i\rangle)\phi_{L}(\langle \bm \eta,\Z_i \rangle)\bm{1}\left\{|\langle \tht,\Z_i\rangle|\leq T\right\} , 
	\end{align*}
	where $T$ is second truncation parameter to be chosen. Note that on the events 
	$ \{ |\langle \bm \eta, \Z_i \rangle |\leq 2 L \}$  and $ \{ |\langle \bm \theta , \Z_i \rangle |\leq T \}$ one has, 
	$$|\langle \tht,\Z_i\rangle+t\langle \bm \eta,\Z_i \rangle| \leq T+ 2L \leq T+2K,$$ since $t, r \in [0, 1]$. This implies, 
	\begin{align}\label{eq:ppnpf1}
		\frac{1}{N} & \sum_{i=1}^{N}\langle \bm \eta, \Z_i\rangle^{2}\psi(\langle \tht,\Z_i\rangle+ t\langle \bm \eta,\Z_i\rangle) \nonumber \\ 
		& \geq  \frac{1}{N}\sum_{i=1}^{N}\psi(\langle \tht,\Z_i\rangle+ t\langle \bm \eta,\Z_i\rangle)\phi_{L}(\langle \bm \eta,\Z_i \rangle)\bm{1}\left\{|\langle\tht,\Z_i\rangle|\leq T\right\} \nonumber \\ 
		& \geq  \frac{\gamma}{N}\sum_{i=1}^{N}\phi_L(\langle \bm \eta, Z_i\rangle)\bm{1}\left\{|\langle \tht,\Z_i\rangle|\leq T\right\},
	\end{align}
	where $\gamma:=\min_{|u|\leq T+2K}\psi(u) > 0$. Based on this lower bound, to prove Proposition \ref{proposition:RSC_random} it is sufficient to show that for all $r \in(0,1]$ and for $\bm \eta\in\mathbb{R}^{d}$ with $\|\bm \eta\|_2= r$, 
	\begin{align}\label{eq:pf_lower_bound}
		\frac{1}{N}\sum_{i=1}^{N}\phi_{L(r)}(\langle \bm \eta,\Z_i\rangle)\bm{1}\left\{|\langle\tht,\Z_i\rangle|\leq T\right\}\geq c_1 r^{2}-c_2\mathcal R_N \|\bm \eta\|_1 r , 
	\end{align}
	holds with probability $1-o(1)$, where $L(r)=K r$ and some positive constants $c_1, c_2$. This is because when \eqref{eq:pf_lower_bound} holds by substituting $\|\bm \eta\|_2=r$ and using \eqref{eq:ppnpf1} gives, 
	\begin{align*}
		\frac{1}{N}\sum_{i=1}^{N}\langle \bm \eta, \Z_i\rangle^{2}\psi(\langle \tht,\Z_i\rangle+ t\langle \bm \eta,\Z_i\rangle)\geq \nu \|\bm \eta\|_2^{2}-c_0 \mathcal R_N \|\bm \eta\|_1^2, 
	\end{align*}
	where the constants $(\nu,c_0)$ depend on $(c_1,c_2,\gamma)$ and we use the inequality $\|\bm \eta\|_2\leq\|\bm \eta\|_1$. In fact, it suffices to prove \eqref{eq:pf_lower_bound} for $r=1$, that is, 
	\begin{align}\label{eq:lower_bound}
		\frac{1}{N}\sum_{i=1}^{N}\phi_{K}(\langle \bm \eta,\Z_i\rangle)\bm{1}\left\{|\langle\tht,\Z_i\rangle|\leq T\right\}\geq c_1 -c_2\mathcal R_N \|\bm \eta\|_1 ,  
	\end{align}
	holds with probability $1-o(1)$. This is because given any vector in $\R^d$ with $\|\bm \eta\|_2=r>0$, we can apply \eqref{eq:lower_bound} to the rescaled unit-norm vector $\bm \eta/r$ to obtain
	\begin{align}\label{eq:ppnpf2}
		\frac{1}{N}\sum_{i=1}^{N}\phi_{K}(\langle \bm \eta/r,\Z_i\rangle)\bm{1}\left\{|\langle\tht,\Z_i\rangle|\leq T\right\}\geq c_1-c_2 \mathcal R_N \frac{\|\bm \eta\|_1}{r} . 
	\end{align}
	Noting that $\phi_{K}(u/r)= \phi_{L(1)}(u/r)=(1/r)^{2}\phi_{L(r)}(u)$ and multiplying both sides of \eqref{eq:ppnpf2} by $r^{2}$ gives \eqref{eq:pf_lower_bound}.

	\subsubsection{Proof of \eqref{eq:lower_bound}}

	Define a new truncation function: 
	\begin{align*}
		\bar \phi_K(u)=u^{2}\bm{1}\left\{|u|\leq K\right\}+(u-2K)^{2}\bm{1}\left\{K<u\leq 2K\right\}+(u+2K)^{2}\bm{1}\left\{-2K\leq u<-K\right\}. 
	\end{align*}
	Note this function is Lipschitz with parameter $2K$. Moreover, since $\phi_K \geq \bar \phi_K$, to prove \eqref{eq:lower_bound} it suffices to show that for all vectors $\bm \eta \in \mathbb R^d$ of unit norm,  
	\begin{align}\label{eq:unit-norm}
		\frac{1}{N}\sum_{i=1}^{N}\bar \phi_K(\langle \bm \eta, \Z_i\rangle)\bm{1}\left\{|\langle \tht,\Z_i \rangle|\leq T\right\}\geq c_1-c_2 \cR_N \|\bm \eta\|_1 , 
	\end{align}
	holds with probability $1-o(1)$. To this end, for a given $\ell_1$ radius $b \geq 1$, define the random variable 
	\begin{align*}
		Y_n(b): =\sup_{\substack{\|\bm \eta\|_2=1 \\ \frac{b}{2} \leq \|\bm \eta\|_1\leq b}} \left|\frac{1}{N}\sum_{i=1}^{N}\bar \phi_K(\langle \bm \eta,\Z_i\rangle)\bm{1}\left\{|\langle\tht,\Z_i\rangle|\leq T\right\}-\mathbb{E}\left(\bar \phi_K(\langle \bm \eta, \Z \rangle)\bm{1}\left\{|\langle\tht,\Z\rangle|\leq T\right\}\right)\right| . 
	\end{align*}
	
	\begin{lemma}\label{lm:expectation12} Under the assumptions of Theorem \ref{thm:stheta} the following hold: 
		\begin{itemize}
			
			\item[$(1)$] By choosing $K^2=8 \kappa_2/\kappa_1$ and $T^{2}=8 \kappa_2 \Theta^{2}/\kappa_1$,  
			$$\mathbb{E}\left(\bar \phi_K(\langle\bm \eta,\Z\rangle)\bm{1}\left\{|\langle\tht,\Z \rangle|\leq T\right\}\right)\geq\frac{3}{4} \kappa_1.$$ 
			
			\item[$(2)$] There exist a positive constant $c_2$ such that
			\begin{align*}
				\mathbb{P}\left(Y_{n}(b)>\frac{1}{2} \kappa_1 + \frac{1}{2} c_2 \mathcal R_N b \right)\leq e^{-O_{\kappa_1, \kappa_2}(n)} .
			\end{align*}
			
		\end{itemize}
	\end{lemma}
	
	Lemma \ref{lm:expectation12} implies that with probability $1 - e^{-O_{\kappa_1, \kappa_2}(n)}$ for all $\bm \eta \in \R^d$ such that $\|\bm \eta\|_2=1$ and $\frac{b}{2} \leq \|\bm \eta\|_1\leq b$, we have 
	\begin{align}\label{eq:LNb}
		\frac{1}{N}\sum_{i=1}^{N}\bar \phi_K(\langle \bm \eta,\Z_i\rangle)\bm{1}\left\{|\langle\tht,\Z_i\rangle|\leq T\right\} & \geq \mathbb{E}\left(\bar \phi_K(\langle \bm \eta, \Z \rangle)\bm{1}\left\{|\langle\tht,\Z\rangle|\leq T\right\}\right) -\frac{1}{2} \kappa_1 - \frac{1}{2} c_2 \mathcal R_N b \nonumber \\ 
		& \geq \frac{1}{4} \kappa_1 - c_2 \mathcal R_N \|\bm \eta\|_1. 
	\end{align} 
	This establishes the bound \eqref{eq:unit-norm} with $c_1=\kappa_1/4$ for all vectors $\bm \eta \in \R^d$ with $\|\bm \eta \|_2 =1$ and $\frac{b}{2} \leq \|\bm \eta\|_1\leq b$, with probability $1 - e^{-O_{\kappa_1, \kappa_2}(n)}$.  
	
	We first prove Lemma \ref{lm:expectation12}. Then we will extend the bound in \eqref{eq:unit-norm} for all vectors  $\bm \eta \in \R^d$ with $\|\bm \eta \|_2 =1$ using a peeling strategy to complete the proof.

	\subsubsection*{Proof of Lemma \ref{lm:expectation12} $(1)$}
	
	To show Lemma \ref{lm:expectation12} (1) it suffices to show that
	\begin{align}\label{eq:kappa12}
		\mathbb{E}\left(\bar \phi_K(\langle \bm \eta, \Z\rangle)\right)\geq\frac{7}{8} \kappa_1 \quad \text{ and} \quad \mathbb{E}\left(\bar \phi_K(\langle \bm \eta, \Z\rangle)\bm{1}\left\{\left|\langle\tht,\Z\rangle \right|>T\right\}\right)\leq\frac{1}{8} \kappa_1.
	\end{align}
	Indeed, if these two inequalities hold, then 
	\begin{align*}
		\mathbb{E}\left(\bar \phi_K(\langle\bm \eta,\Z\rangle)\bm{1}\left\{\left|\langle\tht,\Z\rangle \right|\leq T\right\}\right)
		&
		=\mathbb{E}\left(\bar \phi_K(\langle\bm \eta,\Z\rangle)\right)-\mathbb{E}\left(\bar \phi_K(\langle\bm \eta,\Z\rangle)\bm{1}\left\{\left|\langle\tht,\Z\rangle \right|>T\right\}\right)\\
		&
		\geq\frac{7}{8} \kappa_1 - \frac{1}{8} \kappa_2 = \frac{3}{4} \kappa_1 . 
	\end{align*} 
	
	To prove first inequality in \eqref{eq:kappa12}, note that 
	\begin{align}\label{eq:kappa12_pf1}
		\mathbb{E}\left(\bar \phi_K(\langle\bm \eta,\Z\rangle)\right)\geq\mathbb{E}\left(\langle \bm \eta,\Z\rangle^{2}\bm{1}\left\{\left|\langle\bm \eta,\Z\rangle\right|\leq K \right\}\right)
		&
		=\mathbb{E}\left(\langle\bm \eta,\Z\rangle^{2}\right)-\mathbb{E}\left(\langle\bm \eta,\Z\rangle^{2}\bm{1}\left\{\left|\bm \eta,\Z\right|> K \right\}\right) \nonumber \\
		&
		\geq \kappa_1 - \mathbb{E}\left(\langle\bm \eta,\Z\rangle^{2}\bm{1}\left\{\left|\langle\bm \eta,\Z\rangle\right| > K \right\}\right) . 
	\end{align} 
	To lower bound the second term, applying the Cauchy-Schwarz inequality yields, 
	\begin{align*}
		\mathbb{E}\left(\langle\bm \eta,\Z\rangle^{2}\bm{1}\left\{\left|\langle\bm \eta,\Z\rangle\right| > K \right\}\right) & \leq \sqrt{\mathbb{E}\left(\langle \bm \eta,\Z\rangle^{4}\right)}\sqrt{\mathbb{P}\left(\left|\langle\bm \eta,\Z\rangle\right| > K \right)}\nonumber \\ 
		& \leq \frac{\mathbb{E}(\langle \bm \eta,\Z\rangle^{4})}{K^2} \tag*{(by Markov's inequality)} \nonumber \\ 
		& \leq \frac{\kappa_2}{K^{2}} , 
	\end{align*}
	using the assumption $\mathbb{E}(\langle \bm \eta,\Z \rangle^{4}) \leq \kappa_2$, for all 
	$\bm \eta$ such that $\| \bm \eta \|_2 \leq 1$. Therefore, setting $K^{2}=8 \kappa_2/\kappa_1$ and using \eqref{eq:kappa12_pf1} proves the first inequality in \eqref{eq:kappa12}.  
	
	Now, we turn to prove the second inequality in \eqref{eq:kappa12}. For this note that $\bar \phi_K\left(\langle\bm \eta,\Z\rangle\right) \leq \langle \bm \eta,\Z\rangle^{2} $ and by Markov's inequality, 
	\begin{align*}
		\mathbb{P}\left(\left|\langle\tht,\Z\rangle\right|\geq T\right)\leq\frac{\kappa_2 \|\tht\|_2^{4}}{T^{4}} . 
	\end{align*} 
	Then the Cauchy-Schwarz inequality implies, 
	\begin{align*}
		\mathbb{E}\left(\bar \phi_K\left(\langle\bm \eta,\Z\rangle\right)\bm{1}\left\{\left|\langle\tht,\Z\rangle\right|>T\right\}\right) \leq\frac{\kappa_2 \|\tht\|_2^{2}}{T^{2}} \leq\frac{\kappa_2 \Theta^2}{T^{2}}
	\end{align*}
	Thus setting $T^{2}=8 \kappa_2 \Theta^{2}/\kappa_1$ shows the second inequality in \eqref{eq:kappa12}.

	\subsubsection*{Proof of Lemma \ref{lm:expectation12} $(2)$} 
	
	We begin by recalling the functional Hoeffding's  inequality. Towards this, let $\mathcal{F}$ be a symmetric collection of  functions from $\R^d$ to $\R$, that is, if $f \in \mathcal{F}$, then $-f\in\mathcal{F}$. Suppose $X_1, X_2, \ldots, X_N$ are i.i.d. from a distribution supported on $\mathcal X \subseteq \R^d$ and let 
	$$Z:=\sup_{f\in\mathcal{F}}\left\{\frac{1}{N}\sum_{i=1}^{N}f(X_i)\right\} . $$  Then we have the following result: 
	
	\begin{lemma}[{\citep[Theorem 3.26]{wainwrighthighdimensional}}]\label{lm:fab}
		For each $f\in\mathcal{F}$ assume that there are real numbers $a_{f}\leq b_{f}$ such that $f(x)\in[a_{f},b_{f}]$ for all $x\in\mathcal{X}$. Then for all $\delta\geq0$, \begin{align*}
			\mathbb{P}\left(Z - \mathbb{E}(Z) \geq \delta\right)\leq\exp\left(-\frac{n\delta^{2}}{4L^{2}}\right) 
		\end{align*}
		where $L^{2}:=\sup_{f\in\mathcal{F}}\left\{(b_{f}-a_{f})^{2}\right\}$.
	\end{lemma} 
	

	For $\bm \eta \in \R^d$ such that $\|\bm \eta\|_2=1$ and $\|\bm \eta\|_1\leq b$, define 
	\begin{align*}
		f_{\bm \eta}(\Z): =  \bar \phi_K(\langle \bm \eta,\Z \rangle)\bm{1}\left\{|\langle\tht, \Z \rangle|\leq T\right\}-\mathbb{E}\left(\bar \phi_K(\langle \bm \eta, \Z \rangle)\bm{1}\left\{|\langle\tht, \Z \rangle|\leq T\right\}\right) 
	\end{align*} 
	Note that 
	$$Y_n(b) = \sup_{\substack{\|\bm \eta\|_2=1 \\ \frac{b}{2} \leq \|\bm \eta\|_1\leq b}} \left\{ \left| \frac{1}{N} \sum_{i=1}^N f_{\bm \eta}(\bm Z_i) \right| \right\} . $$ 
	Since $|f_{\bm \eta}(\bm z)| \leq K^2$, for any positive constant $c_3$ by Lemma \ref{lm:fab}, 
	\begin{align}\label{eq:Ynb}
		\mathbb{P}\left(Y_n(b) - \mathbb{E}(Y_n(b)) \geq c_3 \mathcal R_N b +\frac{1}{2} \kappa_1 \right)\leq e^{-c_4 n \mathcal R_N^2 b^{2} -c_4 n} \leq e^{-c_4 n} , 
	\end{align}
	where $c_4$ is a positive constant depending on $K$ and $c_3$. Now, suppose $\{\varepsilon_{i}\}_{i=1}^{n}$ be a sequence of i.i.d. Rademacher variables. Then a symmetrization argument \citep[Proposition 4.11]{wainwrighthighdimensional} implies that

	\begin{align*}
		\mathbb{E}\left(Y_n(b)\right)\leq 2 \mathbb{E}_{\Z, \bm \varepsilon}\left(\sup_{\substack{\|\bm \eta\|_2=1 \\ \frac{b}{2} \leq \|\bm \eta\|_1\leq b}}\left|\frac{1}{N}\sum_{i=1}^{N}\varepsilon_i\bar \phi_K\left(\langle\bm \eta,\Z_i\rangle\right)\bm{1}\left\{\left|\langle\tht,\Z_i\rangle\right|\leq T\right\}\right|\right).
	\end{align*}
	Since $\bm{1}\left\{\left|\langle\tht,\Z_i\rangle\right|\leq T\right\}\leq 1$ and $\bar \phi_K$ is Lipschitz with parameter $2K$, the contraction principle yields 
	\begin{align*}
		\mathbb{E}\left(Y_n(b)\right) & \leq 8K\mathbb{E}_{\Z,\bm \varepsilon}\left(\sup_{\substack{\|\bm \eta\|_2=1 \\ \frac{b}{2} \leq \|\bm \eta\|_1\leq b}}\left|\frac{1}{N}\sum_{i=1}^{N}\varepsilon_i\langle\bm \eta,\Z_i\rangle\right|\right) \nonumber \\ 
		& \leq 8K b\mathbb{E}_{\Z,\bm \varepsilon}\left(\sup_{\|\bar{\bm \eta} \|_1\leq 1} \left| \left \langle \bar{\bm \eta}, \frac{1}{N}\sum_{i=1}^{N} \varepsilon_i \Z_i \right \rangle\right|\right) \tag*{ (where $\overline{ \bm \eta}:= \bm \eta/\| \bm \eta\|_1)$ } \nonumber \\ 
		& = 8K b \mathbb{E}\left(\left\|\frac{1}{N}\sum_{i=1}^{N}\varepsilon_i\Z_i\right\|_{\infty}\right)     = 8 K b \mathcal R_N , 
	\end{align*}
	where the final step follows by applying H\"older's inequality. Then choosing $c_2= 32 K $ gives, 
	\begin{align*}
		\mathbb{P}\left(Y_{n}(b)>\frac{1}{2} \kappa_1 + \frac{1}{2} c_2 \mathcal R_N b \right) &\leq  \mathbb{P}\left(Y_{n}(b) - \mathbb E(Y_n(b)) > \frac{1}{2} \kappa_1 + 8 K b \mathcal R_N \right)   \nonumber \\ 
		& \leq e^{-O_{\kappa_1, \kappa_2}(n)} , 
	\end{align*}
	using \eqref{eq:Ynb} with $c_3= 8 K$. This proves Lemma \ref{lm:expectation12} (2). 

	\subsubsection*{Final Details: Peeling Strategy} 
	
	Recall from \eqref{eq:LNb} that we have proved \eqref{eq:unit-norm} holds for any fixed $b$ such that $\frac{b}{2} \leq \| \bm \eta \|_1 \leq b$ and $\| \bm \eta \|_2 \leq 1$. The final step in the proof of Proposition \ref{proposition:RSC_random} is to remove the restriction on the $\ell_1$ norm of $\bm \eta$. For this, we apply a peeling strategy as in the proof of \citet[Theorem 9.34]{wainwrighthighdimensional}. To this end, consider the set 
	\begin{align*}
		\mathbb{S}_{\ell}:=\left\{\bm \eta \in \mathbb{R}^{d} : 2^{\ell-1}\leq\frac{\left\|\bm \eta\right\|_1}{\left\|\bm \eta\right\|_2}\leq 2^{\ell}\right\}\cap\left\{\bm \eta \in \mathbb{R}^{d} : \left\|\bm \eta\right\|_2\leq 1\right\} , \end{align*} 
	for $\ell=1,\ldots,\lceil\log_2(\sqrt{d}) \rceil$. Then for $\bm \eta\in\mathbb{R}^{d}\cap\mathbb{S}_\ell$ by \eqref{eq:LNb}, 
	\begin{align*}
		\frac{1}{N}\sum_{i=1}^{N}\langle \bm \eta, \Z_i\rangle^{2}\psi(\langle \tht,\Z_i\rangle+ t\langle \bm \eta,\Z_i\rangle)\geq \nu \|\bm \eta\|_2^{2}-c_0\mathcal R_N\|\bm \eta\|_1^{2},
	\end{align*}
	with probability at least $1 - e^{-O_{\kappa_1, \kappa_2}(n)}$. Then by a union bound,  
	\begin{align*}
		\frac{1}{N}\sum_{i=1}^{N}\langle \bm \eta, \Z_i\rangle^{2}\psi(\langle \tht,\Z_i\rangle+ t\langle \bm \eta,\Z_i\rangle) \geq \nu\|\bm \eta\|_2^{2}-c_0\mathcal R_N\|\bm \eta\|_1^{2} , 
	\end{align*}
	for all $\bm \eta \in \mathbb{R}^{d}$ such that $\|\bm \eta\|_2\leq 1$, with probability at least $1- \lceil \log_2(d) \rceil e^{-O_{\kappa_1, \kappa_2}(n)} = 1 - o(1)$. This completes the proof of Proposition \ref{proposition:RSC_random}.

	\section{Proofs from Section \ref{sec:examples}}\label{sec:graphestimatepf}

	In this section, we prove the results stated in Section \ref{sec:examples}. We begin with the proof of Theorem \ref{thm:sparse} in Section \ref{proof_sparse}. Corollary \ref{inh} is proved in Section \ref{sec:inhomogeneouspf}. 
	
	\subsection{Proof of Theorem \ref{thm:sparse}}\label{proof_sparse}

	The proof of Theorem \ref{proof_sparse} follows along the same lines as in Theorem \ref{thm:estimate}, so to avoid repetition we sketch the steps and highlight the relevant modifications. As in the proof of Theorem \ref{proof_sparse}, the first step is to show the concentration of $\nabla L_N(\bg)$. Towards this, following the proof of  Lemma \ref{partialbetalips} shows that $Q_{r}$ (as defined in \eqref{eq:qr}) is $O_\beta(\mathrm{poly}(d_{\mathrm{max}})/N)$-Lipschitz, for each $r \in [\ell]$. Therefore, by arguments as in Lemma \ref{partialbetalips}, 
	\begin{align*}
		\p\left(\left|\sum_{i=1}^N \phi_{i}(\bm X)\right| \ge t\right) & \le \p\left(\max_{r \in [\ell]} |Q_{r}(\bm X)| \ge \frac{t \ell'}{\ell}\right) 
		\lesssim e^{- O_{\beta, M} (N t^2/\pl(d_{\max}))}.
	\end{align*}
	since $\ell= O(\log N)$ and $\ell'/\ell = \Theta(1)$. In this case, following the notations in the proof of Lemma \ref{lm:gradient}, we have $\ell/\ell' = \Theta (d_{\max})$. 
	Similarly, for $s \in [d]$, 
	\begin{align*}
		\p\left(\left|\sum_{i=1}^N \phi_{i, s}(\bm X)\right| \ge t\right) &\lesssim  e^{- O_{\beta, M} (N t^2/\pl(d_{\max}))}. 
	\end{align*} 
	Hence, choosing $\lambda:=\delta \pl(d_{\max})\sqrt{\log(d+1)/N}$ and $t := \lambda/2$ gives, for some constant $C$ depending only on $\beta$ and $M$, 
	$$\p\left(\left|\sum_{i=1}^N \phi_{i}(\bm X)\right| \ge \frac{\lambda}{2} \right) \le (d+1)^{-C\delta^2/4} \quad \text{and} \quad	 \p\left(\left|\sum_{i=1}^N \phi_{i,j}(\bm X)\right| \ge \frac{\lambda}{2} \right) \le (d+1)^{-C\delta^2/4},$$ 
	for all $j \in [s]$.  A final union bound over the $(d+1)$ coordinates now shows that
	\begin{align}\label{eq:derivativeconditionG}
		\bp\left(\norm{\nabla L_N(\bg)}_\infty > \frac{\lambda}{2} \right) = o(1), 
	\end{align} 
	where $\lambda:=\delta \pl(d_{\max})\sqrt{\log(d+1)/N}$ as above and the $o(1)$-term goes to infinity as $d \rightarrow \infty$. This establishes the concentration of the gradient. 
	
	Next, we need to show the strong concavity of the Hessian. To this end, first note that since $|E(G_N)| = O(N)$ and the number of non-isolated vertices of $G_N$ is $\Omega(N)$, there exist constants $L_1, L_2 > 0$, such that $|E(G_N)| \le L_1 N$, and the number of non-isolated vertices of $G_N$ is larger than $L_2 N$, for all $N$ large enough. For $D \geq 1$ define, 
	$$V_N(D) := \{v \in V(G_N):  d_v \in [1,D]\}.$$ 
	Note that 
	$$|V_N(D)| \ge N(1-(2L_1/D)) - N(1-L_2) = N(L_2 - (2L_1/D)), $$
	since $G_N$ has at least $N(1- (2L_1/D))$ vertices with degree not exceeding $D$, among which at most $N(1-L_2)$ are isolated. Hereafter, we choose $D := \lceil 4 L_1 / L_2\rceil$, so that $|V_N(D)| \ge L_2 N/2$. Then, with notations as in \eqref{second_taylor} we have, 
	\begin{align}\label{eq:D}
		& L_N(\hat{\bg}) - L_N(\bg) - \nabla L_N(\bg) ^\top (\hat{\bg}-\bg) \nonumber \\ 
		&= \frac{1}{2} (\hat{\bg}-\bg)^\top \nabla^2 L_N(\underline{\bg}) (\hat{\bg}-\bg) \nonumber  \\ 
		& = \frac{1}{2N}\sum_{i=1}^N \frac{(\hat{\bg}-\bg)^\top \bm U_i \bm U_i^\top (\hat{\bg}-\bg)}{\cosh^2(\underline{\beta} m_i(\bs) + \underline{\bh}^\top \bm Z_i)} \nonumber  \\ 
		&\ge \frac{1}{2N}\sum_{i \in V_N(D)} \frac{(\hat{\bg}-\bg)^\top \bm U_i \bm U_i^\top (\hat{\bg}-\bg)}{\cosh^2(\underline{\beta} m_i(\bs) + \underline{\bh}^\top \bm Z_i)} \nonumber  \\
		& \ge \frac{1}{2N}\sum_{i\in V_N(D)} \frac{(\hat{\bg}-\bg)^\top \bm U_i \bm U_i^\top (\hat{\bg}-\bg)}{\cosh^2(|\beta| D  + (D+M) \|\hat{\bg}-\bg\|_1  + s M \Theta)} , 
	\end{align}
	where the last step uses the bound in \eqref{eq:Theta} and the bound $|\underline{\beta} m_i(\bs)| \leq |\underline{\beta}| |m_i(\bs)| \leq  |\underline{\beta}| D \leq |\beta| D + D \|\hat \bg-\bg\|_1 $. Next, using the bound $|V_N(D)| \ge L_2 N/2$ in \eqref{eq:D} gives, 
	\begin{align}\label{eq:UU}
		& L_N(\hat{\bg}) - L_N(\bg) - \nabla L_N(\bg) ^\top (\hat{\bg}-\bg) \nonumber \\ 
		& \ge \frac{L_2}{4 \cosh^2(|\beta|D  + (D+M)\|\hat{\bg}-\bg\|_1 + s M \Theta)} \cdot \frac{1}{|V_N(D)|} \sum_{i\in V_N(D)} (\hat{\bg}-\bg)^\top \bm U_i \bm U_i^\top (\hat{\bg}-\bg) \nonumber \\ 
		& = \frac{L_2}{4 \cosh^2(|\beta|D  + (D+M)\|\hat{\bg}-\bg\|_1 + s M \Theta)} (\hat{\bg}-\bg)^\top \tilde{\bm G} (\hat{\bg}-\bg), 
	\end{align} 
	where 
	\begin{equation*}
		\tilde{\bm G} :=  \frac{1}{|V_N(D)|}\left(
		\begin{array}{cc}
			\tilde{\bm m}^\top \tilde{\bm m} & \tilde{\bm m}^\top \tilde{\bm Z}\\
			\tilde{\bm Z}^{\top} \tilde{\bm m} & \tilde{\bm Z}^\top \tilde{\bm Z} 
		\end{array}
		\right) , 
	\end{equation*} 
	with $\tilde{\bm m} := (m_i(\bm X))_{i\in V_N(D)}^\top$ and $\tilde{\bm Z} = (\bm Z_i)_{i \in V_N(D)}^\top$. The calculation in \eqref{eq:UU} implies that in order to establish the strong concavity of the Hessian in the setting of Theorem \ref{thm:sparse}, we need to show $\lambda_{\mathrm{min}}(\tilde{\bm G}) \gtrsim 1$ with probability going to $1$, where 
	$\lambda_{\mathrm{min}}(\tilde{\bm G})$ denotes the minimum eigenvalue of $\tilde{\bm G}$. For this step, we follow the steps of Lemma \ref{mineiglem} with the matrix $\bm F$ replaced by $\tilde{\bm F} := \bm I - \tilde{\bm Z}(\tilde{\bm Z}^\top \tilde{\bm Z})^{-1}\tilde{\bm Z}^\top$. Then, repeating the proof of \eqref{fs22} in Lemma \ref{meanlb}, with $\bm A$ replaced by   $\tilde{\bm A} := \bm A|_{V_N(D) \times [N]}$, we can find a set $J \subseteq [N]$ such that 
	$$\e\left(\|\tilde{\bm F} \tilde{\bm m}\|_2^2 \Big| \bm X_{J^c} \right) \gtrsim \Upsilon^2 \left(\|\tilde{\bm A}\|_F^2 - d \cdot \pll(N)\right),$$ 
	since $\|\tilde{\bm A}\|_2 \le \|\bm A\|_\infty = O(\pll(N))$. This implies, 
	\begin{align}\label{eq:meanG}
		\e\left(\|\tilde{\bm F} \tilde{\bm m}\|_2^2 \Big| \bm X_{J^c} \right) \gtrsim N, 
	\end{align} 
	since $\|\tilde{\bm A}\|_F^2 \ge |V_N(D)| \gtrsim N$ (because every vertex in $V_N(D)$ has degree at least 1) and $d = o(N)$. The final step is to establish that $\|\tilde{\bm F} \tilde{\bm m}\|$ concentrates around $\|\tilde{\bm F} \tilde{\bm m}\|_2^2 \Big| \bm X_{J^c}$, conditional on $\bm X_{J^c}$. This follows by repeating the proof of Lemma \ref{concmeanlb}, which introduces an extra $1/\pll(N)$ factor in each of the two exponential terms in the RHS of \eqref{eq:concmean}, since $\|\tilde{\bm A}\|_2 \le \|\bm A\|_\infty = O(\pll(N))$. This, combined with \eqref{eq:meanG}, shows 
	\begin{align}\label{eq:concentrationG}
		\p(\lambda_{\min}(\tilde{\bm G}) \ge C) \ge 1-e^{-\Omega(N/\mathrm{polylog}(N))}, 
	\end{align}
	for some constant $C > 0$. The proof of Theorem \ref{thm:sparse} can be now completed using \eqref{eq:derivativeconditionG} and \eqref{eq:concentrationG}, as in Theorem \ref{thm:estimate}. 
	
	\subsection{Proof of Corollary \ref{inh}}\label{sec:inhomogeneouspf}

	To prove Corollary \ref{inh} we verify that the hypotheses of Theorem \ref{thm:sparse} are satisfied. Note that we can write
	$$|E(G_N)| = \sum_{1\le u < v \le N} B_{uv} ,$$
	where $B_{uv} \sim \mathrm{Ber}(p_{uv})$, and $\{B_{uv} \}_{1\le u<v\le N}$ are independent. This implies, 
	$\e |E(G_N)| =  O(N)$ and $\mathrm{Var}(|E(G_N)|) = O(N)$, 
	since $\sup_{1 \leq i , j \leq N} p_{ij} = O(1/N)$ by \eqref{eq:P1}. Hence, by Chebyshev's inequality, $$|E(G_N)| \le \e |E(G_N)| + N = O(N),$$ with probability $1- o(1)$.
	
	Next, we will show that $d_{\max} = \te(1)$ holds with high probability. To this end, define $\eta_u=\sum_{v \ne u} p_{uv}$, for $1 \leq u \leq N$. Clearly, by assumption \eqref{eq:P1}, $\max_{1 \leq u \leq N} \eta_u = O(1)$. Next, we establish that $\max_{1\le u \le N} \eta_u = \Omega(1)$. Towards this, note that by assumption \eqref{eq:P2} there exists $\varepsilon \in (0,1)$ and $u \in V(G_N)$ such that,  
	$$\limsup_{N\rightarrow \infty} \sum_{v=1}^N \log\left(1- p_{uv} \right) < \log \varepsilon.$$ Next, using the inequality $\log(1-x) \ge -x/(1-x)$, for all $x < 1$, and taking $N$ large enough such that $\sup_{u, v} p_{uv} < 1/2$, gives 
	$$\limsup_{N\rightarrow \infty} \left(-2 \sum_{v=1}^N p_{uv} \right) \le \log \varepsilon \quad \implies \quad \liminf_{N \rightarrow \infty} \eta_u \ge \frac{1}{2} \log \left(\frac{1}{
		\varepsilon}\right).$$ This shows that $\max_{1\le u \le N} \eta_u = \Omega(1)$. Therefore, by Proposition 1.11 in \citet{eigenerdos}, $d_{\max} \le O(\log N)$ with probability $1-o(1)$.

	Finally, we show that the number of non-isolated vertices of $G_N$ is $\Omega(N)$ with high probability. To this end, for each $v \in V(G_N)$, define $Y_v := \bm{1}\{d_v = 0\}$. Then, $I_N:=\sum_{v=1}^N Y_v$ is the total number of isolated vertices of $G_N$. Note that $Y_v \sim \mathrm{Ber}(\prod_{u=1}^N (1- p_{uv}))$. Therefore, by \eqref{eq:P2}, 
	$$\e (I_N) = \sum_{u=1}^N \prod_{v=1}^N \left(1- p_{uv} \right) \leq \alpha N,$$ for some $\alpha \in (0,1)$ and all large $N$ enough. Next, note that for any two distinct vertices $u, v\in V(G_N)$, 
	\begin{align*}
		\mathrm{Cov}(Y_u, Y_v) & =  p_{uv} \left(1-p_{uv}\right) \prod_{w \notin \{u, v\}} \left[\left(1- p_{u w} \right) \left(1- p_{v w} \right) \right] \le p_{uv} =  O\left(\frac{1}{N}\right).
	\end{align*}
	Similarly, it can be checked that $\mathrm{Var}(Y_v) = O(1)$, for $1 \leq v \leq N$. This implies, $\mathrm{Var}(I_N) = O(N)$. Hence, by Chebyshev's inequality,  
	$$\p\left(I_N \ge \left(\frac{1+\alpha}{2}\right) N \right) \le \p\left(I_N \ge \e (I_N) + \left(\frac{1-\alpha}{2} \right) N \right) \le \frac{\mathrm{Var}(I_N)}{\left(\frac{1-\alpha}{2} \right)^2 N^2} = O\left(\frac{1}{N}\right).$$
	This shows that the number of non-isolated vertices of $G_N$ is at least $(1-\alpha)N/2$ with probability $1-o(1)$. This completes the verification of the hypotheses of Theorem \ref{thm:sparse} and hence, Corollary \ref{inh} follows from Theorem \ref{thm:sparse}.

	\section{Proofs of Technical Lemmas}

	In this section we collect the proofs of various technical lemmas. The section is organized as follows: In Appendix \ref{sec:partialbetalipspf} we prove Lemma \ref{partialbetalips}. The proof of Lemma \ref{combinatorial1} is given in Appendix \ref{sec:combinatorial1pf}. In Appendix \ref{sec:variancepf} we prove a variance lower bound for linear functions. 
	
	\subsection{Proof of Lemma \ref{partialbetalips}}
	\label{sec:partialbetalipspf}
	
	To begin with, recall from \eqref{eq:qr} that
	$$Q_{r}(\bs) = -\frac{1}{N} \sum_{i\in I_{r}} m_i(\bs)\left[X_i - \tanh(\beta m_i(\bs) +\bh^\top \bm Z_i)\right].$$ 
	Hence, for any two $\bs,\bs' \in \{-1,1\}^N$, 
	\begin{align}\label{lips1}
		|Q_{r}(\bs) - Q_{r}(\bs')| = T_1 + T_2, 
	\end{align}
	where 
	\begin{align}\label{eq:T1T2}
		T_1 & =  \frac{1}{N} \left|\sum_{i\in I_{r}}\left\{ m_i(\bs) X_i - m_i(\bs') X_i'\right\} \right| \nonumber \\ 
		T_2 & = \frac{1}{N} \left|\sum_{i\in I_{r}}  \left\{m_i(\bs) \tanh(\beta m_i(\bs) +\bh^\top \bm Z_i) - m_i(\bs') \tanh(\beta m_i(\bs') +\bh^\top \bm Z_i)\right\} \right|. 
	\end{align} 
	Now, assume that $\bs$ and $\bs'$ differ only in the $k$-th coordinate, for some $k \in [N]$. Then 	
	\begin{align}\label{lipsfirst}
		T_1 &= \left|\sum_{i\in I_{r}} \sum_{j=1}^N a_{ij} (X_i X_j - X_i' X_j') \right|\nonumber\\&\le \left|\sum_{i\in I_{r}} a_{i k} (X_i X_k - X_i'X_k')\right| + \left|\sum_{j=1}^N a_{k j} (X_k X_j - X_k'X_j')\right|\nonumber\\&\le 2 \sum_{i\in I_{r}} |a_{ik}| + 2 \sum_{j=1}^N |a_{kj}|\nonumber\\&\leq 4\left|\sum_{i=1}^N a_{i k}\right| \leq 4\|\bm A\|_1 \leq 4 \quad\textrm{(by \eqref{eq:Anorm})}.
	\end{align}
	Next, we proceed to bound $T_2$. Towards this note that 	
	\begin{align}\label{lipssecond}
		T_2 & \leq T_{21} + T_{22}, 
	\end{align}
	where 
	\begin{align} 
		T_{21} & :=  \left|\sum_{i\in I_r} m_i(\bs) \left\{ \tanh(\beta m_i(\bs) +\bh^\top \bm Z_i) -  \tanh(\beta m_i(\bs') +\bh^\top \bm Z_i)\right\} \right|\nonumber\\ 
		T_{22} & :=  \left|\sum_{i\in I_r} a_{ik}(X_k-X_k')  \tanh(\beta m_i(\bs') +\bh^\top \bm Z_i) \right| , \nonumber 
	\end{align} 
	since $m_i(\bs) - m_i(\bs') = a_{ik}(X_k - X_k')$. Hence, using $|\tanh x -\tanh y | \leq |x-y|$ gives, 
	\begin{align}\label{eq:T21}
		T_{21} & \leq |\beta| \sum_{i\in I_r} |m_i(\bs)| \left| m_i(\bs) - m_i(\bs') \right| \leq  2|\beta| \sum_{i\in I_r} |m_i(\bs)| |a_{ik}| \leq 2|\beta| \|\bm A\|_\infty  \|\bm A\|_1  \leq 2|\beta|.
	\end{align} 
	and using $\tanh x \leq 1$ gives, 
	\begin{align}\label{eq:T22}
		T_{22} & \leq  2\sum_{i\in I_r} |a_{ir}| \leq  2\|\bm A\|_1 \leq  2.
	\end{align}
	Combining \eqref{lips1}, \eqref{eq:T1T2}, \eqref{lipsfirst}, \eqref{lipssecond}, \eqref{eq:T21}, \eqref{eq:T22} the result in Lemma \ref{partialbetalips} (1) follows.

	Next, we prove (2). To begin with, recall from \eqref{eq:qrs} that 
	$$Q_{r, s}(\bs) = -\frac{1}{N} \sum_{i\in I_r} \bm Z_{i,s}\left[X_i - \tanh(\beta m_i(\bs) +\bh^\top \bm Z_i)\right].$$ 
	Hence, for any two $\bs,\bs' \in \{-1,1\}^N$ differing in the $k$-th coordinate only, 
	\begin{align*}
		|Q_{r, s}(\bs) - Q_{r, s}(\bs')| &\leq \frac{1}{N} \left|\sum_{i\in I_r}  Z_{i,s}(X_i-X_i')\right|\\ &+ \frac{1}{N} \left|\sum_{i\in I_r}   Z_{i,s} \left[\tanh(\beta m_i(\bs) +\bh^\top \bm Z_i) - \tanh(\beta m_i(\bs') +\bh^\top \bm Z_i)\right]\right|\nonumber\\&\le \frac{2|Z_{k,s}|}{N} + \frac{|\beta|}{N}\left|\sum_{i\in I_r}  Z_{i,s} (m_i(\bs) - m_i(\bs'))\right|\nonumber\\& \leq \frac{2|Z_{k,s}|}{N} + \frac{2|\beta|}{N}\sum_{i=1}^N \left|  Z_{i,s} a_{ik}\right| , 
	\end{align*}
	as desired. 
	%
	%
	
	\subsection{Proof of the Lemma \ref{combinatorial1}}
	\label{sec:combinatorial1pf}

	Suppose that $\bm X$ comes from the model:
	\begin{align}\label{eq:betahD}
		\p_{\beta,\bm h}(\bm X) ~\propto~\exp\left(\sum_{i=1}^N h_i X_i +  \bm X^\top \bm D \bm X\right),
	\end{align}
	which is assumed to be $(R, \Upsilon)$-Ising.We will apply Lemma 17 in \citet{cd_ising_estimation} on the matrix $$\bm D :=  ((d_{ij}))_{1 \leq i, j \leq N} := \bm A/R~,$$ where $\bm A$ is the interaction matrix corresponding to the distribution of $\bm X$. Note that $\bm D$ satisfies the hypotheses of Lemma 17 in \citet{cd_ising_estimation}. For $\eta \in (0, R)$, define $\eta' := \eta/R$. This ensures that $\eta' \in (0,1)$. By Lemma 17 in \citet{cd_ising_estimation}, there exist subsets $I_1,\ldots, I_\ell \subseteq [N]$ with $\ell \lesssim  R^2\log N/\eta^2$, such that for all $1\le i \le N$,  $|\{j \in \ell: i\in I_j\}| = \lceil \eta \ell/8R\rceil$, and for all $j \in \ell$, 
	\begin{align}\label{eq:DIj}
		\|\bm D|_{I_j \times I_j}\|_\infty \le \eta' \implies \|\bm A|_{I_j \times I_j}\|_\infty \le \eta, 
	\end{align}
	where for a matrix $\bm M = ((m_{ij})) \in \mathbb{R}^{s\times t}$ and for sets $S\subseteq \{1,\ldots,s\}, T\subseteq \{1,\ldots,t\}$, we define $\bm M|_{S\times T} := ((m_{ij}))_{i\in S, j \in T} \in \mathbb{R}^{|S|\times |T|}$.  
	
	Now, for $j \in [\ell]$,  
	\begin{align}\label{lastpiece}
		& \frac{\mathbb{P}(\bm X_{I_j}  = \bm y|\bm X_{I_j^c} = \bm x_{-I_j})}{\mathbb{P}(\bm X_{I_j} = \bm y'|\bm X_{I_j^c} = \bm x_{-I_j})} \nonumber \\ 
		& = \frac{  \exp\left( \bm y_{I_j}^\top \bm D|_{I_j\times I_j} \bm y_{I_j} + \sum_{u\in I_j, v\notin I_j} d_{uv} y_u x_v  + \sum_{i\in I_j} h_i y_i \right)}{\exp\left(\bm y_{I_j}^\top \bm D|_{I_j\times I_j} \bm y_{I_j} + \sum_{u\in I_j, v\notin I_j} d_{uv} y_u' x_v  +  \sum_{i\in I_j} h_i y_i' \right)}. 
	\end{align} 
	Observe that the RHS of \eqref{lastpiece} is the probability mass function of  an Ising model $\mu$ on $\{-1,1\}^{|I_j|}$ with interaction matrix $\bm D|_{I_j\times I_j}$ and external magnetic field term at site $i$ given by $$h_i' = \sum_{v \notin I_j} D_{iv} x_v + h_i.$$ Recall from \eqref{eq:DIj} that $\| \bm A|_{I_j \times I_j }\|_\infty \leq \eta$. 
	The next step is to show that if $\bm Y \sim \mu$, then $$\min_{1\le u\le |I_j|} \mathrm{Var}(Y_u|\bm Y_{-u}) \ge \Upsilon~.$$ Towards this, note that for any $1\le u \le |I_j|$,  
	$$\p(Y_u = 1| \bm Y_{-u} = \bm y_{-u}) = \p(X_{I_j^u} = 1| \bm X_{I_j\setminus \{I_j^u\}} = \bm y_{-u}, \bm X_{-I_j} = \bm x_{-I_j})$$ where $I_j^u$ is the $u$-$\mathrm{th}$ smallest element of $I_j$.  Hence,
	$$\mathrm{Var}(Y_u|\bm Y_{-u} = \bm y_{-u}) = \mathrm{Var}(X_{I_j^u} | \bm X_{I_j\setminus \{I_j^u\}} = y_{-u}, \bm X_{-I_j} = \bm x_{-I_j}) \ge  \Upsilon,$$ 
	since \eqref{eq:betahD} is a $(R, \Upsilon)$-Ising model. This implies that $X_{I_j}|X_{I_j^c}$ is $(\eta, \Upsilon)$-Ising model, for $j \in [\ell]$.

	To prove \eqref{eq:a}, let us pick $j \in [\ell]$ uniformly at random, and note that for any vector $\bm a$, 
	$$\mathbb{E} \left(\sum_{i \in I_j} a_i \right) = \sum_{i=1}^N a_i \mathbb{P}(I_j \ni i) = \sum_{i=1}^N a_i \frac{\lceil \eta \ell/8R\rceil}{\ell}  \ge \frac{\eta}{8R} \sum_{i=1}^N a_i.$$
	This means that there is a fixed sample point $j$, such that $$\sum_{i \in I_j} a_i \ge \frac{\eta}{8R} \sum_{i=1}^N a_i.$$ This completes the proof of Lemma \ref{combinatorial1}.

	\subsection{Variance Lower Bound for Linear Functions}
	\label{sec:variancepf}
	
	In this section we derive a variance lower bound for linear functions in $(R, \Upsilon)$-Ising models. This follows the proof of Lemma 10 in \citet{cd_ising_estimation} adapted to our setting. We begin with the following definition:  For two probability measures $\mu$ and $\nu$, the $\ell_1$-Wasserstein distance is defined as: 
	$$W_1(\mu,\nu) := \min_{\pi \in \mathscr{C}_{\mu,\nu}} \e_{(\bm U, \bm V)\sim \pi} \|\bm U - \bm V \|_1,$$ 
	where $\mathscr{C}_{\mu,\nu}$ denotes the set of all couplings of the probability measures $\mu$ and $\nu$.

	\begin{lemma}\label{lemma14}
		Let $\bs \in \{-1,1\}^N$ be a sample from an $(R,\Upsilon)$- Ising model for some $R < 1/8$. Then, for each $i \in [N]$, 
		$$W_1\left(\bp_{\bm X_{-i}|X_i=1},~\bp_{\bm X_{-i}|X_i=-1} \right) \le \frac{16R}{1-8R}.$$ 
	\end{lemma} 
	
	\begin{proof}
		It follows from Lemma 4.9 in \citet{kostis3}, that there exists a coupling $\pi$ of the conditional measures $\bp_{\bm X_{-i}|X_i=1}$ and $\bp_{\bm X_{-i}|X_i=-1}$, such that:
		\begin{equation}\label{is}
			\mathbb{E}_{(\bm U, \bm V)\sim \pi} \left[d_H(\bm U, \bm V)\right] \le \frac{\alpha}{1-\alpha},
		\end{equation}
		where $d_H$ denotes the Hamming distance and $\alpha$ the Dobrushin coefficient.  By Lemma 4.4 in \citet{scthesis}, we know that Dobrushin's interdependence matrix is given by $8\bm D$. Hence, it follows from \eqref{is} that (Dobrushin's coefficient in Theorem 2.3 in \citet{kostis3} is given by $\alpha = 8\|\bm D\|_2$, see Theorem 4.3 in \citet{scthesis}),
		\begin{equation*}
			W_1\left(\bp_{\bm X_{-i}|X_i=1},~\bp_{\bm X_{-i}|X_i=-1} \right) \le 2~\mathbb{E}_{(\bm U, \bm V)\sim \pi} \left[d_H(\bm U, \bm V)\right] \le \frac{16 \|\bm D\|_2}{1- 8\|\bm D\|_2} \le \frac{16R}{1-8R}.
		\end{equation*}
		which completes the proof of the lemma. 
	\end{proof}

	Using the above lemma, we now prove the desired variance lower bound: 
	
	\begin{lemma}\label{lemma10}
		Let $\bs$ be a sample from an $(R,\Upsilon)$-Ising model for some $R >0$. Then for any vector $\bm a \in \mathbb{R}^N$,
		$$\mathrm{Var}(\bm a^\top \bs) \gtrsim \frac{\|\bm a\|_2^2\Upsilon^2}{R}.$$  
	\end{lemma}  
	
	\begin{proof}
		First, consider the case $R \le \Upsilon/32 \le 1/32$. In this case, $1-8R \ge 3/4$, so 
		\begin{equation}\label{med1}
			\frac{8R}{1-8R} \le \frac{\Upsilon}{3}
		\end{equation}
		Now, it follows from Lemma \ref{lemma14} that for every $i$,
		\begin{align}
			\sum_{j\in [N]\setminus \{i\}} \left|\e \left(X_j|X_i=1\right) - \e\left(X_j|X_i=-1\right)\right| &\le \sum_{j\in [N]\setminus \{i\}} W_1\left(\p_{X_j|X_i=1}, \p_{X_j|X_i=-1})\right)\nonumber\\ 
			&\le W_1\left(\bp_{\bm X_{-i}|X_i=1},~\bp_{\bm X_{-i}|X_i=-1} \right) \nonumber \\ 
			& \le \frac{16 R}{1-8 R}. \label{med2}
		\end{align}
		
		\noindent Next, note that:
		\begin{align}
			\mathrm{Cov}(X_i,X_j) &= \e[(X_i-\e X_i) X_j] \nonumber \\ 
			& = \e\left[(X_i-\e X_i)\e(X_j|X_i)\right]\nonumber\\
			&= \p(X_i=1) (1-\e X_i) \e(X_j|X_i=1) - \p(X_i=-1) (1+\e X_i) \e(X_j|X_i=-1) \nonumber \\
			&= \frac{1+\e X_i}{2} (1-\e X_i) \e(X_j|X_i=1) - \frac{1-\e X_i}{2}(1+\e X_i) \e(X_j|X_i=-1)\nonumber\\&= \frac{1}{2} [1-(\e X_i)^2]\left[\e(X_j|X_i=1) - \e(X_j|X_i=-1)\right]\nonumber\\\label{med3}&\le \frac{1}{2} \left[\e(X_j|X_i=1) - \e(X_j|X_i=-1)\right]
		\end{align}
		Combining \eqref{med1}, \eqref{med2} and \eqref{med3}, we have for every $i$,
		\begin{equation}\label{med4}
			\sum_{j\in [N]\setminus \{i\}} \left|\mathrm{Cov}(X_i,X_j)\right| \le \frac{8 R}{1-8 R} \le \frac{\Upsilon}{3}.
		\end{equation}
		Hence, we have by \eqref{med4},
		\begin{align}
			\mathrm{Var}(\bm a^\top \bs) &\ge \sum_{i=1}^N a_i^2 \mathrm{Var}(X_i) - \sum_{i\ne j} \left|a_i a_j \mathrm{Cov}(X_i,X_j)\right|\nonumber\\&\ge \sum_{i=1}^N a_i^2 \mathrm{Var}(X_i) - \sum_{i\ne j} \frac{(a_i^2+a_j^2)\left|\mathrm{Cov}(X_i,X_j)\right|}{2}\nonumber\\&= \sum_{i=1}^N a_i^2\left[\mathrm{Var}(X_i) - \sum_{j\in [N]\setminus \{i\}} \left|\mathrm{Cov}(X_i,X_j)\right| \right]\nonumber\\&\ge\sum_{i=1}^N a_i^2 \left(\Upsilon - \frac{\Upsilon}{3}\right) \nonumber \\ 
			& = \frac{2\Upsilon}{3} \|\bm a\|_2^2 \ge \frac{\|\bm a\|_2^2 \Upsilon^2}{R}\cdot \frac{2R}{3} \gtrsim \frac{\|\bm a\|_2^2 \Upsilon^2}{R}.
			\label{c111}
		\end{align}
		
		Now  consider the case $R > \Upsilon/32$. By Lemma \ref{combinatorial1}, we choose a subset $I$ of $[N]$ such that conditioned on $\bs_{I^c}$, $\bs_I$ is a $(\Upsilon/32,\Upsilon)$- Ising model, and
		\begin{equation}\label{rtour}
			\|\bm a_I\|_2^2 \ge \frac{\Upsilon}{256R}\|\bm a\|_2^2.
		\end{equation}
		Hence, we have from \eqref{c111} and \eqref{rtour},
		\begin{equation}\label{pl}
			\mathrm{Var}\left(\bm a^\top \bs | \bs_{I^c}\right) = \mathrm{Var}\left(\bm a_I^\top \bs_I | \bs_{I^c}\right) \ge \frac{2\Upsilon \|\bm a_I\|_2^2}{3} \ge \frac{\Upsilon^2 \|\bm a\|_2^2}{384R}.
		\end{equation}
		Lemma \ref{lemma10} now follows from \eqref{pl} on observing that $\mathrm{Var}(\bm a^\top \bs) \ge \mathbb{E} \left[\mathrm{Var}\left(\bm a^\top \bs | \bs_{I^c}\right)\right]$.
	\end{proof} 
	
	\subsection{Lipschitz Condition for the Gradient of $L_N$}

	In this section we show that $\nabla L_N$, the gradient of the pseudo-likelihood loss function $L_N$ defined in \eqref{eq:LNbetatheta}, is Lipschitz.
	
	\begin{lemma} \label{lm:LNgamma12} Suppose the design matrix  $\bm Z := (\bm Z_1,\ldots, \bm Z_N)^\top$ satisfies $\lambda_{\max}\left(\frac{1}{N} \bm Z^\top \bm Z \right) = O(1)$. Then there exists a constant $L>0$ such that for any two $\bg_1, \bg_2 \in \R^{d+1}$, 
		\begin{align*} 
			|| \nabla L_N(\bg_1) - \nabla L_N(\bg_2) ||_2  \leq L || \bg_1 - \bg_2 ||_2 .
		\end{align*} 
	\end{lemma}
	
	\begin{proof} 
		For any two $\bg_1, \bg_2 \in \R^{d+1}$, there exists $\bg^* \in \R^{d+1}$ such that 
		\begin{align*} 
			\nabla L_N(\bg_1) - \nabla L_N(\bg_2) = ( \bg_1 - \bg_2 )^\top \nabla^2 L_N(\bg^*). 
		\end{align*} 
		Therefore, 
		\begin{align}\label{eq:LNgamma12pf}
			|| \nabla L_N(\bg_1) - \nabla L_N(\bg_2) ||_2  \leq \sup_{\bg^* \in \R^{d+1}} \lambda_{\max}(\nabla^2 L_N(\bg^*) ) || \bg_1 - \bg_2 ||_2 ,  
		\end{align}  
		where $\lambda_{\max}( \nabla^2 L_N(\bg^*) )$ is the largest eigenvalue of the Hessian matrix $\nabla^2 L_N(\bg^*)$. Recall from \eqref{second_taylor} that 
		\begin{align*}	
			\nabla^2 L_N(\bg^*)  & = \frac{1}{N} \sum_{i=1}^N \frac{\bm U_i \bm U_i^\top }{\cosh^2(\beta^* m_i(\bs) + (\bh^*)^\top \bm Z_i)} ,   
		\end{align*}
		where $\bg^* = (\beta^*, (\bh^*)^\top)^\top$ and	 $\bm U_i := (m_i(\bs),\bm Z_i^\top)^\top$, for $1 \leq i \leq N$.  Using $\mathrm{sech}^2(x) \leq 1$ it follows that 
		\begin{align}\label{eq:LN2maximum}	
			\sup_{\bg^* \in \R^{d+1}} \lambda_{\max} ( \nabla^2 L_N(\bg^*) ) & \leq \lambda_{\max}\left( \frac{1}{N} \sum_{i=1}^N \bm U_i \bm U_i^\top \right),   
		\end{align}	
		Now, suppose $\bm w= (u, \bm v^\top)^\top \in \R^{d+1}$ be such that $|| \bm w ||_2 = 1$. Then 
		\begin{align*} 
			\frac{1}{N} \sum_{i=1}^N \bm w^\top \bm U_i \bm U_i^\top \bm w  = \frac{1}{N} \sum_{i=1}^N (\bm w^\top \bm U_i )^2 	 & = \frac{1}{N} \sum_{i=1}^N ( u m_i(\bm X) + \bm  v^\top \bm Z_i)^2 \nonumber \\  
			& \lesssim \frac{1}{N} \sum_{i=1}^N \left \{ u^2 (m_i(\bm X))^2 + (\bm  v^\top \bm Z_i)^2 \right \} \nonumber \\  
			& \leq \frac{1}{N} \sum_{i=1}^N m_i(\bm X)^2 + \frac{1}{N} \sum_{i=1}^N \bm  v^\top \bm Z_i \bm Z_i^\top \bm v .
		\end{align*} 
		Note that, since $|m_i(\bs)| \leq  || \bm A||_\infty \leq 1$ (by \eqref{eq:Anorm}), for $1 \leq i \leq N$, we have $\frac{1}{N} \sum_{i=1}^N m_i(\bm X)^2 \leq 1$. Moreover, $\frac{1}{N} \sum_{i=1}^N \bm  v^\top \bm Z_i \bm Z_i^\top \bm v = \frac{1}{N} \bm  v^\top \bm Z^\top \bm Z \bm v \leq \lambda_{\max} (\frac{1}{N} \bm Z^\top \bm Z ) = O(1) $. This implies $$\lambda_{\max}( \frac{1}{N} \sum_{i=1}^N \bm U_i \bm U_i^\top )= O(1),$$ hence by \eqref{eq:LNgamma12pf} and \eqref{eq:LN2maximum},  there exists a constant $L > 0$ such that 
		$$|| \nabla L_N(\bg_1) - \nabla L_N(\bg_2) ||_2  \leq L || \bg_1 - \bg_2 ||_2.$$ 
		This completes the proof of Lemma \ref{lm:LNgamma12}. 
	\end{proof}

	%
	%
	
\end{document}